\documentclass[11pt]{book}
\usepackage[intlimits]{amsmath}
\usepackage{amsfonts, amssymb, dsfont, mathrsfs}
\usepackage{graphicx}
\DeclareGraphicsExtensions{.pdf, .jpg, .png}
\graphicspath{{figures/}}
\usepackage{color}
\usepackage{cite}
\usepackage[colorlinks=true,linkcolor=blue,citecolor=blue,anchorcolor=blue]{hyperref}
\newtheorem{theorem}{Theorem}[chapter]
\newtheorem{lemma}[theorem]{Lemma}
\newtheorem{proposition}[theorem]{Proposition}

\newtheorem{corollary}[theorem]{Corollary}

\newenvironment{proof}[1][Proof]{\begin{trivlist}
\item[\hskip \labelsep {\bfseries #1}]}{\end{trivlist}}

\newenvironment{poof}[1][``Proof'']{\begin{trivlist}
\item[\hskip \labelsep {\bfseries #1}]}{\end{trivlist}}

\newenvironment{definition}[1][Definition]{\begin{trivlist}
\item[\hskip \labelsep {\bfseries #1}]}{\end{trivlist}}

\newenvironment{problem}[1][Problem]{\begin{trivlist}
\item[\hskip \labelsep {\bfseries #1}]}{\end{trivlist}}

\newenvironment{solution}[1][Solution]{\begin{trivlist}
\item[\hskip \labelsep {\bfseries #1}]}{\end{trivlist}}

\newenvironment{remark}[1][Remark]{\begin{trivlist}
\item[\hskip \labelsep {\bfseries #1}]}{\end{trivlist}}

\newcommand{\qed}{\hfill \rule{1.5ex}{1.5ex}}

\usepackage{fancyhdr}
\renewcommand{\chaptermark}[1]%
        {\markboth{\textsc{\thechapter.\ #1}}{}}
\renewcommand{\sectionmark}[1]%
       {\markright{{\thesection.\ #1}}}
\renewcommand{\headrule}{\vbox to 0pt{\vspace{-3mm}\hbox to \headwidth{\color{red}\dotfill}\vss}}

\newcommand{\CleanPage}{\clearpage{\pagestyle{empty}\cleardoublepage}}

{\makeatletter  \global\@mparswitchfalse}
\let\omarginpar\marginpar
\def\marginpar#1{\omarginpar{\footnotesize \raggedright \color{green}#1}}

\begin{document}
\frontmatter
%%%%%%%%%%%%%%%%%%%%%%%%%%%%%%%%%%%%%%%%%%%%%%%%%%%%%%%%%%%%%%%%%%%%%%
%%%%       COVER PAGE

\begin{titlepage}

\begin{center}
       \colorbox{red}{\textcolor{white}{
       \vbox{
       \vspace{0.3cm}
       {\Large SPHERICAL HARMONICS IN $p$ DIMENSIONS} 
        \vspace{0.3cm}
            }}}
\end{center}

\vspace{1cm}

\begin{center}
\textbf{\textcolor{red}{CHRISTOPHER R. FRYE} \&  \textcolor{red}{COSTAS EFTHIMIOU}}\\
Department of Physics\\
UNIVERSITY OF CENTRAL FLORIDA
\end{center}

 \vfill
 \rightline{\textcolor{red}{VERSION: 1.0}}
 \rightline{May  10, 2012}

\newpage
\thispagestyle{empty}

\ \vfill
\begin{center}
\copyright All rights reserved.
\end{center}

Please contact Christopher Frye at \texttt{christopher.frye@knights.ucf.edu} with any comments or corrections.

\ \vfill
This book was typeset in \TeX\ using the \LaTeXe\  Document Preparation System.

\end{titlepage}

\tableofcontents
\thispagestyle{empty}

%____________________________________________________________________________________________________________________________
\chapter*{Preface}
\addcontentsline{toc}{chapter}{Preface}

The authors prepared the following booklet in order to make several useful topics from the theory of special functions, in particular
the spherical harmonics and Legendre polynomials of $\mathbb{R}^p$, available to undergraduates studying
physics or mathematics. With this audience in mind, nearly all details of the calculations and proofs are written
out, and extensive background material is covered before beginning the main subject matter. The reader is assumed to have knowledge of
multivariable calculus and linear algebra (especially inner product spaces) as well as some level of comfort with reading proofs.

Literature in this area is scant, and for the undergraduate it is virtually nonexistent. To find
the development of the spherical harmonics that arise in $\mathbb{R}^3$, physics students can look in almost any text on mathematical methods,
electrodynamics, or quantum mechanics (see \cite{Methods:Arfken}, \cite{Methods:ByronFuller}, \cite{ED:Jackson}, \cite{QM:Landau},
\cite{QM:Shankar}, for example), and math students can search any book on boundary value problems, PDEs,
or special functions (see  \cite{BVP:Folland}, \cite{PDE:Strauss}, for example). However, the
undergraduate will have a very difficult time finding accessible material on the corresponding topics in arbitrary $\mathbb{R}^p$.

The authors used Hochstadt's \emph{The Functions of Mathematical Physics} \cite{SpecFun:Hochstadt} as a primary reference (which is, unfortunately,
out of print). If the reader seeks a much more concise treatment of spherical harmonics in an arbitrary number of dimensions written at a higher level,
\cite{SpecFun:Hochstadt} is recommended. Much of the theory developed below can be found there.

\newpage
\thispagestyle{empty}

%____________________________________________________________________________________________________________________________

\chapter*{Acknowledgements}
\addcontentsline{toc}{chapter}{Acknowledgements}

The first author worked on this booklet as an undergraduate and is very grateful
to the second author for his guidance and assistance during this project. The first author
would also like to thank Professors Maxim Zinchenko and Alexander Katsevich for their valuable
comments and suggestions on the presentation of the ideas.
In addition, he thanks his friends Jie Liang, Byron Conley, and Brent Perreault
for their feedback after proofreading portions of this manuscript.

This work has been supported in part by NSF Award DUE 0963146 as well as UCF SMART and RAMP Awards.
The first author is grateful to the Burnett Honors College (BHC), the Office of Undergraduate Research (OUR), and the Research and Mentoring Program (RAMP) at UCF for their generous support during his studies. He would like to thank especially Alvin Wang, Dean of BHC;  Martin  Dupuis, Assistant Dean of BHC;
Kim Schneider, Director of OUR; and Michael Aldarondo-Jeffries, Director of RAMP. Last but not least, he thanks Paul Steidle for providing him a quiet office to work in while writing.

\newpage
\thispagestyle{empty}

\mainmatter

%%%%%%%%%%%%%%%%%%%%%%%%%%%%%%%%%%%%%%%%%%%%%%%%%%%%%%%%%%%%%%%%%%%%%%
\chapter{Introduction and Motivation}
\label{ch:intro}
%%%%%%%%%%%%%%%%%%%%%%%%%%%%%%%%%%%%%%%%%%%%%%%%%%%%%%%%%%%%%%%%%%%%%%

Many important equations in physics involve the \emph{Laplace
operator}, which is given by
\begin{align}
  \Delta_2 &= {\partial^2 \over \partial x^2} + {\partial^2 \over
  \partial y^2}, \label{eq:LapOp2}\\
  \Delta_3 &= {\partial^2 \over \partial x^2} + {\partial^2 \over
  \partial y^2} + {\partial^2 \over \partial z^2},
  \label{eq:LapOp3}
\end{align}
in two and three dimensions\footnote{We may drop the subscript if we want to keep the number of dimensions arbitrary.}, respectively.
We will see later (Proposition \ref{prop:invlap}) that the Laplace operator is invariant under a rotation of the coordinate system.
Thus, it arises in many physical situations in which there exists spherical
symmetry, i.e., where physical quantities depend only on the radial distance $r$ from some center of symmetry $\mathcal{O}$.
For example, the electric potential $V$ in free space is found by solving the \emph{Laplace equation},
\begin{equation}
    \Delta\Phi = 0,
\label{pr:Lap}
\end{equation}
which is rotationally invariant. Also, in quantum mechanics, the
wave function $\psi$ of a particle in a central field can be
found by solving the \emph{time-independent Schr$\ddot{o}$dinger
equation},
\begin{equation}
    \left[ -{\hbar^2 \over 2m} \Delta + V(r) \right] \psi = E \psi
    \label{pr:Sch}
\end{equation}
where $\hbar$ is Planck's constant, $m$ is the mass of the
particle, $V(r)$ is its potential energy, and $E$ is its total
energy.

We will give a brief introduction to these problems in two and
three dimensions to motivate the main subject of this discussion.
In doing this, we will get a preview of some of the properties
of \emph{spherical harmonics} --- which, for now, we can just think of as some
special set of functions --- that we will develop later in the general setting of $\mathbb{R}^p$.

%====================================================================
\section{Separation of Variables} \label{sec:SepVar}
%====================================================================

%----------------------------------------
\subsection*{Two-Dimensional Case}
%---------------------------------------

Since we are interested in problems with spherical symmetry, let
us rewrite the Laplace operator in spherical coordinates, which in
$\mathbb{R}^2$ are just the ordinary polar coordinates\footnote{The polar angle $\phi$ actually requires a more elaborate definition, since $\tan^{-1}$ only produces angles in the first and fourth quadrants. However, this detail will not concern us here.},
\begin{equation}
    r = \sqrt{x^2 + y^2}, \quad \phi = \tan^{-1} \left(\frac{y}{x}\right).
    \label{eq:polar}
\end{equation}
Alternatively,
\[
    x = r \cos{\phi}, \quad y = r \sin{\phi}.
\]
Using the chain rule, we can rewrite the Laplace operator as
\begin{equation}
    \Delta_2 = {\partial^2 \over \partial r^2} + {1 \over
    r}{\partial \over \partial r} + {1 \over r^2}{\partial^2 \over
    \partial \phi^2}.
    \label{eq:SphLap2}
\end{equation}
In checking this result, perhaps it is easiest to begin with \eqref{eq:SphLap2} and recover \eqref{eq:LapOp2}.
First we compute
\begin{equation}
    \label{eq:chain}
    {\partial \over \partial r} = {\partial x \over \partial r}{\partial \over \partial x} + {\partial y \over \partial r}{\partial \over \partial y}
    = \cos\phi {\partial \over \partial x} + \sin\phi {\partial \over \partial y},
\end{equation}
which implies that
\[
    {\partial^2 \over \partial r^2} = \cos^2\phi {\partial^2 \over \partial x^2} +2\sin\phi\cos\phi{\partial^2 \over \partial x \partial y} +
    \sin^2\phi {\partial^2 \over \partial y^2},
\]
and
\[
    {\partial \over \partial \phi} = {\partial x \over \partial \phi}{\partial \over \partial x} + {\partial y \over \partial \phi}{\partial \over \partial y}
    = -r \sin\phi {\partial \over \partial x} + r \cos\phi {\partial \over \partial y},
\]
which gives
\[
    {\partial^2 \over \partial \phi^2} = r^2 \sin^2\phi {\partial^2 \over \partial x^2} -2r^2\sin\phi\cos\phi{\partial^2 \over \partial x \partial y}
    +r^2\cos^2\phi{\partial^2 \over \partial y^2}.
\]
Inserting these into \eqref{eq:SphLap2} gives us back \eqref{eq:LapOp2}.

Thus, \eqref{pr:Lap} becomes
\[
    {\partial^2 \Phi \over \partial r^2} + {1 \over
    r}{\partial \Phi \over \partial r} + {1 \over r^2}{\partial^2 \Phi \over
    \partial \phi^2} = 0.
\]
To solve this equation it is standard to assume that
$\Phi(r,\phi) = \chi(r)Y(\phi)$,
where $\chi(r)$ is a function of $r$ alone and $Y(\phi)$ is a function of $\phi$ alone. Then
\[
    Y \, {d^2 \chi \over d r^2} + {Y \over
    r} \, {d \chi \over d r} + {
    \chi \over r^2} \, {d^2 Y \over
    d \phi^2} = 0.
\]
Multiplying by $r^2 / \chi Y$ and rearranging,
\[
    {r^2 \over \chi} \, {d^2 \chi \over d r^2} + {r \over
    \chi} \, {d \chi \over d r} = {-1 \over Y} \, {d^2 Y \over
    d \phi^2}.
\]
We see that a function of $r$ alone (the left side) is equal to a
function of $\phi$ alone (the right side). Since we can vary $r$
without changing $\phi$, i.e., without changing the right
side of the above equation, it must be that the left side of
the above equation does not vary with $r$ either. This means the left
side of the above equation is not really a function of $r$ but a
constant. As a consequence, the right side is the same
constant. Thus, for some $-\lambda$ we can write
\begin{equation}
    {-1 \over Y} \, {d^2 Y \over
    d \phi^2} = - \lambda = {r^2 \over \chi} \, {d^2 \chi \over d r^2} + {r \over
    \chi} \, {d \chi \over d r}.
    \label{eq:SepVar2}
\end{equation}
We solve $Y'' = \lambda Y$ to get the linearly independent
solutions
\[
    Y(\phi) = \left\{
        \begin{array}{lr}
            e^{\sqrt{\lambda} \phi}, \: e^{- \sqrt{\lambda} \phi} & \text{if } \lambda >
            0,\\
            1, \: \phi & \text{if } \lambda = 0,\\
            \sin{\left(\sqrt{|\lambda|} \phi\right)}, \: \cos{\left(\sqrt{|\lambda|}
            \phi\right)} & \text{if } \lambda < 0,
        \end{array}
    \right.
\]
but we must reject some of these solutions. Since $(r_0, \phi_0)$ represents the same point as $(r_0, \phi_0 + 2\pi k)$ for any $k
\in \mathbb{Z}$, we require $Y(\phi)$ to have period $2\pi$. Thus, we can only accept the linearly independent periodic solutions $1$,
$\sin{\left( \sqrt{|\lambda|} \phi \right)}$, and $\cos{\left( \sqrt{|\lambda|} \phi \right)}$, where $\sqrt{|\lambda|}$ must be an
integer. Then, let us replace $\lambda$ with $-m^2$ and write our linearly independent solutions to $Y'' = -m^2 Y$ as
\begin{equation}
    Y_{1,n} = \cos{(n\phi)}, \qquad Y_{2,m} = \sin{(m\phi)},
    \label{eq:SphHar2}
\end{equation}
where\footnote{We use the notation $\mathbb{N}=\{1,2,\dots\}$ and $\mathbb{N}_0=\{0,1,2,\dots\}$.}
$ n\in \mathbb{N}_0$ and $m\in \mathbb{N}$.
Notice from \eqref{eq:SphLap2} that $\partial^2 / \partial \phi^2$
is the angular part of the Laplace operator in two dimensions and
that the solutions given in \eqref{eq:SphHar2} are eigenfunctions
of the $\partial^2 / \partial \phi^2$ operator. We will see in Chapter \ref{ch:sph}
that the functions in \eqref{eq:SphHar2} are actually spherical harmonics;
however, since we have not yet given a definition of a spherical harmonic,
for now we will just refer to these as functions $Y$. The reader should keep in mind
that characteristics of the $Y$'s we comment on here will generalize when we move to $\mathbb{R}^p$.

We can also solve easily for the functions $\chi(r)$ that satisfy
\eqref{eq:SepVar2}, but this does not concern us here. Let us
instead notice a few properties of the functions $Y$. Let us consider these to be
functions $r^m \sin{\left( m\phi \right)}$, $r^n \cos{\left( n\phi
\right)}$ on $\mathbb{R}^2$ that have been restricted to the unit
circle, where $r=1$, and let us analyze the extended functions on
$\mathbb{R}^2$. We will first rewrite them using \emph{Euler's
formula}, $e^{i\phi} = \cos{\phi}+i\sin{\phi}$, which implies
\begin{align*}
  (x+iy)^n &= \left(re^{i\phi}\right)^n \, \, = r^ne^{in\phi} \; = r^n \left[
  \cos{(n\phi)} + i\sin{(n\phi)} \right],\\
  (x-iy)^n &= \left(re^{-i\phi}\right)^n = r^ne^{-in\phi} = r^n \left[
  \cos{(n\phi)} - i\sin{(n\phi)} \right],
\end{align*}
so that
\begin{align*}
  r^n\cos{(n\phi)} &= {1 \over 2} \left[ (x+iy)^n+(x-iy)^n \right] \, \overset{\text{def}}{=} H_{1,n}(x,y),\\
  r^n\sin{(n\phi)} &= {1 \over 2i} \left[ (x+iy)^n-(x-iy)^n \right] \overset{\text{def}}{=} H_{2,n}(x,y).
\end{align*}
We notice that the $Y$'s can be written as
polynomials restricted to the unit circle, where $r=1$. Furthermore, observe that
\[
    H_{1,n}(tx,ty) = t^n H_{1,n}(x,y), \quad \text{and} \quad
    H_{2,n}(tx,ty) = t^n H_{2,n}(x,y);
\]
we call polynomials with this property \emph{homogeneous of
degree $n$}. Moreover, the reader can also check that these
polynomials satisfy the Laplace equation \eqref{pr:Lap}, by either
using \eqref{eq:LapOp2} or the Laplace operator in polar
coordinates \eqref{eq:SphLap2} for the computation, i.e.,
\[
    \Delta_2 \, H_{1,n} = 0, \quad \text{and} \quad \Delta_2 \, H_{2,n} = 0.
\]
Let us also notice that the $Y$'s of different degree
are \emph{orthogonal over the unit circle}, which means
\[ \begin{array}{rlc}
  \vspace{.25cm} \int_0^{2\pi} \sin{(n\phi)} \sin{(m\phi)} \,d\phi\!\! &= 0, &
  \text{ if } n \neq m, \\
  \vspace{.25cm} \int_0^{2\pi} \cos{(n\phi)} \cos{(m\phi)} \,d\phi\!\! &= 0, &
  \text{ if } n \neq m, \\
  \int_0^{2\pi} \sin{(n\phi)} \cos{(m\phi)} \,d\phi\!\! &= 0, &
  \text{ if } n \neq m,
\end{array} \]
as we can easily compute by taking advantage of Euler's
formula. For instance, we can calculate
\[
    \int_0^{2\pi} \sin(n\phi) \sin(m\phi) \, d\phi = \int_0^{2\pi} {e^{in\phi}-e^{-in\phi} \over 2}\cdot{e^{im\phi}-e^{-im\phi} \over 2} \, d\phi
\]
which becomes
\[
    \int_0^{2\pi} \left({e^{i(m+n)\phi}-e^{-i(m+n)\phi} \over 2}-{e^{i(m-n)\phi}+e^{-i(m-n)\phi} \over 2}\right) \, d\phi
\]
or
\[
    \int_0^{2\pi} \Big( \sin[(m+n)\phi] - \cos[(m-n)\phi] \Big) \, d\phi = 0,
\]
for $n \neq m$. The reader can check the rest in a similar fashion.

Finally, by recalling the theorems of Fourier analysis\footnote{Doing so will not be
necessary to understand the material we present here, but the reader unfamiliar with
Fourier analysis may choose to consult \cite{BVP:Folland}.}, we know
that any ``reasonable'' function defined on the unit circle can be
expanded in a \emph{Fourier series}. That is, given a function $f:
[0,2\pi) \rightarrow \mathbb{R}$ satisfying certain conditions
(that do not concern us in this introduction), we can write
\[
    f(\phi) = \sum_{m=1}^\infty a_m \sin{(m\phi)} +
    \sum_{n=0}^\infty b_n \cos{(n\phi)}
\]
for some constants $a_m, b_n$. We say that the $Y$'s make up a \emph{complete set} of
functions over the unit circle, since we can expand any nice function $f(\phi)$ defined on $[0,2\pi)$ in terms of
them.

Before we move on, we will show that \eqref{pr:Sch} can be
approached using a method almost identical to the method we used
above. Let us assume that the solution $\psi$ to \eqref{pr:Sch}
can be written as $\psi(r, \phi) = \chi(r) Y(\phi)$. Then, using
polar coordinates, the equation becomes
\[
    Y \, \left( -{\hbar^2 \over 2m} {d^2  \over d r^2} - {\hbar^2 \over 2mr} \, {d \over d r} + V(r) \right)\chi
    + {\chi \over r^2} \, \left( -{\hbar^2 \over 2m} {d^2  \over d \phi^2}\right)Y = E \chi Y.
\]
Multiplying by $r^2 / \chi Y$ and rearranging,
\[
    {1 \over \chi} \left( -{r^2\hbar^2 \over 2m} {d^2 \over dr^2} -{r\hbar^2 \over 2m} {d \over dr} + r^2V(r)
    \right)\chi - Er^2 = {1 \over Y} \left({\hbar^2 \over 2m} {d^2 \over
    d\phi^2}\right)Y.
\]
Once again, we see that a function of $r$ alone is equal to a
function of $\phi$ alone, and we conclude that both sides of the
above equation must be equal to the same constant. Using the same
reasoning as before, we write this constant as $-\ell^2\hbar^2 /
2m$, where $\ell$ is an integer. If we carry out the calculation, we will see that the functions
$Y(\phi)$ are the same ones we found previously in this subsection. We will also find the radial equation
\[
    {1 \over \chi} \left( -{r^2\hbar^2 \over 2m} {d^2 \over dr^2} -{r\hbar^2 \over 2m}{d \over dr} + r^2V(r)
    \right)\chi - Er^2 = -{\ell^2\hbar^2 \over 2m},
\]
which we can rewrite as
\[
    \left\{ -{\hbar^2 \over 2m} \left( {d^2 \over dr^2}-{1 \over r}{d \over dr} \right) + \left[ V(r)+ {\ell^2\hbar^2 \over 2mr^2} \right]
    \right\}\chi = E \chi.
\]
It is interesting to note that this equation resembles
\eqref{pr:Sch}. If we think of $r$ as our single independent
variable and $\chi$ as our wave function, we have an effective
potential energy
\[
    V_{\text{eff}}(r) = V(r) + {\ell^2\hbar^2 \over 2mr^2},
\]
where we call the second term the \emph{centrifugal} term. Thus $\chi$ represents a fictitious particle that feels an effective forece
\[
    \vec{F}_\text{eff} = - \nabla V_\text{eff} = - {dV_\text{eff}
    \over dr} \, \hat{r}.
\]
We see that the centrifugal term contributes a force
\[
    \vec{F}_\text{centrifugal} = {\ell^2 \hbar^2 \over mr^3}
    \, \hat{r},
\]
pushing the particle away from the center of symmetry.

%----------------------------------------
\subsection*{Three-Dimensional Case}
%----------------------------------------

Here, we will follow a procedure almost identical to that of the last subsection, but we will not take the discussion as far in
three dimensions. The $Y$'s we will find in three dimensions are more widely used than those in any other number of dimensions;
however, a thorough development of the functions in $\mathbb{R}^3$ could take many pages, and this would
distract us from our goal of moving to $p$ dimensions. Furthermore, the results that we would find in three
dimensions are only special cases of more general theorems we will develop later in the discussion. If the reader is interested in
studying the usual spherical harmonics of $\mathbb{R}^3$ in depth, there are a multitude of sources we can recommend, including
\cite{Methods:Arfken}, \cite{Methods:ByronFuller}, \cite{BVP:Folland} and \cite{PDE:Strauss}.

Let us rewrite the Laplace operator in spherical coordinates,
which are given by
\begin{equation}
    r = \sqrt{x^2 + y^2 + z^2}, \quad \theta = \tan^{-1}\left(
    {\sqrt{x^2+y^2} \over z} \right), \quad \phi = \tan^{-1}\left( {y
    \over x} \right).
    \label{eq:SphCoor}
\end{equation}
Alternatively,
\[
    x = r \sin{\theta} \cos{\phi}, \quad y = r \sin{\theta}
    \sin{\phi}, \quad z = r \cos{\theta}.
\]
Using the chain rule, we find
\begin{equation}
    \Delta_3 = {1 \over r^2}{\partial \over \partial r} \left( r^2
    {\partial \over \partial r} \right) + {1 \over r^2} \left[ {1
    \over \sin{\theta}} {\partial \over \partial \theta} \left(
    \sin{\theta} {\partial \over \partial \theta} \right) + {1 \over
    \sin^2{\theta}} {\partial^2 \over \partial \phi^2} \right].
    \label{eq:SphLap3}
\end{equation}
As in the two-dimensional case, it is probably easiest to verify this formula
by starting with \eqref{eq:SphLap3} and producing \eqref{eq:LapOp3}.
The reader should check this result for practice, proceeding
exactly as we did beginning in \eqref{eq:chain}.

Inserting \eqref{eq:SphLap3} into \eqref{pr:Lap} gives
\[
    {1 \over r^2}{\partial  \over \partial r} \left( r^2
    {\partial \Phi \over \partial r} \right) + {1 \over r^2} \left[ {1
    \over \sin{\theta}} {\partial \over \partial \theta} \left(
    \sin{\theta} {\partial \Phi \over \partial \theta} \right) + {1 \over
    \sin^2{\theta}} {\partial^2 \Phi \over \partial \phi^2} \right] = 0.
\]
Searching for solutions $\Phi$ in the form $\Phi(r,\theta,\phi) = \chi(r)Y(\theta, \phi)$, where $\chi(r)$ is a function of $r$
alone and $Y(\theta,\phi)$ is a function of only $\theta$ and $\phi$, this becomes
\[
    Y \, {1 \over r^2}{d \over d r} \left( r^2
    {d \over d r} \right)\chi + {\chi \over r^2} \left[ {1
    \over \sin{\theta}} {\partial \over \partial \theta} \left(
    \sin{\theta} {\partial  \over \partial \theta} \right) + {1 \over
    \sin^2{\theta}} {\partial^2  \over \partial \phi^2} \right]Y = 0.
\]
Multiplying by $r^2 / \chi Y$ and rearranging,
\[
    {1 \over \chi} \, {d \over d r} \left( r^2
    {d \over d r} \right)\chi =  {-1 \over Y} \, \left[ {1
    \over \sin{\theta}} {\partial \over \partial \theta} \left(
    \sin{\theta} {\partial  \over \partial \theta} \right) + {1 \over
    \sin^2{\theta}} {\partial^2  \over \partial \phi^2} \right]Y.
\]
Since we have found that a function of $r$ alone is equal to a
function of only $\theta$ and $\phi$, we use the same reasoning as
in the previous subsection to conclude that both sides of the above
equation must be equal to the same constant, call it $-\lambda$.
This implies that
\begin{equation}
    \left[ {1
    \over \sin{\theta}} {\partial \over \partial \theta} \left(
    \sin{\theta} {\partial \over \partial \theta} \right) + {1 \over
    \sin^2{\theta}} {\partial^2 \over \partial \phi^2} \right]Y =
    \lambda Y.
    \label{eq:Yeig}
\end{equation}
At this point, we will stop. It turns out that the $Y$'s which satisfy this equation
are actually spherical harmonics. Comparing \eqref{eq:SphLap3} and \eqref{eq:Yeig},
we see that the $Y$'s are eigenfunctions of the angular part of the Laplace
operator, just as in two dimensions. If we studied the functions $Y$ further
we would also find that, analogously to what we noticed in $\mathbb{R}^2$,
the $Y$'s form a complete set over the unit sphere, each $Y$ is a homogeneous polynomial restricted to the unit sphere,
and these polynomials satisfy the Laplace equation. However, to
develop these results and the many others that exist would require
us to study Legendre's equation, Legendre polynomials, and
associated Legendre functions, and we will choose to leave such an
in-depth analysis to the general case of $p$ dimensions.

We will now move on to see how these functions $Y$ are related to
angular momentum in quantum mechanics.

%====================================================================
\section{Quantum Mechanical Angular Momentum}
%====================================================================

We have seen that spherical harmonics in two and three dimensions relate to the one-dimensional sphere (circle) and the
two-dimensional sphere (surface of a regular ball) respectively. In quantum mechanics, rotations of a system are generated by the angular momentum
operator. Spherical symmetry means invariance under all such rotations. Therefore, a relation between the theory of spherical harmonics and the
theory of angular momentum is not only expected but is a natural and fundamental result.

Recall that in classical mechanics, the angular momentum of a
particle is defined by the cross product
\[
    \vec{L} = \vec{r} \times \vec{p},
\]
where $\vec{p}$ is its linear momentum. To find the quantum
mechanical angular momentum operator, we make the substitution
$p_i \mapsto -i \hbar \, \partial / \partial x_i$ where $x_1 =
x$, $x_2 = y$, and $x_3 = z$. Thus, we see that\footnote{The hat above the angular momentum indicates that it is an operator (not a unit vector) in this section.}
\[
    \hat{\vec{L}} = -i\hbar \, \vec{r} \times \nabla,
\]
where
\[
    \nabla = \left( {\partial \over \partial x}, {\partial \over \partial y}, {\partial \over \partial
    z}
    \right).
\]

%-------------------------------------------------
\subsection*{Two-Dimensional Case}
%-------------------------------------------------

In the plane, the angular momentum operator has only one component, given by
\[
    \hat{L} = -i\hbar \left( x {\partial \over \partial y} - y {\partial \over \partial x}
    \right).
\]
Using the polar coordinates defined in \eqref{eq:polar} and the
chain rule, we can rewrite this as
\[
    \hat{L} = -i\hbar {\partial \over \partial \phi},
\]
as the reader should have no problem checking using the same strategy
s/he used to verify \eqref{eq:SphLap2} and \eqref{eq:SphLap3}.
Then
\[
    \hat{L}^2 = -\hbar^2 {\partial^2 \over \partial \phi^2},
\]
and we check using \eqref{eq:SphHar2} that the functions
$Y(\phi)$ are eigenfunctions of the $\hat{L}^2$ operator. In
particular,
\[
    \hat{L}^2 Y_{m,j}(\phi) = -\hbar^2 {\partial^2 \over \partial
    \phi^2} Y_{m,j} = \hbar^2m^2 Y_{m,j},
\]
and we see that the function $Y_{m,j}$ is associated
with the eigenvalue $\hbar^2m^2$.

In quantum mechanics, operators such as $\hat{L}$ represent
dynamical variables. If an operator $\hat{O}$ has eigenfunctions
$\psi_k$ with corresponding eigenvalues $\lambda_k$, then a
particle in state $\psi_k$ will be observed to have a value of
$\lambda_k$ for the dynamical variable $\hat{O}$. Therefore, we
see that a particle in the state $Y_{m,j}$ will be observed to
have a value of $\hbar^2m^2$ for its angular momentum squared. We
say the function $Y_{m,j}$ carries angular momentum
$\hbar m$.

%--------------------------------------------------
\subsection*{Three-Dimensional Case} \label{subs:3dang}
%--------------------------------------------------

Things work similarly in three dimensions, where the angular
momentum operator has the components
\begin{align*}
  \hat{L}_x &= -i\hbar \left( y {\partial \over \partial z} - z {\partial \over \partial y}
  \right),\\
  \hat{L}_y &= -i\hbar \left( z {\partial \over \partial x} - x {\partial \over \partial z}
  \right),\\
  \hat{L}_z &= -i\hbar \left( x {\partial \over \partial y} - y {\partial \over \partial x}
  \right).
\end{align*}

Using the spherical coordinates defined in \eqref{eq:SphCoor} and
the chain rule, we can rewrite these as
\begin{align*}
  \hat{L}_x &= i\hbar \left( \sin{\theta} {\partial \over \partial \theta} + \cos{\phi}\cot{\theta} {\partial \over \partial \phi}
  \right),\\
  \hat{L}_y &= -i\hbar \left( \cos{\theta} {\partial \over \partial \theta} + \sin{\phi}\cot{\theta} {\partial \over \partial \phi}
  \right),\\
  \hat{L}_z &= -i\hbar \, {\partial \over \partial \phi} ;
\end{align*}
the reader should check these formulas. These equations allow us to compute $\hat{\vec{L}}^2 =
\hat{\vec{L}} \cdot \hat{\vec{L}} = \hat{L}_x^2+\hat{L}_y^2+\hat{L}_z^2$.
Carrying out the multiplication,
\begin{equation} \label{eq:L2}
    \hat{\vec{L}}^2 = -\hbar^2 \left[ {1
    \over \sin{\theta}} {\partial \over \partial \theta} \left(
    \sin{\theta} {\partial \over \partial \theta} \right) + {1 \over
    \sin^2{\theta}} {\partial^2 \over \partial \phi^2} \right],
\end{equation}
and we see that by \eqref{eq:Yeig}, the functions $Y$
are eigenfunctions of the $\hat{\vec{L}}^2$ operator.

We claimed in the last section that the functions $Y(\theta,
\phi)$ were homogeneous polynomials with restricted domain, so let
us write $Y_\ell(\theta, \phi)$ where $\ell$ denotes the degree of
homogeneity. In Section \ref{sec:sphharm}, we will see an easy way
to compute the eigenvalue of $\hat{\vec{L}}^2$ associated with
$Y_\ell$, and it will turn out to be $\hbar^2\ell(\ell+1)$. So we
claim that in three dimensions, the function
$Y_\ell(\theta, \phi)$ carries an angular momentum of $\hbar
\sqrt{\ell(\ell+1)}$.

In Chapter \ref{ch:sph}, we will give rigorous
foundations to the seemingly coincidental facts we have discovered in this chapter about
the functions $Y$ that arose as solutions to certain differential equations.
But first, we will devote a chapter to gaining some practice and intuition working in $\mathbb{R}^p$.

%%%%%%%%%%%%%%%%%%%%%%%%%%%%%%%%%%%%%%%%%%%%%%%%%%%%%%%%%%%%%%%%%%%%%%
\chapter{Working in $p$ Dimensions} \label{ch:rp}
%%%%%%%%%%%%%%%%%%%%%%%%%%%%%%%%%%%%%%%%%%%%%%%%%%%%%%%%%%%%%%%%%%%%%%

In this chapter, we spend some time developing our skills in performing calculations
in $\mathbb{R}^p$ and exercising our abilities in visualizing a $p$-dimensional space for an arbitrary natural number $p$. We will
use the majority of the results we obtain here in the development of our main subject,
but some topics we discuss just out of pure interest or to improve our intuition.

First, let us generalize the definition of the Laplace operator to $\mathbb{R}^p$, where a point\footnote{We will not place vector arrows above points $x$ in $\mathbb R^p$.}
$x$ is given by the ordered pair $(x_1,x_2,\ldots,x_p)$.

\begin{definition}
    The \emph{Laplace operator} in $\mathbb{R}^p$ is given by
    \begin{equation}
        \Delta _p \overset{\text{def}}{=} \sum _{i=1}^p {\partial ^2 \over \partial  x_i^2}
    \label{eq:pLaplacian}
    \end{equation}
    The \emph{del operator} in $\mathbb{R}^p$ is the vector operator
    \[
        \nabla_p \overset{\text{def}}{=} \left( {\partial \over \partial x_1}, {\partial \over \partial x_2}, \ldots, {\partial \over \partial x_p}
        \right).
    \]
\end{definition}

%====================================================================
\section{Rotations in $\mathbb{R}^p$} \label{sec:rota}
%====================================================================

Let us quickly consider orthogonal rotations of the coordinate axes in $\mathbb{R}^p$.
Such rotations leave the lengths of vectors unchanged. Indeed, the length of a vector is a geometric quantity; rotating the coordinate system
we use to describe the vector leaves its length invariant. In fact, in a more abstract setting, we could define a rotation to be any
transformation of coordinates that leaves the lengths of vectors unchanged.

In what follows, we let $x$ denote a
column vector\footnote{The superscript $t$ denotes the operation of matrix transposition.}
$(x_1,x_2,\ldots,x_p)^t$ in $\mathbb{R}^p$ and use $\langle \cdot , \cdot \rangle$ to represent
the dot product of two vectors. The fact that a rotation matrix $R$ leaves the length of $x$ invariant
means $\langle Rx,Rx \rangle = \langle x,x \rangle$. Moreover,
since the dot product between any two vectors $x,y$ can be written as
\[
    \langle x,y \rangle = {1\over 2} \big( \langle x+y, x+y \rangle - \langle x,x \rangle - \langle y,y \rangle \big),
\]
it follows that coordinate rotations leave all dot products invariant.

Notice further that we can write dot products such as $\langle x,y \rangle$ as matrix products $y^t x$. In this notation, the requirement that
$\langle Rx,Ry \rangle = \langle x,y \rangle$ translates into the necessity of $(Ry)^t (Rx) = y^t x$, or $y^tR^tRx = y^tx$. Since this equation
must hold for all $x,y \in \mathbb{R}^p$, we can conclude that any rotation matrix $R$ must satisfy $R^tR=I$, where $I$ is the identity matrix.

Now we can verify our claim in the first few sentences of this booklet that the Laplace operator $\Delta_p$ remains unchanged after being subjected
to a rotation of coordinates.

\begin{proposition} \label{prop:invlap}
  The Laplace operator $\Delta_p$ is invariant under coordinate
  rotations. That is, if $R$ is a rotation matrix and $x' = Rx$,
  then $\Delta'_p = \Delta_p$, i.e. \[\sum_{j=1}^p \left({\partial \over \partial x'_j} \right)^2 =
  \sum_{j=1}^p \left({\partial \over \partial x_j} \right)^2.\]
\end{proposition}

\begin{proof}
    This  can be proved very easily by noticing that $\Delta_p = \nabla_p \cdot \nabla_p$ is a dot product of vector operators. Since all dot
    products are unchanged by coordinate rotations, we can conclude that $\Delta_p$ is not affected by any rotation $R$.

    In case the reader is not satisfied with this quick justification, let us compute $\Delta_p'$, the Laplace operator after application of a rotation
    of coordinates $R$.  Since $R$ is a rotation matrix, it is orthogonal, i.e. $RR^t = I$. Then, using the chain rule,
  \begin{align*}
    \Delta_p &= \sum_{j=1}^p \left({\partial \over \partial x_j} \right)^2 =
    \sum_{j=1}^p \left[
    \left( \sum_{k=1}^p {\partial x'_k \over \partial x_j}{\partial \over \partial x'_k} \right)
    \left( \sum_{\ell=1}^p {\partial x'_\ell \over \partial x_j}{\partial \over \partial x'_\ell} \right) \right] \\
    &= \sum_{j=1}^p \left[
    \left( \sum_{k=1}^p R_{kj}{\partial \over \partial x'_k} \right)
    \left( \sum_{\ell=1}^p R_{\ell j}{\partial \over \partial x'_\ell} \right) \right],
  \end{align*}
  so
  \begin{align*}
    \Delta_p &= \sum_{k,\ell = 1}^p {\partial \over \partial
    x'_k}{\partial \over \partial x'_\ell} \left( {\sum_{j=1}^p R_{kj}R^t_{j\ell}} \right)
    = \sum_{k=1}^p \left( {\partial \over \partial x'_k} \right)^2 = \Delta'_p.\
  \end{align*}
  This proves the proposition. \qed
\end{proof}

%====================================================================
\section{Spherical Coordinates in $p$ Dimensions} \label{sec:pdim}
%====================================================================

Now, in order to develop some experience and intuition working in
higher-dimensional spaces, we will develop the spherical  coordinate system for $\mathbb{R}^p$ in considerable detail. In
particular, we will use an inductive technique to come up with the expression of the
spherical coordinates in $p$ dimensions in terms of the corresponding Cartesian coordinates.

We will let our space have axes denoted $x_1, x_2, \ldots$ First,
in two dimensions, spherical coordinates are just the polar
coordinates given in \eqref{eq:polar},
\[ \begin{array}{ccccc}
  \vspace{0.25cm} r &=& \sqrt{x_1^2 + x_2^2} &\in& [0,\infty),\\
  \phi &=& \tan^{-1}{(x_2 / x_1)} &\in& [0, 2\pi),
\end{array} \]
where $r$ is the distance from the origin and $\phi$ is the azimuthal angle\footnote{We keep the definition of $\phi$ sloppy throughout this section. A more precise formula would use a two-argument $\tan^{-1}$ function that produces angles on the entire unit circle.} in the plane that measures the rotation around the origin.
The  inverse transformation is
$$
      x_1 ~=~ r \, \cos\phi~, \quad
      x_2 ~=~ r \, \sin\phi~.
$$

When we move to three dimensions we add an axis, naming it $x_3$, perpendicular to
the plane. Now, the polar coordinates above can only define a location in the plane; thus, they only tell us on which vertical
line (i.e., line parallel to the $x_3$-axis) we lie, as we can see in Figure \ref{fig:pdim} with $p=3$. To pinpoint our location on this line, we introduce a new angle
$\theta_1$. When we also redefine $r$ to be the three-dimensional distance from the
origin, we have the spherical coordinates given in \eqref{eq:SphCoor},
\[
 \begin{array}{ccccc}
     \vspace{0.25cm} r &=& \sqrt{x_1^2 + x_2^2+x_3^2} &\in& [0,\infty),\\
      \vspace{0.25cm} \phi &=& \tan^{-1}{(x_2 / x_1)} &\in& [0, 2\pi),\\
      \theta_1 &=& \tan^{-1}{ \left( \sqrt{x_1^2+x_2^2} \big / x_3
     \right)}  &\in& [0,\pi].
\end{array}
\]

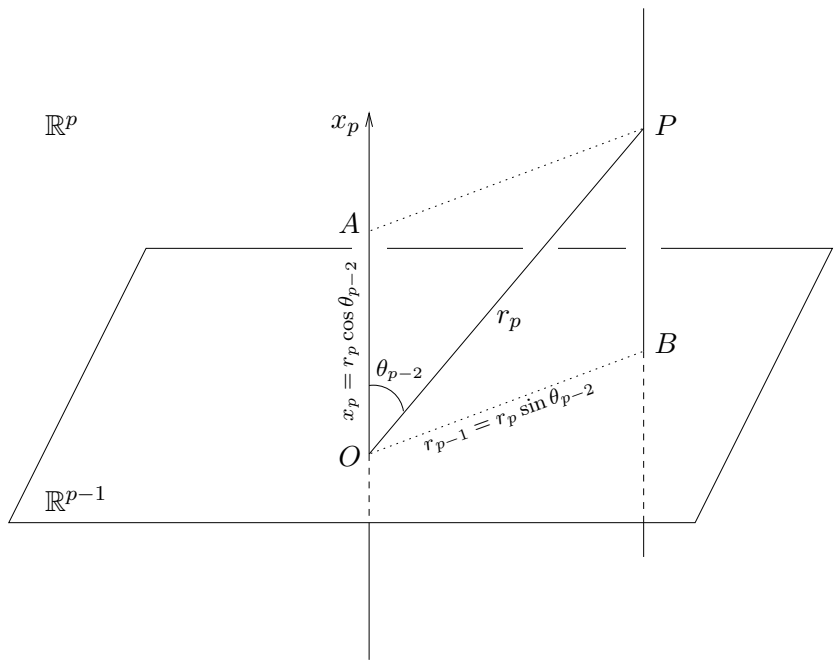
\begin{figure}[h!]
    \begin{center}
    \setlength{\unitlength}{1mm}
    \begin{picture}(110,80)
        \put(0,0){\includegraphics[width=11cm]{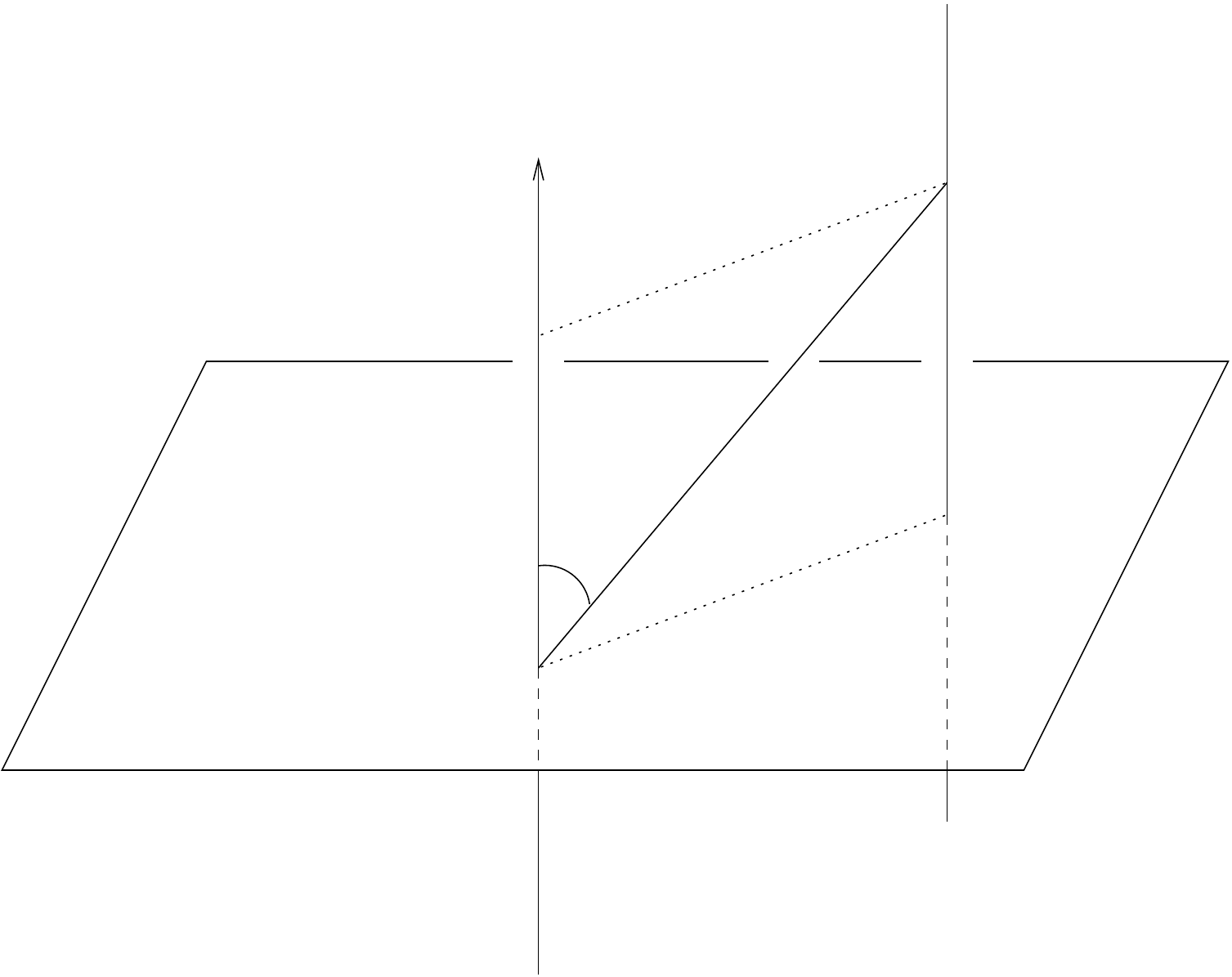}}
        \put(44,26){$O$} \put(44,57){$A$} \put(86,70){$P$} \put(86,41){$B$} \put(43,71){$x_p$}
        \put(5,70){$\mathbb{R}^p$}  \put(5,20){$\mathbb{R}^{p-1}$}
        \put(49,38){\footnotesize $\theta_{p-2}$}
        \put(65,45){$r_p$}
        \put(55,28){\rotatebox{20}{\scriptsize $r_{p-1}=r_p\sin\theta_{p-2}$}}
        \put(44,32){\rotatebox{90}{\scriptsize $x_p=r_p\cos\theta_{p-2}$}}
       \end{picture}
    \end{center}
    \caption{\footnotesize In going from $\mathbb{R}^{p-1}$ to $\mathbb{R}^p$, we visualize $\mathbb{R}^{p-1}$ as a plane and add  a new perpendicular direction.   We introduce a new angular coordinate $\theta_{p-2}$ to determine  the location in the new direction.}
    \label{fig:pdim}
\end{figure}

Now let us imagine moving to four dimensions by adding an axis --- name it the $x_4$-axis ---
perpendicular to the three-dimensional space just discussed. The
3D spherical coordinates given above can only define a location in
$\mathbb{R}^3$, so they only tell us on which ``vertical'' line
(i.e., line parallel to the $x_4$-axis) we lie, as in Figure
\ref{fig:pdim} with $p=4$. We thus introduce a new angle $\theta_2$ to determine the location on this line. We redefine $r$ to be the four-dimensional distance from the origin, and this completes the construction of the 4D spherical coordinates,
\[ \begin{array}{ccccc}
  \vspace{0.25cm} r &=& \sqrt{x_1^2 + x_2^2+x_3^2+x_4^2} &\in& [0,\infty),\\
  \vspace{0.25cm} \phi &=& \tan^{-1}{(x_2 / x_1)} &\in& [0, 2\pi),\\
  \vspace{0.25cm} \theta_1 &=& \tan^{-1}{ \left( \sqrt{x_1^2+x_2^2} \big / x_3
  \right)}  &\in& [0,\pi],\\
  \theta_2 &=& \tan^{-1}{ \left( \sqrt{x_1^2+x_2^2+x_3^2} \big /
  x_4 \right)}  &\in& [0,\pi],
\end{array} \]
where it should be clear from the figure that the new angle
$\theta_2$ only ranges from $0$ to $\pi$.

In going from three to four dimensions, we mimicked the way in
which we transitioned from two to three dimensions. We follow the same
procedure each time we move up in dimension. For arbitrary $p$, this yields
\[ \begin{array}{ccl}
  \vspace{0.25cm} r &=& \sqrt{x_1^2 + x_2^2+\cdots+x_p^2} ,\\
  \vspace{0.25cm} \phi &=& \tan^{-1}{(x_2 / x_1)} ,\\
  \vspace{0.25cm} \theta_1 &=& \tan^{-1}{ \left( \sqrt{x_1^2+x_2^2} \big / x_3
  \right)}  ,\\
  &\vdots&\\
  \theta_{p-2} &=& \tan^{-1}{ \left( \sqrt{x_1^2+x_2^2+\cdots+x_{p-1}^2} \Big /
  x_p \right)} ,
\end{array} \]
where the ranges on the coordinates are as expected from the
previous cases.

We have thus defined the spherical coordinates in $p$-dimensions in terms of the corresponding Cartesian coordinates. To write down the inverse relations we use projections, with Figure \ref{fig:pdim} as an aid. As before, we will derive these relations in detail for a few instructive cases before writing down the most general expressions.

We have already written down $x_1,x_2$ in terms of $r,\phi$ so let's use $\mathbb R^3$ as our first example. We imagine  $\mathbb{R}^3$  as the direct sum of a two-dimensional plane $\mathbb{R}^2$ with the real line $\mathbb{R}$. Then, given the point $P$ in $\mathbb{R}^3$ with spherical coordinates $(r_3, \phi, \theta_1)$, the vector $\overrightarrow{OP}$ can be written as a sum $\overrightarrow{OA}+\overrightarrow{OB}$, where $\overrightarrow{OA}$ lies along the $x_3$-axis and has magnitude
$r_3\,\cos\theta_1$ and $\overrightarrow{OB}$ lies in the plane and has magnitude $r_2=r_3\,\sin\theta_1$. The point $B$ thus has spherical coordinates $(r_2, \phi)$ in the plane, implying
\[
    x_1=r_2 \, \cos\phi \quad \text{and} \quad x_2=r_2\,\sin\phi,
\]
so
\begin{eqnarray}
      x_1  &=&  r_3 \, \sin\theta_1 \, \cos\phi ~, \label{eq:sphx1}\\
      x_2  &=&  r_3 \, \sin\theta_1 \, \sin\phi~, \\
      x_3  &=&  r_3 \, \cos\theta_1 ~. \label{eq:sphx3}
\end{eqnarray}

Let us examine the situation in $\mathbb R^4$ using the same technique. Given the point $P$ with radial distance $r_4$ from the origin, we decompose $\overrightarrow{OP}$ into two vectors
$\overrightarrow{OA}+\overrightarrow{OB}$, where $\overrightarrow{OA}$ lies along the $x_4$-axis and has magnitude
$r_4\,\cos\theta_2$ and $\overrightarrow{OB}$ lies in the ``plane'' and has magnitude $r_{3}=r_4\,\sin\theta_2$. The point $B$ thus has spherical coordinates
$(r_3, \phi, \theta_1)$ in the three-dimensional space, so $x_1, x_2, x_3$ are as in \eqref{eq:sphx1}--\eqref{eq:sphx3}.
Therefore
\begin{eqnarray*}
      x_1  &=&  r_4 \, \sin\theta_2 \, \sin\theta_1 \, \cos\phi ~, \\
      x_2  &=&  r_4 \, \sin\theta_2 \, \sin\theta_1 \, \sin\phi~, \\
      x_3  &=&  r_4 \, \sin\theta_2 \, \cos\theta_1~, \\
      x_4  &=&  r_4  \, \cos\theta_2 ~.
\end{eqnarray*}

If the reader has understood the previous constructions, the expressions in $p$ dimensions should be evident. Given the radius $r_p$ in $\mathbb{R}^p$, we project the position vector onto $\mathbb{R}^{p-1}$, obtaining the radius in the subspace $\mathbb{R}^{p-1}$ given by $r_{p-1} = r_p \sin{\theta_{p-2}}$. In this way we can perform a series of projections to get down to a space in which we already know the relations. This procedure leads to (dropping the subscript on $r_p$)
\[ \begin{array}{rcl}
  x_1 &=& r \, \sin{\theta_{p-2}} \, \sin{\theta_{p-3}} \cdots
  \sin{\theta_{3}} \, \sin{\theta_{2}} \, \sin{\theta_{1}} \, \cos{\phi}, \\
  x_2 &=& r \, \sin{\theta_{p-2}} \, \sin{\theta_{p-3}} \cdots
  \sin{\theta_{3}} \, \sin{\theta_{2}} \, \sin{\theta_{1}} \, \sin{\phi}, \\
  x_3 &=& r \, \sin{\theta_{p-2}} \, \sin{\theta_{p-3}} \cdots
  \sin{\theta_{3}} \, \sin{\theta_{2}} \, \cos{\theta_1}, \\
  x_4 &=& r \, \sin{\theta_{p-2}} \, \sin{\theta_{p-3}} \cdots
  \sin{\theta_{3}} \, \cos{\theta_2}, \\
  &\vdots&\\
  x_{p-1} &=& r \, \sin{\theta_{p-2}} \, \cos{\theta_{p-3}},\\
  x_p &=& r \, \cos{\theta_{p-2}}.
\end{array} \]

With this chapter as an exception, we will rarely refer explicitly to the angles $\phi, \theta_1, \ldots, \theta_{p-2}$.
However, we will frequently use $r = \sqrt{x_1^2 + \cdots + x_p^2}$.

%====================================================================
\section{The Sphere in Higher Dimensions} \label{sec:spheres}
%====================================================================

We will give definitions of the sphere and the ball in an
arbitrary number of dimensions that are analogous to the
definitions of the familiar sphere and ball we visualize embedded
in $\mathbb{R}^3$.

\begin{definition}
  The \emph{$(p-1)$-sphere of radius $\delta$ centered at $x_0$} is the set
  \[
    S_\delta^{p-1}(x_0) \overset{\text{def}}{=} \{x \in \mathbb{R}^p : |x-x_0| = \delta\}.
  \]
  The \emph{unit $(p-1)$-sphere}\footnote{Frequently, we will drop the ``unit,'' though it is still implied.} is the set
  $S^{p-1} \overset{\text{def}}{=} S_1^{p-1}(0)$.
\end{definition}

\begin{definition}
  The \emph{open $p$-ball of radius $\delta$ centered at $x_0$} is the set
  \[
    B_\delta^p(x_0) \overset{\text{def}}{=} \{ x \in \mathbb{R}^p: |x-x_0| < \delta \}.
  \]
  The \emph{open unit $p$-ball} is the set $B^p \overset{\text{def}}{=} B_1^p(0)$.
  The \emph{closed $p$-ball} is the set $\bar{B}^p \overset{\text{def}}{=} B^p \cup S^{p-1}$.
\end{definition}

Notice we call the sphere that we picture embedded in
$\mathbb{R}^p$ the $(p-1)$-sphere because it is
$(p-1)$-dimensional. It requires $p-1$ angles to define one's
location on the sphere, as we saw in Section \ref{sec:pdim}.
However, it requires $p$ coordinates to locate a point in the
ball because we must also specify $r$; this justifies the notation $B^p$. Notice also that the
$(p-1)$-sphere is the boundary of the $p$-ball when we think of these as subsets of
$\mathbb{R}^p$; we write $S^{p-1} = \partial B^p$.

Let us compute the surface area of the $(p-1)$-sphere. Towards this goal, we recall that that the gamma function is
defined by
\[
    \Gamma(z) = \int_0^\infty e^{-t}t^{z-1}\,dt,
\]
for any $z \in \mathbb{C}$ such that $\text{Re}(z) > 0$. Then

\begin{lemma}
\label{lemma:GammaFunction}
For all $p \in \mathbb{C}$ such that $\text{Re}(p) >0$, we have
\[
    \int_0^{+\infty} e^{-r^2}r^{p-1}\,dr = {1 \over 2} \, \Gamma \left({p\over 2}\right).
\]
\end{lemma}

\begin{proof}
Using the substitution $u = r^2$,
\[
    \int_0^{+\infty} e^{-r^2}r^{p-1}\,dr = \int_0^\infty e^{-u}u^{p-1 \over 2}\,{dr \over 2 \sqrt{u}} = {1 \over 2} \int_0^{+\infty} e^{-u}u^{{p \over 2}-1}\,du = {1 \over 2} \, \Gamma \left({p\over 2}\right),
\]
as sought. \qed
\end{proof}

\begin{proposition}
  If $\Omega_{p-1}$ denotes the solid angle in $\mathbb R^p$ (equivalent numerically to the surface area of $S^{p-1}$), then
  \[
    \Omega_{p-1} = {2\pi^{p/2} \over \Gamma(p/2)}.
  \]
  \label{lem:S}
\end{proposition}

Before proving this we will digress slightly and discuss how the surface area of the $(p-1)$-sphere relates to the volume of the $p$-ball. First, consider the  $p$-ball of radius $r$ centered at the origin. Notice that the radius $r$ completely determines such a ball. If we were to determine the volume $V_p$ of this $p$-ball, we would obtain $V_p = c\,r^p$ where $c$ is some constant. Here, we have determined that $V_p \propto r^p$ because $V_p$ must have dimensions of $[\text{length}]^p$ and the only variable characterizing the $p$-ball is $r$ (which has dimensions of length). Now if we differentiate $V_p$ with respect to $r$, we get
\[
    {dV_p \over dr} = (p-1)\,c\,r^{p-1}.
\]
Thus, if the radius of a $p$-ball changes by an infinitesimal amount $\delta r$, its volume will change by some infinitesimal amount $\delta V_p$, and
\begin{equation} \label{eq:dV1}
    \delta V_p = (p-1)\,c\,r^{p-1}\,\delta r.
\end{equation}
But in this case, the small change in volume $\delta V$ should equal the surface area $A_{p-1}(r)$ of the $p$-ball multiplied by the small change in radius $\delta r$,
\begin{equation} \label{eq:dV2}
    \delta V_p = A_{p-1}(r)\,\delta r.
\end{equation}
From \eqref{eq:dV1} and \eqref{eq:dV2} we see that
\[
    A_{p-1}(r) = (p-1)\,c\,r^{p-1},
\]
and if we let $\Omega_{p-1} = (p-1)\,c$ denote the numerical value of the surface area when $r=1$ (as in the statement of Proposition \ref{lem:S}), we get\footnote{We should stress that, since a physicist assigns physical
dimensions to each and every quantity, s/he would differentiate between the solid angle $\Omega_{p-1}$ and the surface area $A_{p-1}$ of the $(p-1)$-sphere. Although numerically the two quantities are equal for the unit sphere since $r=1$, they are different quantities since $\Omega_{p-1}$ is dimensionless while $A_{p-1}$ has dimensions of [length]$^{p-1}$.}
\begin{equation} \label{eq:romega}
    A_{p-1}(r) = \Omega_{p-1} \, r^{p-1}.
\end{equation}
We see that if we want to carry out an integral over $\mathbb{R}^p$ when the integrand depends only on $r$, we can use the differential volume element
\[
    dV_p  = A_{p-1}(r)\,dr = r^{p-1}\,\Omega_{p-1}\,dr.
\]
Now we are ready to prove Proposition \ref{lem:S}.

\begin{proof}
  Consider the integral,
  \[
    J = \int_{-\infty}^\infty \, dx_1 \int_{-\infty}^\infty \,
          dx_2 \cdots \int_{-\infty}^\infty \, dx_p \, e^{-(x_1^2 + x_2^2 + \cdots + x_p^2)}.
  \]
  This is really
  $$
    J = \left( \int_{-\infty}^\infty e^{-x^2} \, dx \right)^p = \left( \sqrt{\pi} \right) ^p .
  $$
  Using spherical coordinates however, we can write
  \[
    J = \int_{S^{p-1}}\int_0^{+\infty}  \, dV_p  ~  e^{-r^2}
      = \Omega_{p-1} \, \int_0^\infty e^{-r^2} \, r^{p-1} \, dr,
  \]
  since $\Omega_{p-1}$ is  just a constant. Therefore
  \[
    \Omega_{p-1} = {\pi^{p/2} \over \int_0^\infty e^{-r^2}\cdot r^{p-1}\,  dr}
  \]
  and, with the help of Lemma \ref{lemma:GammaFunction} we arrive at the advertised result.
 \qed
\end{proof}

\begin{remark}
  As a check, we can use this formula to determine the surface
  area (or circumference) of the 1-sphere (or circle) as well as
  the surface area of the 2-sphere, which is the familiar sphere
  that we embed in $\mathbb{R}^3$. As expected,
  \[
    \Omega_1 = 2\pi, \quad \Omega_2 = {2\pi^{3/2} \over
    \Gamma(3/2)} = 4\pi,
  \]
  using $\Gamma(3/2) = \sqrt{\pi}/2$. It is interesting to
  consider the 0-sphere, i.e., the sphere that we visualize
  embedded in $\mathbb{R}$. $S^0$ consists of all points on the real line that are
  unit distance from the origin, so $S^0 = \{-1,1\}$. Using
  Lemma \ref{lem:S}, we find $\Omega_0 = 2$, since $\Gamma(1/2) =
  \sqrt{\pi}$. That is, the surface area of just two points on the real
  line is finite! This is consistent with standard calculus, though. On the real line, the radial distance is the absolute value
  of $x$. Hence, the concept of a  function $f(r)$ possessing spherical symmetry  coincides with the concept of an even function, $f(x) = f(-x)$.
  Then
  \[
    \int_{-\infty}^{+\infty} f(x)\, dx = 2 \int_0^\infty f(x)\, dx
                                                         = \int_0^\infty f(r) \, \Omega_0 \, r^0 \, dr.
  \]
\end{remark}

%------------------------------------------------------------------
\section{Arc Length in Spherical Coordinates}
%------------------------------------------------------------------

Now, let us compute the formula for the infinitesimal arc length $\delta s$ associated with a small displacement along a curve in $\mathbb{R}^p$.
In this section, we use infinitesimals in the physicist's way; that is, $\delta s$ and the other small lengths involved represent tiny displacements
that are small enough so that the errors involved in our approximations (such as assuming a small portion of a curve is almost straight) are as
small as desired (say, less than some given $\epsilon > 0$). We only use the formulas we derive under this assumption in the limit as $\delta s \to 0$
and likewise with the other infinitesimal quantities, so our approximations introduce no error into future calculations.

Working in $\mathbb{R}^p$, let us use the spherical coordinates we developed in Section \ref{sec:pdim}. Consider a point $P$ with coordinates
$(r,\phi,\theta_{1},\ldots,\theta_{p-2})$. Consider a small displacement vector $\delta \vec{s} = \overrightarrow{PP'}$ at the point $P$ in an arbitrary direction.
One component of this displacement is along the radial direction and has magnitude $\delta r$. The other $p-1$ components of $\delta \vec{s}$ are orthogonal to $\delta \vec{r}$ and
thus\footnote{The remaining components of $\delta \vec{s}$ lie in a surface where $\delta r = 0$ or $r = $ const. This is a $(p-1)$-sphere.} are along the surface $S^{p-1}_r(0)$. Thus by the \emph{Pythagorean Theorem}, we have so far $\delta s^2 = \delta r^2 + \cdots$ where the rest to be filled in
involves small changes in the angles $\theta_{p-2},\ldots,\theta_1,\phi$.

\begin{figure}[h!]
     \begin{center}
    \setlength{\unitlength}{1mm}
    \begin{picture}(80,80)
        \put(0,0){\includegraphics[width=8cm]{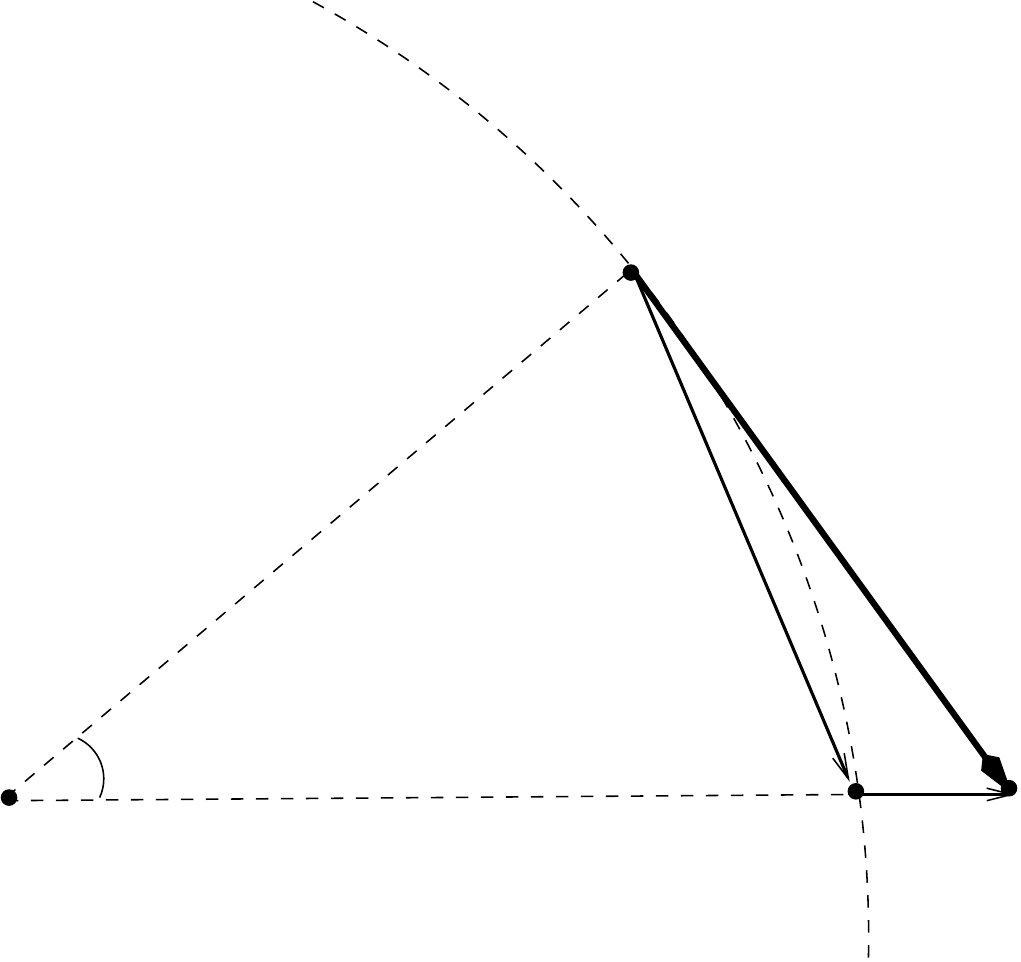}}
        \put(-1,9){$O$}  \put(50,55){$P$} \put(63,9){$Q$} \put(78,9){$P'$}
        \put(65,35){$\delta\vec s$}   \put(52,35){\rotatebox{-64}{$r \, \delta\omega_{p-1}$}}
        \put(71,9){$\delta r$}  \put(9,15){\small $\delta \omega_{p-1}$}
        \end{picture}
    \end{center}
    \caption{\footnotesize The plane containing the points $O,P,P'$. The displacement vector $\delta\vec{s}$ has a radial component of magnitude $\delta r$ and a
                   component  along the $(p-1)$-sphere of magnitude $r\,\delta\omega_{p-1}$.}
    \label{fig:ds}
\end{figure}
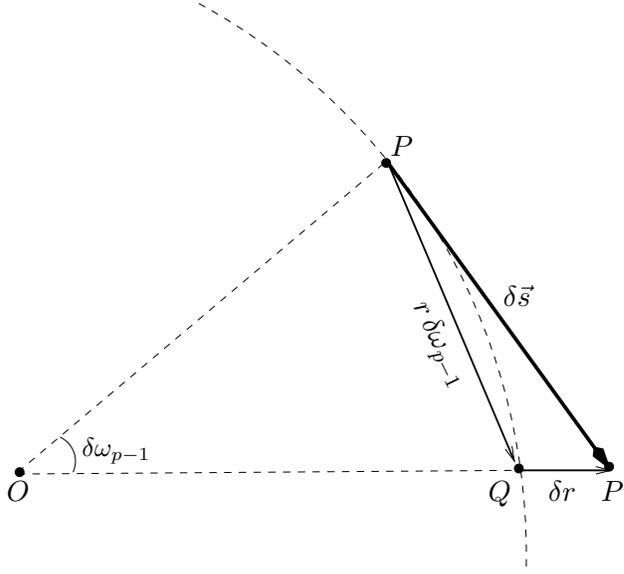

Let us consider the plane containing the vectors $\overrightarrow{OP}$ and $\overrightarrow{PP'}$,
where $O$ is the origin  (see Figure \ref{fig:ds}). We are familiar with the formula for an infinitesimal arc length in this two-dimensional plane, namely
\begin{equation}
\label{eq:ds}
    \delta s^2 = \delta r^2 + r^2 \, \delta \omega_{p-1}^2,
\end{equation}
where $\delta \omega_{p-1}$ is the small angle between the vectors $\overrightarrow{OP}$ and $\overrightarrow{OP'}$.
Note $\delta \omega_{p-1}$ is the angular displacement that occurs along the unit $(p-1)$-sphere.

Let us constrain the displacement $\delta \vec s$ to lie along a sphere by requiring $\delta r = 0$. If we compute $\delta s$ in this special case, we have found $\delta \omega_{p-1} = \delta s/r$ which we can plug into \eqref{eq:ds} to obtain the general result.

We can determine $\delta \omega_{p-1}$ inductively, starting in the plane with $p=2$. We know that on the unit circle $\delta s = r \, \delta \phi$, so
\[
    \delta \omega_1^2 = {\delta s^2 \over r^2} = \delta \phi^2.
\]

In $\mathbb{R}^3$, $\delta \vec s$ has a component $r \, \delta \theta_1$ along the direction of increasing $\theta_1$, so that $\delta s^2 = r^2 \, \delta \theta_1^2$ or $\delta \omega_2^2 = \delta \theta_1^2 + \cdots$,
and the remaining orthogonal components are
parallel to the $\mathbb{R}^2$ plane. Projecting $\overrightarrow{PP'}$ onto $\mathbb{R}^2$, we see that the infinitesimal displacement lies on a
circle  of radius $\sin \theta_1$. Using our result 
from $\mathbb{R}^2$, we get
\[
    \delta \omega_2^2 = \delta \theta_1^2 + \sin^2 \theta_1 \, \delta \omega_1^2 = \delta \theta_1^2 + \sin^2 \theta_1 \, \delta \phi^2,
\]
the familiar result in $\mathbb{R}^3$.

We can move to $\mathbb{R}^4$ in an analogous fashion. An infinitesimal arc length $\delta \vec s$ on the 3-sphere has a component $r \, \delta \theta_2$ along the
direction of increasing $\theta_2$, so that $\delta s^2 = r^2 \, \delta \theta_2^2$ or $\delta \omega_3^2 = \delta \theta_2^2 + \cdots$, and the remaining orthogonal components are parallel to $\mathbb{R}^3$
which we picture as a plane. Projecting $\overrightarrow{PP'}$ onto $\mathbb{R}^3$, we see that the infinitesimal displacement lies on a
2-sphere of radius $\sin \theta_2$, so
\[
    \delta \omega_3^2 = \delta \theta_2^2 + \sin^2 \theta_2 \, \delta \omega_2^2 = \delta \theta_2^2 + \sin^2 \theta_{p-2} \, \delta \theta_1^2 + \sin^2 \theta_2 \, \sin^2 \theta_1 \, \delta \phi^2.
\]

It is clear that this pattern continues, so that in $\mathbb{R}^p$, 
\[
    \delta \omega_{p-1} = \delta \theta_{p-2}^2 + \sin^2 \theta_{p-2} \, \delta \omega_{p-2}^2.
\]
By  filling in the form of $\delta \omega_{p-2}^2$ in this expression, we get
\[
    \delta \omega_{p-1} = \delta \theta_{p-2}^2 + \sin^2 \theta_{p-2} \, (\delta \theta_{p-3}^2 + \sin^2 \theta_{p-3} \, \delta \omega_{p-3}^2).
\]
Continuing to expand this expression, we come up with
\begin{align*}
    \delta \omega_{p-1}^2 = \ & \delta \theta_{p-2}^2 + \sin^2 \theta_{p-2} \, \delta \theta_{p-3}^2 + \sin^2 \theta_{p-2} \, \sin^2 \theta_{p-3} \, \delta \theta_{p-4}^2 \\
                        & + \cdots + \sin^2 \theta_{p-2} \, \sin^2 \theta_{p-3} \, \cdots \, \sin^2 \theta_1 \, \delta \phi^2.
\end{align*}
Putting this into \eqref{eq:ds} completes our computation.

%====================================================================
\section{The Divergence Theorem in $\mathbb{R}^p$} \label{sec:div}
%====================================================================

Let us now return to the use of Cartesian coordinates to give an intuitive ``derivation''
of the divergence theorem in $p$ dimensions, which we will use numerous times
in Chapter \ref{ch:sph}. For brevity, our justification of this theorem will be rather physical.
For a more mathematical treatment, the interested reader should consult a calculus book (for example, \cite{Calc:Stewart}) for the theorem in $\mathbb{R}^3$ and
see if s/he can generalize the proof there to $\mathbb{R}^p$. For a rigorous proof of the divergence theorem in $p$ dimensions,
s/he may consult an analysis text (for example, \cite{Anal:Rudin}).

\begin{theorem}
    Let $\vec{F}$ be a continuously differentiable vector field defined in the neighborhood of some closed, bounded domain $V$ in $\mathbb{R}^p$
    which has smooth boundary $\partial V$. Then\footnote{Here and from now on, $d^p x = dx_1 dx_2 \cdots dx_p$.}
    \begin{equation} \label{eq:divthm}
        \int_V \nabla \cdot \vec{F} \, d^p x = \oint_{\partial V} \vec{F} \cdot \hat{n} \, d \sigma,
    \end{equation}
    where $\hat{n}$ is the unit outward normal vector on $\partial V$ and $d \sigma$ is the differential element of surface area on $\partial V$.
\end{theorem}

\begin{poof}
    We will interpret $\vec{F}$ as the flux density of some $p$-dimensional fluid moving through the volume $V$. In unit time at a point $x$, the
    volume of fluid which flows past an arbitrarily oriented unit surface with unit normal vector $\hat{n}$ is
    given by $\vec{F}(x) \cdot \hat{n}$.

    Let us fill the interior
    of $V$ with a ``grid'' of disjoint boxes, none intersecting the boundary $\partial V$ --- see Figure \ref{fig:divthm}. If the boxes are comparable in size to $V$,
    they will not make a complete covering of the interior, since many regions of $V$ near $\partial V$ will remain uncovered;
    however, as the lengths of the box edges approach
    zero, the entire interior of $V$ can be covered by these boxes.
    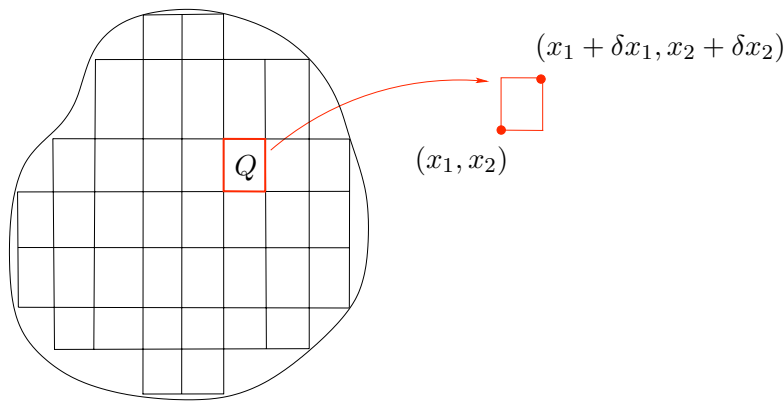
\begin{figure}[h!]
        \begin{center}
            \setlength{\unitlength}{1mm}
            \begin{picture}(90,60)
            \put(0,0){\includegraphics[width=0.55\linewidth]{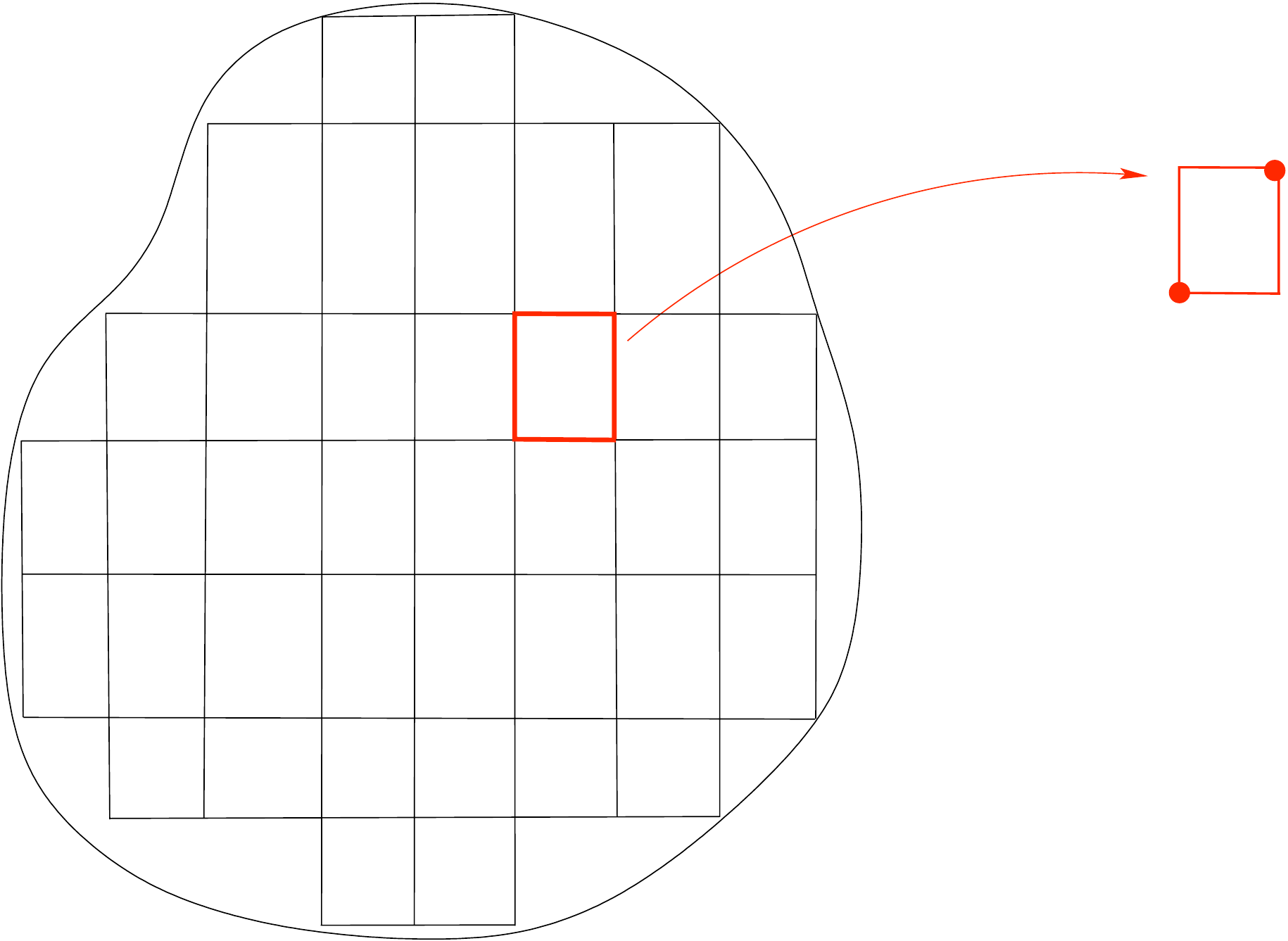}} \put(54,31){$(x_1,x_2)$}  \put(70,46){$(x_1+\delta x_1, x_2 + \delta x_2)$} \put(30,30){$Q$}
            \end{picture}
        \end{center}
        \caption{An example of a domain $V$ in $\mathbb{R}^2$. A similar picture would describe the situation in arbitrary $\mathbb{R}^p$.
            We consider one box $Q$ with bottom-left corner at $(x_1,x_2)$ and top-right corner at $(x_1+\delta x_1,x_2+\delta x_2)$.}
        \label{fig:divthm}
    \end{figure}
    We will consider one small box $Q$ in the interior of $V$, say with one corner at $(x_1,x_2,\ldots,x_p)$ and another at $(x_1+\delta x_1,x_2+\delta x_2,\ldots,x_p+\delta x_p)$,
    where the $\delta x_i$ are all positive.
    Consider the integral
    \[
        J = \oint_{\partial Q} \vec{F} \cdot \hat{n} \, d \sigma.
    \]
    The boundary $\partial Q$ consists of $2p$ planes, two orthogonal to each coordinate axis. The two planes orthogonal to the $x_1$ axis have surface area equal to $\delta x_2\,\delta x_3\cdots \delta x_p$. Since $\vec{F}$ is continuous, the integral of
    $\vec{F} \cdot \hat{n}$ over these surfaces can be written as\footnote{Here, $\hat{x}_1$ is a unit vector in the direction of increasing $x_1$.}
    \[
        \vec{F}(x_1+\delta x_1,x_2,\ldots,x_p) \cdot \hat{x}_1 \, \delta x_2 \, \delta x_3 \cdots \delta x_p,
    \]
    and
    \[
        \vec{F}(x_1,x_2,\ldots,x_p) \cdot (-\hat{x}_1) \, \delta x_2 \, \delta x_3 \cdots \delta x_p,
    \]
    since $\vec{F}$ changes very little over this plane as $\delta x_1,\delta x_2,\ldots,\delta x_p \to 0$. Now, there was nothing special about the $x_1$-axis, so we see that for    each $x_i$-axis, the integral $J$ will include a term
    \[
        \left[ \vec{F}(x_1,\ldots,x_i+\delta x_i,\ldots,x_p) - \vec{F}(x_1,\ldots,x_i,\ldots,x_p) \right] \cdot \hat{x}_i \, {\delta x_1 \cdots \delta x_p \over \delta x_i}.
    \]
    This can be rewritten as
    \[
        \left[ {\vec{F}(x_1,\ldots,x_i+\delta x_i,\ldots,x_p) - \vec{F}(x_1,\ldots,x_i,\ldots,x_p) \over \delta x_i} \right] \cdot \hat{x}_i \, \delta x_1 \cdots \delta x_p,
    \]
    which becomes
    \[
        {\partial \vec{F}(x_1,\ldots,x_p) \over \partial x_i} \cdot \hat{x}_i \, \delta x_1 \cdots \delta x_p
        \quad \text{or} \quad
        {\partial F_i(x_1,\ldots,x_p) \over \partial x_i} \, \delta x_1 \cdots \delta x_p
    \]
    in the limit with which we are concerned. Putting all these terms together, we get
    \[
        J = \sum_{i=1}^p {\partial F_i(x_1,\ldots,x_p) \over \partial x_i} \, \delta x_1 \cdots \delta x_p.
    \]
    But this is just
    \[
        \nabla_p \cdot \vec{F}(x_1,\ldots,x_p) \, \delta x_1 \cdots \delta x_p = \int_Q \nabla \cdot \vec{F} \, d^px
    \]
    in our limit. We have thus shown that
    \[
        \int_Q \nabla \cdot \vec{F} \, d^p x = \oint_{\partial Q} \vec{F} \cdot \hat{n} \, d \sigma
    \]
    for the special case of an infinitesimal box $Q$ inside $V$. Let us now sum this result over all the boxes inside $V$,
    \begin{equation} \label{eq:almostdivthm}
        \sum_{Q_i} \ \int_{Q_i} \nabla \cdot \vec{F} \, d^p x \ = \ \sum_{Q_i} \ \oint_{\partial Q_i} \vec{F} \cdot \hat{n} \, d \sigma,
    \end{equation}
    where we are concerned with the limit as the lengths of the box edges go to zero and the number of boxes in the covering approaches infinity.
    In this limit, the volume of the uncovered regions of $V$ near $\partial V$ approaches zero. Since $\vec{F}$ is continuously differentiable,
    it follows that $\nabla \cdot \vec{F}$ is continuous and thus bounded over the closed and bounded region $V$. Thus, the integral of
    $\nabla \cdot \vec{F}$ over regions in $V$ not covered by boxes approaches zero.
    The left side of \eqref{eq:almostdivthm} therefore becomes
    \begin{equation} \label{eq:divleft}
        \sum_{Q_i} \int_{Q_i} \nabla \cdot \vec{F} \, d^p x \ \longrightarrow \ \int_{V} \nabla \cdot \vec{F} \, d^p x
    \end{equation}
    in this limit. Now, in the right side of \eqref{eq:almostdivthm}, let us consider all the planes which bound the boxes $Q_i$ that are included
    in the sum. Each ``interior'' plane appears twice in the sum, once as the ``right'' side of one box and a second time as the ``left'' side of
    another box. In each of these appearances, both $\vec{F}$ and $d\sigma$ remain the same but the vector $\hat{n}$ shows up with opposite sign.
    Thus, the integral over all the interior planes vanishes. The only terms that are not canceled in the integral on the right side of
    \eqref{eq:almostdivthm} make up the integral over the ``exterior'' planes, which we write as $\partial \bigcup Q_i$. That is,
    \[
        \sum_{Q_i} \ \oint_{\partial Q_i} \vec{F} \cdot \hat{n} \, d \sigma \ = \oint_{\partial \bigcup Q_i} \vec{F} \cdot \hat{n} \, d \sigma,
    \]
    so using \eqref{eq:divleft},
    \begin{equation} \label{eq:closerdivthm}
        \int_{V} \nabla \cdot \vec{F} \, d^p x = \oint_{\partial \bigcup Q_i} \vec{F} \cdot \hat{n} \, d \sigma.
    \end{equation}
    Now we are very close to \eqref{eq:divthm}, but there is one difficulty due to the fact that the outward normal vector $\hat{n}$ for
    ${ \partial} \bigcup_i Q_i$ always points along one of
    the coordinate axes while $\hat{n}$ can point in any direction for $\partial V$. To reconcile this difference, we realize that integrating
    $\vec{F} \cdot \hat{n}$ over a closed surface gives us the volume of fluid that has passed through this surface in unit time.
    The same volume of fluid must pass through both ${ \partial} \bigcup_i Q_i$ and $\partial V$ since
    ${ \partial} \bigcup_i Q_i \to \partial V$. Therefore,
    \[
        \oint_{{ \partial} \bigcup_i Q_i} \vec{F} \cdot \hat{n} \, d \sigma = \oint_{\partial V} \vec{F} \cdot \hat{n} \, d \sigma.
    \]
    Combining this with \eqref{eq:closerdivthm} completes the proof. \qed
\end{poof}

%====================================================================
\section{$\Delta_p$ in Spherical Coordinates} \label{sec:pdimlap}
%====================================================================

To compute $\Delta_p$ in spherical coordinates, we could use the chain rule. This would involve converting all the derivatives with respect to $x_i$ in \eqref{eq:pLaplacian} to derivatives with respect to spherical coordinates as we did for $\Delta_2$ and $\Delta_3$ in Section \ref{sec:SepVar}. Such an undertaking would be quite messy, however, and adds little additional insight.

Another approach would be to use the general formula from differential geometry
$$
     \Delta  = {1\over \sqrt{g}} \, \partial_\mu \sqrt{g} g^{\mu\nu} \partial_\nu  ~,
$$
where $g_{\mu\nu}$ is the metric tensor of the space of interest, $g=\det[g_{\mu\nu}]$ and $g^{\mu\nu}$ the inverse of the metric tensor.
Since we have avoided this route thus far, we will not use it here either.

We can can actually use an integration trick to determine the form of the
$p$-dimensional Laplace operator $\Delta _p$ in spherical coordinates. This is the approach followed below.

\begin{figure}[h!]
    \begin{center}
        \setlength{\unitlength}{1mm}
        \begin{picture}(65,40)
            \put(0,0){\includegraphics[width=7cm]{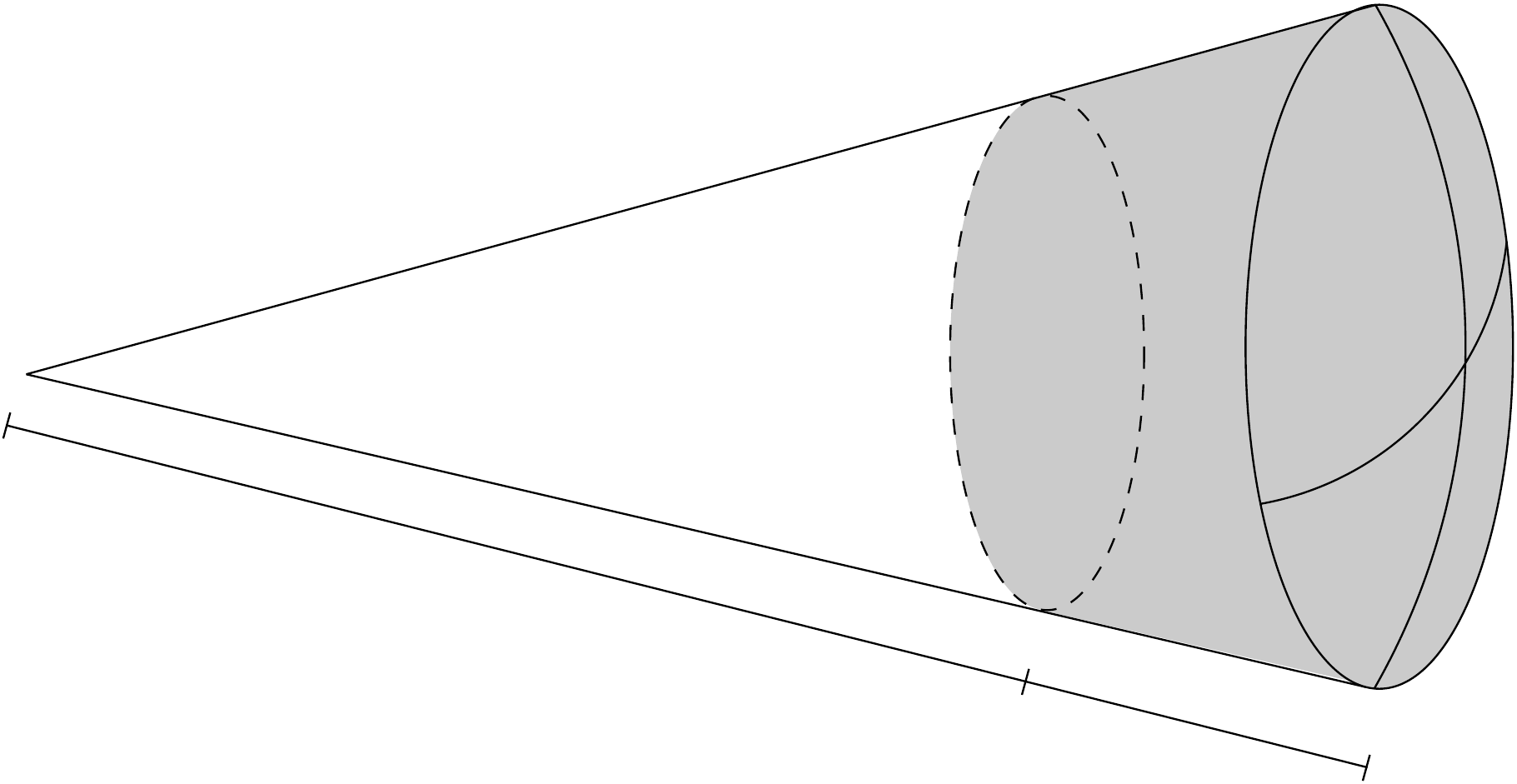}} \put(-5,18){$\mathcal O$}  \put(8,18){$\omega$} \put(52,-2){$\delta r$} \put(22,7){$r$} \put(52,36){$\mathcal D$}
        \end{picture}
    \end{center}
    \caption{A region $\mathcal D$ in $\mathbb{R}^p$.}
    \label{fig:cone}
\end{figure}
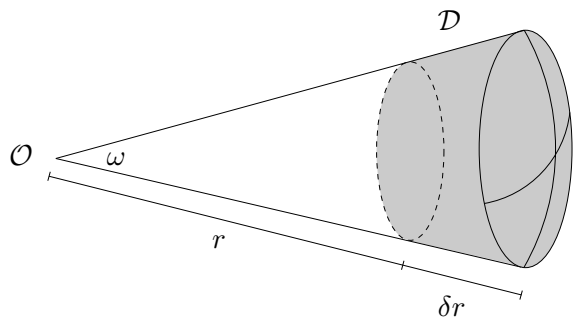

Consider the $p$-dimensional cone-like region depicted in Figure
\ref{fig:cone}. Let $\mathcal{D}$ be the shaded region (a truncated cone), which has
left and right boundaries given by portions of $S_r^{p-1}(0)$ and
$S_{r+\,\delta r}^{p-1}(0)$ respectively. We are concerned with the limit where the quantity $\delta r \to 0$. This region spans a solid
angle of $\omega$, which means the surface area of the region is a fraction $\omega / \Omega_{p-1}$ of the surface area of a complete
sphere of the same radius. Using \eqref{eq:romega}, this implies
that the surface area of the left boundary
is $\omega r^{p-1}$ while that of the right boundary is $\omega
(r+\,\delta r)^{p-1}$.

Now let $f:\mathbb{R}^p \rightarrow \mathbb{R}$ be a twice continuously differentiable
function, and use the divergence theorem in $p$ dimensions to find\footnote{We introduce here our notation for a normal derivative: ${\partial f \over \partial n} = \nabla_p f \cdot \hat{n}$.}
\[
    \underset{\mathcal{D}}{\int} \Delta_p f \, d^px = \int_{\mathcal{D}} \nabla_p \cdot (\nabla_p f) \, d^px =
    \underset{\partial \mathcal{D}}{\int} \nabla_p f \cdot \hat{n} \,
    d\sigma = \underset{\partial \mathcal{D}}{\int} {\partial f \over \partial n} \,
    d\sigma,
\]
where $\hat{n}$ is the external unit
normal vector to $\mathcal{D}$ and $d\sigma$ is a differential
element of surface area. By the \emph{mean value theorem for integration},
\[
    \int_\mathcal{D} \Delta_p f \, d^p x = \Delta_p f(x^*) \int_\mathcal{D} d^p x = \Delta_p f(x^*) \cdot \text{vol}(\mathcal D),
\]
for some $x^*$ in $\mathcal D$. But in the limit as the region of integration
becomes infinitesimally small, it does not matter which $x^*$ we use in $\mathcal D$ since $f$ is continuous, so
\[
    \underset{\mathcal{D}}{\int} \Delta_p f \, d^px \to \Delta_pf \cdot \text{vol}(\mathcal{D})
    \quad \text{as} \quad \text{vol}(\mathcal{D}) \to 0.
\]
Thus
\[
    \Delta_p f = \lim_{\text{vol}(\mathcal{D}) \rightarrow 0} \, \left[
    {1 \over \text{vol}(\mathcal{D})} \int_{\partial \mathcal{D}} \left( {\partial f \over \partial
    {n}} \right) \, d\sigma \right] \, .
\]
Note that in this limit, the volume of $\mathcal{D}$ approaches
$\omega r^{p-1}\,\delta r$. We can break up the integral in the numerator into one integral over the right bounding surface, one over the left bounding
surface, and one over the lateral surface which we will denote by $\partial \mathcal{D}'$. Since in this  limit the bounding surfaces are infinitesimally small, we can pull the integrands outside the integrals
and rewrite the above equation as
\[
    \Delta_p f = \lim_{\delta r \rightarrow 0 \atop \omega \rightarrow
    0} { {\partial f \over \partial r} \big|_{r+\,\delta r} \omega(r+\,\delta r)^{p-1}
    - {\partial f \over \partial r} \big|_r \omega r^{p-1}
    \over \omega r^{p-1} \, \delta r } +
    \lim_{\text{vol}(\mathcal{D}) \rightarrow 0}
    \left[ {\int_{\partial \mathcal{D}'} \left( {\partial f \over \partial
    {n}} \right) \, d\sigma \over \text{vol}(\mathcal{D})}\right] \, ,
\]
which becomes, using the binomial expansion,\footnote{Here, $\mathcal{O}(\delta r^2)$ denotes terms which are such that ${\text{terms} \over \delta r^2}$ remains
finite as $\delta r \to  0$. In particular, ${\text{terms} \over \delta r} \to 0$ as $\delta r \to 0$.}
\[
    \lim_{\delta r\rightarrow 0} {{\partial f \over \partial r}
    \big|_{r+\,\delta r} \left[ r^{p-1} + (p-1)r^{p-2}\,\delta r+ \mathcal{O}(\delta r^2) \right]
    - {\partial f \over \partial r} \big|_r r^{p-1} \over
    r^{p-1}\,\delta r} + {\text{\footnotesize contribution from} \atop \text{\footnotesize lateral surface,}}
\]
or,
\begin{equation} \label{eq:almostsphlap}
    \Delta_p f = \lim_{\delta r\rightarrow 0} \left[{ {\partial f \over \partial r}\big|_{r+\,\delta r} -  {\partial f \over \partial r} \big|_r
    \over \delta r} + {p-1 \over r}\,{\partial f \over \partial r}\right]+ {\text{\footnotesize contribution from} \atop \text{\footnotesize lateral surface.}}
\end{equation}
Notice that the  contribution-from-the-lateral-surface term
\[
    \lim_{\text{vol}(\mathcal{D}) \rightarrow 0} \, \left[ {1\over \text{vol}(\mathcal{D}) } \, \int_{\partial \mathcal{D}'}   {\partial f \over \partial {n}}  \, d\sigma \right]
\]
contains directional derivatives only in directions orthogonal to the radial direction. This term thus contains no derivatives with respect to $r$ and only derivatives with respect to
the angles $\phi,\theta_1,\theta_2,\ldots, \theta_{p-2}$; we will indicate it by
\[
     {1 \over r^2}     \Delta_{S^{p-1}} f ,
\]
where we determined the prefactor $1 / r^2$ through dimensional analysis. We call the operator $\Delta_{S^{p-1}}$ the \emph{spherical Laplace operator in $p-1$ dimensions}. Equation \eqref{eq:almostsphlap} thus implies the following proposition.

\begin{proposition}
\label{prop:AngLap}
  \[
    \Delta_p = {\partial^2  \over \partial r^2} + {p-1 \over r}
    \, {\partial  \over \partial r} + {1 \over r^2}
    \Delta_{S^{p-1}}.
  \]
\end{proposition}

Now that we are comfortable working in $\mathbb{R}^p$ let us move on to briefly study the subject of orthogonal polynomials which will be very
useful when we come to our main topic of discussion.

 \CleanPage

%%%%%%%%%%%%%%%%%%%%%%%%%%%%%%%%%%%%%%%%%%%%%%%%%%%%%%%%%%%%%%%%%%%%%%%%%%%%%%%
\chapter{Orthogonal Polynomials} \label{ch:pol}
%%%%%%%%%%%%%%%%%%%%%%%%%%%%%%%%%%%%%%%%%%%%%%%%%%%%%%%%%%%%%%%%%%%%%%%%%%%%%%%

This chapter provides the reader with several useful facts from
the theory of orthogonal polynomials that we will require in order
to prove some of the theorems in our main discussion. In
what follows, we will assume the reader has taken an introductory
course in linear algebra and, in particular, is familiar with
inner product spaces and the Gram-Schmidt orthogonalization
procedure. If this is not the case, we refer him or her to \cite{LinAlg:Friedberg}.

%====================================================================
\section{Orthogonality and Expansions}\label{sec:inner}
%====================================================================

We will deal with a set $\mathcal{P}$ where the
members are real polynomials of finite degrees in one variable defined on some interval\footnote{For now,
we allow $I$ to be finite or infinite and  open, closed, or neither.} $I$. With vector addition defined as the usual addition of
polynomials and scalar multiplication defined as ordinary
multiplication by real numbers, it is easy to check that we have,
so far, a well defined vector space.

Let us put more structure on this space by adding an inner product.
For any $p, q \in \mathcal{P}$, define a function $\langle \cdot,\cdot \rangle_w : \mathcal{P} \times
\mathcal{P} \to \mathbb{R}$ by
\begin{equation}
    \langle p, q \rangle_w \overset{\text{def}}{=} \int_I  p(x)q(x)w(x)\,dx,
    \label{eq:inn}
\end{equation}
where $w$, called a \emph{weight function}, is some positive
function defined on $I$ such that
\[
    \int_I r(x)w(x)\,dx
\]
exists and is finite for all polynomials $r \in \mathcal{P}$. Notice that, since the product
of any two polynomials is itself a polynomial, the above
requirement ensures the existence of $\langle p,q \rangle_w$ for
all polynomials $p,q$. Moreover, the function \eqref{eq:inn} is symmetric and linear in both arguments and,
since $w$ is positive, $\langle r,r \rangle_w \geq 0$ for all polynomials $r$.  Finally, since $r^2$ is continuous for
all polynomials $r$,
\[
    \langle r, r \rangle_w = \int_I r(x)^2 \, w(x) \, dx = 0 \quad
    \text{if and only if} \quad r(x) = 0 \text{ for all } x \in I.
\]
For suppose $r(x_0) \neq 0$ for some $x_0 \in I$. By
continuity, $r$ must not vanish in some neighborhood about
$x_0$, and thus the above integral will acquire a positive value.
Therefore, $\mathcal{P}$ together with the function \eqref{eq:inn}
is a well defined inner product space.
We call $\langle \cdot , \cdot \rangle_w$ the \emph{inner product with respect to the weight $w$}.

The reader should recall the following two inequalities that hold in any inner product space.
We state them here for reference without any proofs which may be found in \cite{LinAlg:Friedberg}.

\begin{proposition}[Cauchy-Schwarz Inequality]
    For any vectors $p,q$ in an inner product space $\mathcal{P}$, we have
    $|\langle p,q \rangle| \leq \|p\| \|q\|$.
\end{proposition}

\begin{proposition}[Triangle Inequality]\hspace{-2mm}\footnote{This is also called the Minkowski inequality.}
    For any vectors $p,q$ in an inner product space $\mathcal{P}$, we have
    $\|p+q\| \leq \|p\| + \|q\|$.
\end{proposition}

Now that we have an inner product space, we can speak of
orthogonal polynomials. We call two polynomials $p, q$
\emph{orthogonal with respect to the weight\footnote{Although the weight function is important, often we tend
not to stress it and call the two polynomials simply \emph{orthogonal}. We also use the simplified
notation $\langle \cdot, \cdot \rangle$.} $w$} provided $\langle p,q \rangle_w = 0$.

Suppose now that we have a basis $\mathcal{B} =
\{\phi_n\}_{n=0}^\infty$ for $\mathcal{P}$ consisting of
orthogonal polynomials where $\phi_n$ is of degree $n$. We can
easily see that such a basis exists by beginning with a set of
monomials $\{x^n\}_{n=0}^\infty$ and, using the Gram-Schmidt process, to come up with the $\{\phi_n\}_{n=0}^\infty$:
We leave the first member $x^0=1$ of the basis as is. That is, $\phi_0=1$.
Then, from the second member $x^1=x$, we subtract its component along the first member to get
$\phi_1=x-{\langle x,1 \rangle \over \langle 1,1 \rangle}\cdot 1$, so that the second member of our new basis is orthogonal to the first.
We continue in this fashion, at each step subtracting from $x^n$ its components along the first $n-1$ orthogonalized basis members.
This creates an orthogonal basis for $\mathcal{P}$.

Now since $\mathcal{B}$ is a basis, given any polynomial $q \in
\mathcal{P}$ of degree $k$ we can write\footnote{We do not need to worry
about the convergence of the following series --- we will see in a moment that
it is actually a finite sum.}, for some scalars $c_n$,
\begin{equation}
    q = \sum_{n=0}^\infty c_n \phi_n,
    \label{eq:expans}
\end{equation}
i.e., we can expand $q$ in terms of the $\phi_n$. Since each
polynomial $\phi_n$ in $\mathcal{B}$ has degree $n$, we can
clearly write any $q$ in terms of the set $\mathcal{B}_k =
\{\phi_n\}_{n=0}^k$. Indeed, the space of all polynomials of
degree $k$ or less has dimension $k+1$, and the linearly
independent set $\mathcal{B}_k$ has $k+1$ elements. Using the
uniqueness of the expansion \eqref{eq:expans}, we can see that
\begin{equation}
    c_n = 0 \text{ for all } n> k = \deg{q}.
    \label{eq:termin}
\end{equation}
Since the coefficients are given by
\[
    c_n = {\langle q,\phi_n \rangle \over \langle \phi_n,\phi_n
    \rangle},
\]
we have the following result.

\begin{proposition}
  Given any polynomial $q$ and any set of orthogonal polynomials
  $\{\phi_n\}_{n=0}^\infty$, where $\phi_n$ has degree $n$,
  \[
    \langle \phi_n, q \rangle _w = \int_I \phi_n q \; w \, dx = 0, \text{ for all } n > \deg{q}.
  \]
  \label{fact:int0}
\end{proposition}

Since we will come across expressions of the form $\langle p, p
\rangle$ frequently, let us define the \emph{norm of p with
respect to the weight $w(x)$} as
\[
    \| p \|_w \overset{\text{def}}{=} \sqrt{\langle p,p \rangle_w}.
\]

Keeping in mind the way we constructed our set of orthogonal polynomials $\{\phi_n\}_{n=0}^\infty$
(requiring $\phi_0$ to have degree 0 --- i.e., requiring $\phi_0$ to be the constant polynomial $\phi_0 = a_0$ --- and choosing each
successive $\phi_n$ to be of degree $n$ and orthogonal to the previous $n$ members) the following
result should not surprise the reader.

\begin{proposition}
    Any two orthogonal bases of polynomials $\{\phi_n\}_{n=0}^\infty$ and $\{\psi_n\}_{n=0}^\infty$ for which
    $\deg{\phi_n} = \deg{\psi_n} = n$, must be the same up to some nonzero multiplicative factors.
    \label{fact:samebases}
\end{proposition}

\begin{proof}
    Since the $\{\phi_n\}_{n=0}^\infty$ constitute a basis, we can write any $\psi_m$
    as, using \eqref{eq:termin},
    \[
        \psi_m = \sum_{n=0}^m c_n \phi_n
    \]
    for some constants $c_n$ given by
    \[
        c_n = {1 \over \|\phi_n\|^2} \langle \psi_m, \phi_n  \rangle .
    \]
    By Proposition \ref{fact:int0}, $c_n = 0$ for all $n \neq   m$. We are thus left with
    \[
        \phi_m = c_m \psi_m,
    \]
    and we are done. \qed
\end{proof}

%====================================================================
\section{The Recurrence Formula}
%====================================================================

We can find a useful recursive relationship between any three
consecutive orthogonal polynomials $\phi_{n+1}, \phi_n,
\phi_{n-1}$. In what follows, we will let $k_n$, $\ell_n$ denote
the coefficients of the $x^n$ term and the $x^{n-1}$ term in
$\phi_n$, respectively.

\begin{proposition}
  Given any set of orthogonal polynomials
  $ \mathcal{B} = \{\phi_n\}_{n=0}^\infty$, where $\phi_n$ has degree $n$,
  \begin{equation}
    \phi_{n+1} - (A_nx+B_n)\phi_n + C_n\phi_{n-1} = 0,
    \label{eq:recurr}
  \end{equation}
  where
  \begin{equation}
    A_n = {k_{n+1} \over k_n}, \quad B_n = A_n \left( {\ell_{n+1} \over k_{n+1}}-{\ell_n \over k_{n}}
    \right), \quad C_n = {A_n \over A_{n-1}} \, {\| \phi_n \|^2 \over \|
    \phi_{n-1}\|^2},
    \label{eq:coeffic}
  \end{equation}
  and  $C_0 = 0$.
  \label{fact:recurr}
\end{proposition}

\begin{proof}
  From the formula given for $A_n$ in \eqref{eq:coeffic}, we see  that $\phi_{n+1} - A_n x \phi_n$ must be a polynomial of degree
  $n$ since the leading term of $\phi_{n+1}$ is canceled. Since the set $\mathcal{B}$ is a basis, we can thus write
  \[
    \phi_{n+1} - A_n x \phi_n = \sum_{j=0}^n \gamma_j \phi_j,
  \]
  using \eqref{eq:termin}. Taking the inner product of each side  with $\phi_{j'}$, where $0 \leq j' \leq n$, we find
  \[
    \gamma_{j'} = {1 \over \|\phi_{j'}\|^2} \, \left( \langle \phi_{n+1},
    \phi_{j'} \rangle - A_n \langle x \phi_n, \phi_{j'} \rangle
    \right),
  \]
  using the orthogonality of $\mathcal{B}$ and the linearity of  the inner product. From the definition of
  the inner product \eqref{eq:inn}, we see that the above equation  is the same as
  \[
    \gamma_{j'} = {1 \over \|\phi_{j'}\|^2} \, \left( \langle \phi_{n+1},
    \phi_{j'} \rangle - A_n \langle \phi_n, x \phi_{j'} \rangle
    \right).
  \]
  Since $j' \neq n+1$, the first inner product in the above
  equation vanishes. Since $\deg{(x \phi_{j'})} = \deg{(\phi_{j'})} + 1 = j'+1$, we can use Proposition \ref{fact:int0}
  to determine that the second inner product above is zero unless $j' = n-1$ or  $j'= n$. If we rename $\gamma_n = B_n$ and
  $\gamma_{n-1} = -C_n$, we have
  \[
    \phi_{n+1} - (A_n x + B_n) \phi_n + C_n \phi_{n-1} = 0,
  \]
  as in \eqref{eq:recurr}, with
  \begin{equation}
    B_n = {-A_n \over \|\phi_{n}\|^2} \, \langle \phi_n, x \phi_{n}
    \rangle, \quad C_n = {A_n \over \|\phi_{n-1}\|^2} \, \langle
    \phi_n, x \phi_{n-1} \rangle,
    \label{eq:coefform}
  \end{equation}
  and it remains to compute the coefficients $B_n, C_n$. We will
  start with $C_n$, rewriting $x \phi_{n-1} = x (k_{n-1}x^{n-1}+\cdots\,)$ as
  \begin{align*}
    k_{n-1} x^{n} + \text{lower order terms} =& \:{k_{n-1}
    \over k_n} \left(k_n x^n + \text{lower order terms} \right)\\
    =& \: {k_{n-1}\over k_n} \left( \phi_n + \text{lower order terms} \right).
  \end{align*}
  Since the lower order terms `die' in the inner product with
  $\phi_n$, \eqref{eq:coefform} implies
  \[
    C_n = {A_n \over \|\phi_{n-1}\|^2} \, {k_{n-1}\over k_n}\, \|
    \phi_n \|^2 = {A_n \over A_{n-1}} \, {\| \phi_n \|^2 \over \|
    \phi_{n-1}\|^2},
  \]
  as required. We calculate $B_n$ in a similar way, rewriting $x
  \phi_{n} = k_nx^{n+1} + \ell_nx^{n} + \cdots$ as
  \begin{align*}
    x\phi_n = & {k_n \over k_{n+1}} \left[ k_{n+1}x^{n+1} + \ell_{n+1}x^n -
    \left(\ell_{n+1} - {k_{n+1}\ell_n \over k_n} \right)x^n +
    \text{L.O.T.} \right]\\
    = & \ {k_n \over k_{n+1}} \phi_{n+1} + \left( {\ell_n \over k_{n}} - {
    \ell_{n+1} \over k_{n+1}} \right) k_nx^n + \text{L.O.T.}\\
    = & \ {k_n \over k_{n+1}} \phi_{n+1} + \left( {\ell_n \over k_{n}} - {
    \ell_{n+1} \over k_{n+1}} \right) \phi_n + \text{L.O.T.}
  \end{align*}
  where  L.O.T.   is an acronym which stands for ``lower order terms''.  Only the $\phi_n$ term will survive the inner product with
  $\phi_n$, and \eqref{eq:coefform} implies
  \[
    B_n = {-A_n \over \|\phi_{n}\|^2} \, \left( {\ell_n \over k_{n}} - {
    \ell_{n+1} \over k_{n+1}} \right) \|\phi_n\|^2 = A_n \left( {\ell_{n+1} \over k_{n+1}}-{\ell_n \over k_{n}}
    \right),
  \]
  and we are done.
\qed
\end{proof}

We see that if we know the coefficients $A_n, B_n, C_n$ then once we have chosen two consecutive members of our set of orthogonal
polynomials, the rest are determined.

%====================================================================
\section{The Rodrigues Formula}
%====================================================================

Until now, we spoke of the arbitrary interval $I$. For the remainder of our discussion, we will be most concerned with the
interval $[-1,1]$.

Consider the functions $\psi_n(x)$ defined on $[-1,1]$ for $n=0,1,\ldots$ as
\begin{equation}
  \psi_n(x) = {1 \over w(x)} \left({d \over dx} \right)^n \left[
  w(x)(1-x^2)^n\right].
  \label{eq:Rod}
\end{equation}
For arbitrary choice of the weight function $w$, we cannot say whether these functions are polynomials nor whether they are
orthogonal with respect to $w$. In an effort to force the $\psi_n$ to be polynomials, we can start by looking at $\psi_1$
and requiring it to be a polynomial of degree $1$,
\[
    \psi_1(x) = {w'(x) \over w(x)}(1-x^2) - 2x \overset{\text{req}}{=} ax+b.
\]
This creates a differential equation, which we can write as
\[
    {w' \over w} = {(a+2)x+b \over (1+x)(1-x)}.
\]
Integrating,
\[
    \ln{w} = \int {(a+2)x+b \over (1+x)(1-x)}\,dx + \text{const.}
\]
Decomposing the integrand into partial fractions to compute the integral, we find
\[
    \ln{w} = -{a+b+2 \over 2}
    \ln{(1-x)} + {b-a-2 \over 2} \ln{(1+x)}+ \text{const.},
\]
or,
\[
    w(x) = \text{const.} \times (1-x)^\alpha(1+x)^\beta,
\]
where $\alpha, \beta$ are the coefficients from the previous expression. Clearly, the arbitrary constant in the above
expression cancels when $w(x)$ is inserted into \eqref{eq:Rod}, so we will take it to be unity. We will also require $\alpha, \beta >
-1$ so that we can integrate any polynomial with respect to this weight. For our purposes, we will see that these restrictions on
$w(x)$ are all we need.

\begin{proposition}
  Let $w(x) = (1-x)^\alpha(1+x)^\beta$ with $\alpha, \beta > -1$.
  Then \eqref{eq:Rod} defines a set of polynomials
  $\{\psi_n\}_{n=0}^\infty$ where $\deg{\psi_n} = n$.
\end{proposition}

In what follows, we will use $(k)_\ell$ to denote the \emph{falling factorial},
\begin{align*}
    &(k)_\ell \overset{\text{def}}{=} k(k-1) \cdots (k-\ell+1) \text{ for each } \ell \in \mathbb{N},  \\
    &(k)_0 \overset{\text{def}}{=} 1,
\end{align*}
and
\[
    {n \choose k} \overset{\text{def}}{=} {n! \over k!(n-k)!}.
\]

\begin{proof}
    We can see this using the \emph{Leibnitz rule} for differentiating products,
$$
   \left( d \over dx \right)^n (fg) = \sum_{k=0}^n {n \choose k} {d^kf \over dx^k}\:
    {d^{n-k}g \over dx^{n-k}},
$$
which can be proved easily by induction.

With the weight $w(x)$ given previously, \eqref{eq:Rod} becomes
    \begin{align*}
        \psi_n(x) =& \: (1-x)^{-\alpha}(1+x)^{-\beta} \left( {d \over
        dx} \right)^n \left[ (1-x)^{\alpha+n}(1+x)^{\beta+n} \right]\\
        =& \sum_{k=0}^n {n \choose
        k} \left[ (1-x)^{-\alpha} \left({d \over dx} \right)^{\!\!k} \!\!\!(1-x)^{n+\alpha} \right]
        \left[  (1+x)^{-\beta} \left({d \over dx} \right)^{\!\!n-k} \!\!\!(1+x)^{n+\beta}
        \right]\\
        =& \sum_{k=0}^n {n \choose k} \left[ (-1)^k (n+\alpha)_k (1-x)^{n-k}
        \right] \left[ (n+\beta)_{n-k} (1+x)^{k} \right].
    \end{align*}
    We can see that $\psi_n$ is a polynomial. Let us examine the
    leading term in $\psi_n$. We see from the last line in the
    above equation that
    \begin{align*}
        \psi_n(x) =& \sum_{k=0}^n {n \choose k} (-1)^k (n+\alpha)_k (-x)^{n-k} (n+\beta)_{n-k}
        (x)^{k} \ + \ \text{L.O.T.}\\
        =& \: (-1)^n x^n \sum_{k=0}^n {n \choose k} (n+\alpha)_k
        (n+\beta)_{n-k} \ + \ \text{L.O.T.}
    \end{align*}
    Since the sum in the last line of the above equation is
    strictly positive, we see that the $x^n$ term in $\psi_n$ has
    a nonvanishing coefficient; i.e., the degree of $\psi_n$ is
    precisely $n$. \qed
\end{proof}

\begin{proposition}
  Let $w(x) = (1-x)^\alpha(1+x)^\beta$, for $\alpha, \beta > -1$.
  Then the $\psi_n$ defined in \eqref{eq:Rod} satisfy
  \begin{equation*}
    \int_{-1}^1 \psi_n(x) x^k w(x) \, dx = 0, \quad \text{for } \ 0
    \leq k < n.
  \end{equation*}
  \label{fact:Rod}
\end{proposition}

\begin{proof}
    Let $J$ be the above integral. Inserting \eqref{eq:Rod} for   $\psi_n$ and integrating by parts $n$ times,
    \begin{align*}
      J &= \int_{-1}^1 x^k \left({d \over dx} \right)^n \left[
      w(x)(1-x^2)^n\right] \, dx\\
      &= x^k \left({d \over dx} \right)^{\!\!n-1}\!\!
      \left[w(x)(1-x^2)^n\right] \Big|_{-1}^1 - \int_{-1}^1 kx^{k-1} \left({d \over dx} \right)^{\!\!n-1}\!\!
      \left[w(x)(1-x^2)^n\right] \,dx\\
      &\ \vdots\\
      &= \sum_{\ell=0}^k (-1)^\ell (k)_\ell \, x^{k-\ell} \left({d \over dx} \right)^{\!\!n-\ell-1}\!\!
      \left[w(x)(1-x^2)^n\right] \Big|_{-1}^1.
    \end{align*}
    Upon inserting the assumed form of $w(x)$, this becomes
    \[
        J = \sum_{\ell=0}^k (-1)^\ell (k)_\ell \, x^{k-\ell} \left({d \over dx} \right)^{n-\ell-1}
      \left[(1-x)^{n+\alpha}(1+x)^{n+\beta}\right] \Big|_{-1}^1\:.
    \]
    In each term of the above sum, the factors $(1-x)^{n+\alpha}$ and $(1+x)^{n+\beta}$ keep positive exponents after being operated upon by
    $\left({d \over dx} \right)^{n-\ell-1}$ so that $J$ vanishes
    after we evaluate the sum at $-1$ and $1$. \qed
\end{proof}

Since each $\psi_n$ has degree $n$ and is orthogonal to all polynomials with degree less than $n$, we have the following result.

\begin{corollary}
    \label{cor:rod}
    The formula \eqref{eq:Rod} defines an orthogonal set of polynomials $\{\psi_n\}_{n=0}^\infty$ known as the
    \emph{Jacobi polynomials} with $\deg \psi_n = n$ for each $n$.
\end{corollary}

We have thus found that \eqref{eq:Rod}, called the \emph{Rodrigues formula}, with an acceptable weight function $w(x)$, defines a set
of orthogonal polynomials. We are free to multiply each polynomial $\psi_n$ by an arbitrary constant without upsetting this fact. The
different classical orthogonal polynomials can be defined using the Rodrigues formula by adjusting this constant and the exponents
$\alpha, \beta$ in $w(x)$.

Now, suppose we discover that a set of polynomials $\{ \phi_n \}_{n=0}^\infty$ such that $\deg{\phi_n} = n$ is orthogonal with
respect to the weight $w(x)$. Can we conclude that these polynomials satisfy the Rodrigues formula \eqref{eq:Rod}?
Comparing Corollary \ref{cor:rod} and Proposition \ref{fact:samebases}, we can answer ``yes.''
That is, we can define the $\phi_n$ by the Rodrigues formula with
suitable multiplicative constants,
\[
    \phi_n(x) = {c_n \over w(x)} \left({d \over dx} \right)^n \left[
    w(x)(1-x^2)^n\right].
\]

%====================================================================
\section{Approximations by Polynomials} \label{sec:approx}
%====================================================================

The main result in this section will be the Weierstrass approximation theorem, which will allow us to use polynomials to
approximate continuous functions on closed intervals. First, we will introduce some ideas that we will use to prove this important
theorem.

Given a function $f:[0,1]\rightarrow\mathbb{R}$, we will use $B_n(x;f)$ to denote a \emph{Bernstein polynomial}, defined on the closed interval $[0,1]$ as
\[
    B_n(x;f) \overset{\text{def}}{=} \sum_{k=0}^n {n \choose k}
    f\left({k \over n}\right) x^k (1-x)^{n-k}
\]
for $n = 1,2,\ldots$. Since $f\left({k \over n}\right)$ is just a constant, these are clearly polynomials.

We will compute three special cases of the Bernstein polynomials.

\begin{lemma}
\label{Lemma:Bernstein}
For every natural number $n$,
  \begin{align}
    B_n(x;1)   &= 1, \label{eq:B1}                     \\
    B_n(x;x)   &= x, \label{eq:B2}                    \\
    B_n(x;x^2) &= x^2 + {x(1-x) \over n}. \label{eq:B3}
  \end{align}
\end{lemma}

\begin{proof}
  By the binomial theorem,
  \begin{equation}
    (x+y)^n = \sum_{k=0}^n {n \choose k} x^k y^{n-k}.
    \label{eq:forB1}
  \end{equation}
  Substituting $y=1-x$ gives
  \begin{equation}
    1 = \sum_{k=0}^n {n \choose k} x^k (1-x)^{n-k} = B_n(x;1),
    \label{eq:Bsum}
  \end{equation}
  proving \eqref{eq:B1}. Differentiating \eqref{eq:forB1} with respect to $x$,
  \[
    n(x+y)^{n-1} = \sum_{k=0}^n {n \choose k} kx^{k-1} y^{n-k},
  \]
  and multiplying by $x/n$,
  \begin{equation}
    x(x+y)^{n-1} = \sum_{k=0}^n {n \choose k} \left({k \over
    n}\right) x^k y^{n-k}.
    \label{eq:forB2}
  \end{equation}
  Substituting $y=1-x$ gives
  \[
    x = \sum_{k=0}^n {n \choose k} \left({k \over
    n}\right) x^k (1-x)^{n-k} = B_n(x;x),
  \]
  proving \eqref{eq:B2}. Differentiating \eqref{eq:forB2} with  respect to $x$,
  \[
    (x+y)^{n-1} + (n-1)x(x+y)^{n-2} = \sum_{k=0}^n {n \choose k}
    \left({k^2 \over n}\right) x^{k-1}y^{n-k},
  \]
  and multiplying by $x / n$,
  \[
    \left(x^2+{xy \over n}\right) (x+y)^{n-2} = \sum_{k=0}^n {n
    \choose k} \left({k \over n}\right)^2 x^k y^{n-k}.
  \]
  Substituting $y=1-x$ gives
  \[
    x^2 + {x(1-x) \over n} = \sum_{k=0}^n {n
    \choose k} \left({k \over n}\right)^2 x^k (1-x)^{n-k} =
    B_n(x;x^2),
  \]
  proving \eqref{eq:B3}.
\qed
\end{proof}

We are now well equipped to prove the following result.

\begin{theorem}[Weierstrass Approximation Theorem]
\label{fact:Weie}
  Let the function $f:[a,b] \rightarrow \mathbb{R}$ be continuous,
  and let $\epsilon > 0$. Then there exists a polynomial $p(x)$
  such that $|f(x)-p(x)| < \epsilon$ for all $x$ in the interval
  $[a,b]$.
\end{theorem}

We will show that the Bernstein polynomials, for sufficiently
large $n$, work to approximate the function $f$ to the required
accuracy.

\begin{proof}
  Let $\epsilon >0$. First, we will redefine the function $f$ so that we can work
  in the interval $[0,1]$. Using the linear transformation
  \[
    T(x) = (b-a)x +a,
  \]
  we set
  \[
    \tilde{f} = f \circ T: [0,1] \rightarrow \mathbb{R}.
  \]
  If we can find a polynomial $p$ to adequately approximate
  $\tilde{f}$, we can use the polynomial $p \circ T^{-1}: [a,b] \rightarrow \mathbb{R}$ as our
  required approximation of $f$. From here on, we will write $f$
  instead of $\tilde{f}$ for simplicity.

  Since $f$ is continuous on the closed and bounded interval $[0,1]$, we know
  that $f$ is bounded and uniformly continuous. By the definition of boundedness, there exists
  an $M$ such that
  \[
    |f(x)| < M, \text{ for every } x \in [0,1],
  \]
  which implies
  \begin{equation}
    |f(x) - f\left({k \over n}\right)| \leq \left| f(x) \right| + \left| f\left( {k \over n} \right) \right| < 2M, \text{ for all } x \in [0,1]
    \text{ and } 0 \leq k \leq n.
    \label{eq:boundM}
  \end{equation}
  By the definition of uniform continuity, there
  exists a $\delta > 0$ such that
  \begin{equation}
    \left|f(x) - f\left({k \over n}\right)\right| < {\epsilon \over 2},
    \text{ whenever } |x-{k\over n}| < \delta.
    \label{eq:unifcont}
  \end{equation}
  Now, let
  \[
    E = |f(x) - B_n(x;f)|,
  \]
  and we will estimate this error from above. Using
  \eqref{eq:B1} and the triangle inequality,
  \begin{align*}
    E &= \Big|f(x) - \sum_{k=0}^n {n \choose k} f\left({k \over n}\right)
    x^k (1-x)^{n-k} \Big|\\
    &= \Bigg| \sum_{k=0}^n {n \choose k} \left[ f(x) - f\left({k \over n}\right) \right] x^k(1-x)^{n-k}
    \Bigg|\\
    & \leq \sum_{k=0}^n {n \choose k} \Big| f(x) - f\left({k \over
    n}\right) \Big| \, x^k (1-x)^{n-k}.
  \end{align*}
  Splitting up the sum,
    \begin{align*}
        E \leq & \sum_{|x-{k\over n}| < \delta} {n \choose k} \Big| f(x) - f\left({k \over
        n}\right) \Big| \, x^k (1-x)^{n-k} \\ + & \sum_{|x-{k\over n}| \geq
        \delta} {n \choose k} \Big| f(x) - f\left({k \over
        n}\right) \Big| \, x^k (1-x)^{n-k}.
    \end{align*}
   Using \eqref{eq:boundM} and \eqref{eq:unifcont},
  \begin{align*}
    E & < {\epsilon \over 2} \sum_{|x-{k\over n}| < \delta} {n \choose k} x^k
    (1-x)^{n-k} + 2M \sum_{|x-{k\over n}| \geq
    \delta} {n \choose k} x^k (1-x)^{n-k} \\
    & \leq {\epsilon \over 2} \sum_{k=0}^n {n \choose k} x^k
    (1-x)^{n-k} + {2M \over \delta^2} \sum_{k=0}^n {n \choose k}\left(x-{k\over n}\right)^2 x^k
    (1-x)^{n-k},
  \end{align*}
  or
  \[
        E < {\epsilon \over 2} + {2M \over \delta^2} \left[ x^2
        B_n(x;1) - 2x B_n(x;x) + B_n(x;x^2) \right]
  \]
  Using Lemma \ref{Lemma:Bernstein},
  \begin{align*}
    E & < {\epsilon \over 2} + {2M \over \delta^2} \left(x^2 - 2x^2 +
    x^2 + {x(1-x) \over n} \right).
  \end{align*}
  Since the function $x(1-x)$ has a maximum value of $1/4$, we can
  write
  \[
    E < {\epsilon \over 2} + {2M \over 4n\delta^2} \leq \epsilon,
    \text{ for all } n \geq {M \over \epsilon \delta^2}.
  \]
  Thus, we see that the Bernstein polynomials do the job for large
  enough $n$, completing the proof. \qed
\end{proof}

Now that we know polynomials are a good tool for approximating
continuous functions, we ask how to find the best such
approximation. This is a result from linear algebra. For a
function $f$, we define the best approximation $p$ to be the
one that minimizes the norm of the error, i.e., for which $\|f -
p\|$ is smallest.

\begin{proposition}
  Let $f$ be a function, and let $\{\phi_k\}_{k=0}^\infty$ be an orthogonal set of polynomials with $\deg{\phi_n} = n$. The polynomial
  \begin{equation}
    p_n = \sum_{k=0}^n a_k \phi_k, \quad \text{where} \quad
    a_k = {\langle f, \phi_k \rangle \over \| \phi_k \|^2},
    \label{eq:approx}
  \end{equation}
  is the unique polynomial of degree $n$ that best approximates
  $f$, i.e., that minimizes $\|f-p_n\|$.
\end{proposition}

\begin{proof}
  Let's  choose a different polynomial $\tilde{p}_n$ of degree $n$.
  We will write $\tilde{p}_n$ as
  \[
    \tilde{p}_n = p_n + q_n, \quad \text{where} \quad
    q_n = \sum_{k=0}^n b_k\phi_k,
  \]
  for some constants $b_k$ not all zero. Then, we can write
  $\|f-\tilde{p}_n\|^2$ as
  \begin{equation*}
    \|f-p_n-q_n\|^2 = \|f-p_n\|^2 +
    \|q_n\|^2 -2\langle f-p_n, q_n \rangle.
  \end{equation*}
  But
  \[
    \langle f-p_n, q_n \rangle = \sum_{k=0}^n b_k \left( \langle f,\phi_k \rangle - \langle p_n, \phi_k \rangle
    \right)= 0,
  \]
  since each term in parentheses vanishes, as we can see from \eqref{eq:approx}. Thus, for any $\tilde p_n$,
  \[
    \|f-\tilde{p}_n\|^2 = \|f-p_n\|^2 + \|q_n\|^2 > \|f-p_n\|^2,
  \]
  the error is greater than that of $p_n$.
\qed
\end{proof}

This fact is really just a special case of a more general result
from the study of inner product spaces in linear algebra. We will
state the more general fact next. The proof is exactly the  same.

\begin{proposition}
  Let $f$ be a vector in an inner product space, and let $\{\phi_k\}_{k=0}^\infty$ be an orthogonal set of basis vectors. The
  vector
  \begin{equation*}
    p_n = \sum_{k=0}^n a_k \phi_k, \quad \text{where} \quad
    a_k = {\langle f, \phi_k \rangle \over \| \phi_k \|^2},
  \end{equation*}
  is the unique linear combination of the first $n+1$ basis vectors that best approximates
  $f$, i.e., that minimizes $\|f-p_n\|$.
  \label{fact:best}
\end{proposition}

%====================================================================
\section{Hilbert Space and Completeness}
%====================================================================

Later in the main discussion, we will be interested in
expanding an arbitrary function $f$ in an infinite series of the
spherical harmonic functions. In this section, we will address the question
of when such an expansion is possible. Since we are concerned
primarily with spherical harmonics in an arbitrary number of
dimensions, we will keep the discussion general, making no
reference here to the number of variables on which the functions
depend. In what follows we will thus let $x$ denote a vector in
$\mathbb{R}^p$, $\mathcal{D}$ denote a subset of $\mathbb{R}^p$,
and $d\Omega$ denote the differential volume element of
$\mathcal{D}$.

We will restrict our attention to the space of square-integrable
functions with domain $\mathcal{D}$ with respect to the weight
$w(x)$, i.e., those functions $f: \mathcal{D} \rightarrow
\mathbb{R}$ for which the integral
\[
    \underset{\mathcal{D}}{\int} f(x)^2 w(x) \,d\Omega
\]
exists and is finite. When endowed with the usual operations of function
addition and multiplication by real numbers and paired with the
inner product
\[
    \langle f, g \rangle = \underset{\mathcal{D}}{\int}
    f(x)g(x)w(x)\,d\Omega,
\]
the set of such functions becomes an inner product space.
Such a space has an induced norm
\[
    \|f\| = \sqrt{\langle f, f \rangle} .
\]

We will now work up to a definition of a Hilbert space.

\begin{definition}
  In a normed space, a sequence $\{x_n\}_{n=0}^\infty$ is called \emph{Cauchy} provided that for any
  $\epsilon > 0$ there exists an $N$ such that
  \[
    \|x_n-x_m\|< \epsilon \text{ for all } n,m \geq N.
  \]
\end{definition}

A Cauchy sequence is always convergent. However, it is possible that its limit does not belong to the same space.

\begin{definition}
  A normed space is \emph{complete} provided every Cauchy sequence converges to an element in the space.
\end{definition}

\begin{definition}
  A complete  inner product space is called a \emph{Hilbert space}.
\end{definition}

We will state the following fact without proof. The interested
reader may consult \cite{Real:Royden} for the proof.

\begin{theorem}
  The  inner product space of square-integrable functions
  defined above is a Hilbert space.
\end{theorem}

We can restate the question that opened this section as follows.
Given a Hilbert space $\mathcal{H}$ and an \emph{orthonormal} set
$\{\phi_n\}_{n=0}^\infty \subseteq \mathcal{H}$ (which can be
obtained from any orthogonal set by dividing each member by its
norm so $\|\phi_n\| = 1$), when can we write any arbitrary member $f \in \mathcal{H}$ as a linear combination\footnote{A linear combination is usually assumed to contain a finite number of terms. We will not need this restriction for our purposes. Here, linear combinations can be finite or infinite sums.} of the $\phi_n$?

\begin{definition}
  An orthonormal set $\{\phi_n\}_{n=0}^\infty \subseteq \mathcal{H}$ is called \emph{complete}
  provided that for each $f \in
  \mathcal{H}$, there exist scalars $c_1, c_2, \ldots$, such that
  \begin{equation}
    \lim_{n \rightarrow \infty} \|f - \sum_{k=0}^n c_k\phi_k \| =
    0.
    \label{eq:compl}
  \end{equation}
\end{definition}

 We know from Proposition \ref{fact:best} that out of all  linear combinations, the combination
\[
    p_n = \sum_{k=0}^n \langle f, \phi_k \rangle \phi_k
\]
minimizes $\|f - p_n\|$. Notice that we have now normalized the vectors $\phi_k$ to unity:  $\|\phi_k\|=1$.
Thus, if there exist scalars $c_k$ for which the norm in \eqref{eq:compl} converges to zero, then certainly
\begin{equation}
    \lim_{n \rightarrow \infty} \|f-\sum_{k=0}^n \langle f, \phi_k \rangle
    \phi_k \| = 0,
    \label{eq:limnorm}
\end{equation}
since
\[
    \|f-\sum_{k=0}^n \langle f, \phi_k \rangle
    \phi_k \| \leq \|f - \sum_{k=0}^n c_k\phi_k \|, \text{ for every } n.
\]
Let us rewrite \eqref{eq:limnorm} by computing
\[
    \left\| f-\sum_{k=0}^n \langle f_n, \phi_k \rangle  \phi_k \right\|^2,
\]
which is the same as
\[
    \|f\|^2 -2 \Big \langle f, \sum_{k=0}^n \langle f, \phi_k \rangle \phi_k \Big \rangle +
    \Big \langle \sum_{k=0}^n  \langle f, \phi_k \rangle \phi_k , \sum_{\ell=0}^n \langle f, \phi_\ell \rangle \phi_\ell \Big \rangle.
\]
Using the linearity of the inner product, this becomes
\[
    \|f\|^2 -2 \sum_{k=0}^n \langle f, \phi_k \rangle ^2 +
    \sum_{k, \ell = 0}^n \langle f, \phi_k \rangle \langle f, \phi_\ell
    \rangle \langle \phi_k, \phi_\ell \rangle.
\]
By the orthonormality of the $\phi_n$, we have
\begin{equation}
      \left\|f-\sum_{k=0}^n \langle f_n, \phi_k \rangle
     \phi_k \right\|^2  = \|f\|^2 - \sum_{k=0}^n \langle f, \phi_k \rangle ^2.
\label{eq:preParseval}
\end{equation}

We note that since the norm-squared is non-negative, then
\[
    \|f\|^2 - \sum_{k=0}^n \langle f,\phi_k \rangle ^2 \geq 0,
\]
or,
\begin{equation}
    \sum_{k=0}^n \langle f,\phi_k \rangle ^2 \leq \|f\|^2, \text{ for every } n.
    \label{eq:BesIn}
\end{equation}
This is known as \emph{Bessel's inequality} and tells us that the sum $\sum_{k=0}^\infty \langle f, \phi_k \rangle ^2$
converges.  If the $\phi_k$ form a complete set, $\|f-\sum_{k=0}^n \langle
f_n, \phi_k \rangle \phi_k \|^2$ must converge to zero as $n
\rightarrow \infty$. In this case, from equation \eqref{eq:preParseval}, it follows that Bessel's inequality becomes an equality:
\[
    \sum_{k=0}^\infty \langle f,\phi_k \rangle ^2 = \|f\|^2,
\]
known as \emph{Parseval's equality}. We have thus arrived at the following conclusion.

\begin{proposition}
An orthonormal set $\{\phi_n\}_{n=0}^\infty \subseteq \mathcal{H}$
is complete if and only if Parseval's equality holds for each
$f \in \mathcal{H}$. \label{fact:pars}
\end{proposition}

We give one more definition.

\begin{definition}
  A set $\{\phi_n\}_{n=0}^\infty$ is \emph{closed}
  provided
  \[
    \langle f, \phi_n \rangle = 0 \text{ for every } n \qquad \text{implies} \qquad f=0.
  \]
\end{definition}

Now we are ready to prove the main result of this section.

\begin{theorem}
\label{fact:Hilb}
    The orthonormal set $\{\phi_n\}_{n=0}^\infty \subseteq \mathcal{H}$ is complete if and only if it is closed.
\end{theorem}

\begin{proof}(Sufficient)
  Let $\{\phi_n\}_{n=0}^\infty \subset \mathcal{H}$ constitute a closed orthonormal set
  and $f \in \mathcal{H}$. We will show that
  \[
    \lim_{n \rightarrow \infty} \left( \|f\|^2 - \sum_{k=0}^n \langle f, \phi_k \rangle^2
    \right) = 0,
  \]
  since this guarantees completeness by Proposition \ref{fact:pars}.

  Define $g_n = f - \sum_{k=0}^n \langle f, \phi_k \rangle
  \phi_k$, and notice that $\{g_n\}_{n=0}^\infty$ is a Cauchy
  sequence since, if $n>m$,
  \begin{align*}
    \|g_n-g_m\|^2 &= \Big \| \sum_{k=m+1}^n \langle f , \phi_k
    \rangle \phi_k \Big \|^2\\ &= \sum_{k, \ell=m+1}^n \langle f,
    \phi_k \rangle \langle f, \phi_\ell \rangle \langle \phi_k,
    \phi_\ell \rangle\\ &= \sum_{k=m+1}^n \langle f, \phi_k
    \rangle^2,
  \end{align*}
  which can be made arbitrarily small for large enough $n,m$ since
  the series $\sum_{k=0}^\infty \langle f, \phi_k
  \rangle^2$ converges, by \eqref{eq:BesIn}. Since
  $\mathcal{H}$ is complete by definition, this Cauchy sequence
  must converge to some $g \in \mathcal{H}$, i.e.,
  \begin{equation}
    \lim_{n \rightarrow \infty} \|g_n - g \| = 0.
    \label{eq:cauchyseq}
  \end{equation}
  Now fix $j$ and take $n \in \mathbb{N}_0$ such that $n>j$.
  By the Cauchy-Schwartz inequality
  \[
    |\langle g, \phi_j \rangle| = | \langle g_n - g, \phi_j \rangle
    | \leq \|g_n-g\|\|\phi_j\| = \|g_n-g\|.
  \]
  Since this holds for any $n > j$, we can use
  \eqref{eq:cauchyseq} to conclude
  \[
    |\langle g,\phi_j\rangle | \leq
    \lim_{n \rightarrow \infty} \|g_n-g\| = 0,
  \]
  which implies that for arbitrary $j$,
  $\langle g, \phi_j \rangle = 0$. Since we assumed the $\phi_j$
  constitute a closed set, we have $g = 0$. Therefore,
  \[
    \lim_{n \rightarrow \infty} \left( \|f\|^2 - \sum_{k=0}^n \langle f, \phi_k \rangle^2
    \right)= \lim_{n \rightarrow \infty} \|g_n \|^2 = \lim_{n \rightarrow \infty} \|g_n-g
    \|^2= 0.
  \]

  (Necessary) Suppose $\{ \phi_n \}_{n=0}^\infty$ is complete but not closed.
  Then there exists a function $f \neq 0$ such that $\langle f, \phi_n \rangle = 0$ for every $n$. Then
  \[
    \lim_{n\to\infty} \left( \|f\|^2 - \sum_{k=0}^n \langle f,\phi_k \rangle^2 \right) =
    \|f\|^2 \neq 0.
  \]
  So, by Proposition \ref{fact:pars}, $\{\phi_n\}_{n=0}^\infty$ is not complete --- a contradiction. This proves that
  if an orthonormal set in $\mathcal{H}$ is closed, then it is complete.
\qed
\end{proof}

In Chapter \ref{ch:sph}, we will use this result to show that the
set of spherical harmonics form a complete set in the space of square-integrable functions
by showing that it is closed.

 \CleanPage

%%%%%%%%%%%%%%%%%%%%%%%%%%%%%%%%%%%%%%%%%%%%%%%%%%%%%%%%%%%%%%%%%%%%%%%%%%%%%%%
\chapter{Spherical Harmonics in $p$ Dimensions} \label{ch:sph}
%%%%%%%%%%%%%%%%%%%%%%%%%%%%%%%%%%%%%%%%%%%%%%%%%%%%%%%%%%%%%%%%%%%%%%%%%%%%%%%

We will begin by developing some facts about a special kind of
polynomials. We will then define a spherical harmonic to be one of
these polynomials with a restricted domain, as hinted at in
Section \ref{sec:SepVar}. After discussing some properties of
spherical harmonics, we will introduce the Legendre polynomials.
And once we have produced a considerable number of results, we will move on
to an application of the material developed  to
boundary value problems.

%====================================================================
\section{Harmonic Homogeneous Polynomials} \label{sec:harmhom}
%====================================================================

\begin{definition}
A polynomial $H_n(x_1, x_2, \ldots, x_p)$ is \emph{homogeneous of degree $n$} 
in the $p$ variables $x_1, x_2, \ldots, x_p$ provided
 \[
    H_n(tx_1, tx_2, \ldots, tx_p) = t^nH_n(x_1, x_2, \ldots,
    x_p)~.
 \]
\end{definition}

In the definition of a homogeneous polynomial, let's set $u_i= t x_i$, for all $i$ and differentiate the defining equation with respect to $t$:
$$
    \sum_{i=1}^p {\partial H_n(u_1,u_2,\dots,u_p)\over \partial u_i} \, {du_i\over dt}  =   n \, t^{n-1} \, H_n(x_1,x_2,\dots,x_p) ,
$$
or
$$
    \sum_{i=1}^p {\partial H_n(u_1,u_2,\dots,u_p)\over \partial u_i} \,   x_i  =   n \, t^{n-1} \, H_n(x_1,x_2,\dots,x_p) .
$$
Finally, we set $t=1$ to find the following functional equation satisfied by the homogeneous polynomial,
\begin{equation}
    \sum_{i=1}^p {\partial H_n(x_1,x_2,\dots,x_p)\over \partial x_i} \, x_i  =   n \,  H_n(x_1,x_2,\dots,x_p)~,
\label{eq:Euler}
\end{equation}
known as \emph{Euler's equation}.

The following calculation will be useful in the counting of linearly independent homogeneous polynomials.

\begin{lemma}
For $0 < |r| < 1$,
    \begin{equation}
     {1\over (1-r)^p} = \sum_{j=0}^\infty {(p+j-1)! \over  j!(p-1)!} \, r^j \, .
    \end{equation}
    \label{lem:count}
\end{lemma}

\begin{proof}
We will use a counting trick to prove this. Since $0 < |r| < 1$, we can write $(1-r)^{-p}$ in terms of a product of geometric series, i.e.,
\begin{align*}
    (1-r)^{-p}  &= \underset{p \text{ times}}{\underbrace{\left({1 \over 1-r}\right) \left({1 \over 1-r}\right) \cdots \left({1 \over 1-r}\right)}} \\
                     &= \underset{p \text{ times}}{\underbrace{\left( \sum_{n=0}^\infty r^n \right)
                           \left( \sum_{n=0}^\infty r^n \right) \cdots \left( \sum_{n=0}^\infty r^n \right)}}.
\end{align*}
We will compute the \emph{Cauchy product} of the $p$ infinite series. The
result will be an infinite series including a constant term and
all positive integer powers of $r$,
\[
    (1-r)^{-p} = \sum_{j=0}^\infty c_j r^j .
\]
It remains to compute the coefficients $c_j$. To determine each $c_j$, we must compute how many $r^j$'s are
produced in the Cauchy product, i.e., in how many different ways we produce an $r^j$ in the multiplication.

Note that this computation is equivalent to asking, ``In how many
different ways can we place $j$ indistinguishable balls into $p$
boxes?'' The $p$ boxes correspond to the $p$ series in the
product, and choosing to place $k<j$ balls into a certain box
corresponds to choosing the $r^k$ term in that series when
computing a term of the Cauchy product.

Let us use a diagram to assist in this calculation. We will use a
vertical line to denote a division between two boxes and a dot to
denote a ball. Each configuration of the $j$ balls in $p$
boxes can thus be represented by a string of $j$ dots and $p-1$
lines (since $p$ boxes require only $p-1$ divisions). For example,
\[
\bullet \bullet | \bullet | \bullet | \ | \bullet \bullet \, |
\bullet | \ | \bullet
\]
represents one configuration of $j=8$ balls in $p=8$ boxes.

Now, the number of
ways to arrange $j$ indistinguishable balls in $p$ boxes is the
same as the number of distinct arrangements of $j$ dots and $p-1$
lines. This is given by ``$p-1+j$ choose $j$,'' i.e.,
\[
    c_j = {p-1+j \choose j} = {(p-1+j)! \over j!(p-1)!},
\]
and the lemma is proved. \qed
\end{proof}

\begin{proposition}
If $K(p,n)$ denotes the number of linearly independent homogeneous
polynomials of degree $n$ in $p$ variables, then
\[
    K(p,n) = {(p+n-1)! \over n!(p-1)!}~.
\]
\label{thm:K}
\end{proposition}

We will give two proofs of this claim. The first uses a recursive relation obeyed by $K(p,n)$, while the second employs another counting trick.

\begin{proof}[Proof 1]
Let $H_n(x_1, x_2, \ldots, x_p)$ be a
homogeneous polynomial of degree $n$ in its $p$ variables. Notice
that $H_n$ is a polynomial in $x_p$ of degree at most $n$. For if
$H_n$ contained a power of $x_p$ greater than $n$, the polynomial
could not be homogenous of degree $n$, since $H_n(\ldots, tx_p)$
would contain a power of $t$ greater than $n$. Thus we can write,
\begin{equation}
    H_n(x_1, x_2, \ldots, x_p) = \sum_{j=0}^n x_p^j \, h_{n-j}(x_1, x_2,
    \ldots, x_{p-1}),
    \label{eq:hom}
\end{equation}
where the $h_{n-j}$ are polynomials. Moreover, notice that the
$h_{n-j}$ must be homogeneous of degree $n-j$ in their $p-1$
variables. Indeed, using the homogeneity of $H_n$,
 \[
    t^n \sum_{j=0}^n x_p^j h_{n-j}(x_1, \ldots, x_{p-1}) = H_n(tx_1,
    \ldots, tx_p) = \sum_{j=0}^n (tx_p)^j h_{n-j}(tx_1, \ldots,
    tx_{p-1}),
 \]
which implies that
 \[
    \sum_{j=0}^n x_p^j \left[ t^{n-j}h_{n-j}(x_1, \ldots,
    x_{p-1}) - h_{n-j}(tx_1, \ldots, tx_{p-1}) \right] = 0.
 \]

The expression in brackets must vanish by the linear
independence of the $x_p^j$. Each $h_{n-j}$ can be written in
terms of a basis of $K(p-1, n-j)$ homogeneous polynomials of
degree $n-j$ in $p-1$ variables, and thus $H_n$ can be written in
terms of a basis of
 \[
    K(p,n) = \sum_{j=0}^n K(p-1, n-j) = \sum_{j=0}^n K(p-1, j)
 \]
linearly independent elements. We have found a recursive relation
that the $K(p,n)$ must satisfy. Now for some $0 < |r| < 1$, let
\begin{equation}
    G(p) = \sum_{n=0}^\infty r^nK(p,n).
    \label{eq:Gen}
\end{equation}
Then,
\[
    G(p) = \sum_{n=0}^\infty r^n \sum_{j=0}^n K(p-1, j) = \sum_{j=0}^\infty K(p-1, j)
    \sum_{n=j}^\infty r^n,
\]
or
\begin{align*}
    G(p) =& \sum_{j=0}^\infty K(p-1, j) r^j
    \sum_{n=0}^\infty r^n = \, {1 \over 1-r} \, \sum_{j=0}^\infty r^j K(p-1,j),
\end{align*}
so
\[
    G(p) = {G(p-1) \over 1-r}.
\]
Using an inductive argument, we can show that
\[
    G(p) = {G(1) \over (1-r)^{p-1}}\, .
\]
By noticing that $K(1,n) = 1$, since every homogeneous polynomial
of degree $n$ in one variable can be written as $cx_1^n$, we see
that
\[
    G(1) = \sum_{j=0}^\infty r^n = {1 \over 1-r}\, .
\]
Thus,
\[
    G(p) = (1-r)^{-p} = \sum_{n=0}^\infty {(p+n-1)! \over
    n!(p-1)!} \, r^n \, ,
\]
by the above lemma. Comparing this to \eqref{eq:Gen} then proves the theorem. \qed
\end{proof}

\begin{proof}[Proof 2]
  Using reasoning similar to that used in the above proof, we see
  that every \emph{monomial} (i.e., product of variables) in an $n$-th degree homogeneous polynomial must
  be of degree $n$. For example, a fourth degree homogeneous polynomial in the
  variables $x,y$ can have an $x^2y^2$ but not an $x^2y$ term.

  To uniquely determine such a polynomial, we must give
  the coefficient of every possible $n$-th degree monomial. So to find
  $K(p,n)$, we must find out how many possible $n$-th degree monomials
  there are in $p$ variables. But this is equivalent to asking, ``In how
  many ways can we place $n$ indistinguishable balls into $p$ boxes?'' The $p$ boxes
  represent the $p$ variables from which we can choose, and
  placing $j<n$ balls into a certain box corresponds to choosing to
  raise that variable to the $j$ power. For example, using the same notation as
  in the proof of Lemma \ref{lem:count}, the string
  \[
    \bullet \bullet | \bullet | \bullet | \
  \]
  could represent the $w^2x^1y^1z^0$ term in a fourth-degree homogeneous polynomial. As in the above lemma, the number
  of such arrangements is
  \[
    K(p,n) = {n + p -1 \choose n} = {(n+p-1)! \over n!(p-1)!},
  \]
  proving the theorem. \qed
\end{proof}

\begin{definition}
A polynomial $q(x_1, x_2, \ldots, x_p)$ is \emph{harmonic}
provided
\begin{equation}
    \Delta_p \, q = 0.
    \label{eq:lapl}
\end{equation}
Equation \eqref{eq:lapl} is called the \emph{Laplace equation}.
\end{definition}

The following property of combinations will be used to prove the  theorem that follows it.

\begin{lemma}
  If $k, \ell \in \mathbb{N}$, then
  \[
    {k \choose \ell} = {k \over \ell} {k-1 \choose \ell -1}.
  \]
\end{lemma}

\begin{proof}
Just compute:
\[
  {k \over \ell} {k-1 \choose \ell -1} = {k \over \ell} \cdot
  {(k-1)! \over (\ell -1)! (k - \ell)!} = {k! \over \ell !
  (k-\ell)!} = {k \choose \ell},
\]
as required. \qed
\end{proof}

\begin{theorem}
If $N(p,n)$ denotes the number of linearly independent homogeneous
harmonic polynomials of degree $n$ in $p$ variables, then
\[
    N(p,n) = {2n+p-2 \over n} {n+p-3 \choose n-1}.
\]
\label{thm:N}
\end{theorem}

\begin{proof}
Let $H_n$ be a homogeneous harmonic polynomial of degree $n$ in
$p$ variables. As in \eqref{eq:hom}, we write
\[
    H_n(x_1, x_2, \ldots, x_p) = \sum_{j=0}^n x_p^j \, h_{n-j}(x_1, x_2,
    \ldots, x_{p-1}),
\]
and we operate on $H_n$ with the Laplace operator. Since $H_n$ is
harmonic,
\begin{align*}
    0 &= \Delta_p H_n = \left( {\partial^2 \over \partial x_p^2} +
    \Delta_{p-1} \right) H_n = \sum_{j=2}^n j(j-1)x_p^{j-2} \, h_{n-j}
    + \sum_{j=0}^n x_p^j \Delta_{p-1} h_{n-j},
\end{align*}
so
\[
    0 = \sum_{j=0}^n x_p^j \left[ (j+2)(j+1)h_{n-j-2} + \Delta_{p-1}h_{n-j} \right].
\]
where we define $h_{-1} = h_{-2} = 0$. Since the $x_p^j$ are
linearly independent, each of the coefficients must vanish,
\begin{align}
  2h_{n-2} + \Delta_{p-1} h_n =& \:0, \notag \\
  6h_{n-3} + \Delta_{p-1} h_{n-1} =& \:0, \notag \\
  \vdots \ & \label{eq:recursion} \\
  n(n-1)h_0 + \Delta_{p-1}h_2 =& \:0, \notag \\
  \Delta_{p-1}h_1 =& \:0, \notag \\
  \Delta_{p-1}h_0 =& \:0. \notag
\end{align}
We have found a recursive relationship that the $h_{n-j}$ must
obey. Thus, choosing $h_n$ and $h_{n-1}$ determines the rest of
the $h_{n-j}$. By Theorem \ref{thm:K}, $h_n$ can be written in
terms of $K(p-1, n)$ basis polynomials and $h_{n-1}$ can be
written in terms of $K(p-1,n-1)$ basis polynomials. Thus we must
give $K(p-1,n)+K(p-1,n-1)$ coefficients to determine $h_n,
h_{n-1}$ which thereby determine $H_n$. Therefore,
\[
    N(p,n) = K(p-1,n)+K(p-1,n-1) = {p+n-2 \choose n} + {p+n-3
    \choose n-1}.
\]
By the above lemma, we can write this as
\[
    N(p,n) = {p+n-2 \over n} {p+n-3 \choose n-1} + {p+n-3 \choose
    n-1} = {2n + p - 2 \over n} {n+p-3 \choose n-1},
\]
and thus the theorem is proved. \qed
\end{proof}

In what follows, $x$ will always denote the vector $(x_1, x_2,
\ldots, x_p)$, $r$ or $|x|$ will denote its norm $\sqrt{x_1^2 + x_2^2
+ \cdots + x_p^2}\,$, and $\xi$ will denote the vector $(\xi_1,
\xi_2, \ldots, \xi_p)$ having unit norm. Keep in mind that $x,\xi$
represent vectors while the $x_j,\xi_j$ denote their components. Also,
$H_n$ will always denote a harmonic homogeneous polynomial.
Lastly, in the remainder of this chapter, we will often suppress
explicit reference to the number of dimensions in which we work. It
should be assumed that we deal with $p$ dimensions unless stated
otherwise.

%====================================================================
\section{Spherical Harmonics and Orthogonality}
\label{sec:sphharm}
%====================================================================

We are now ready to introduce the spherical harmonics. First,
notice that
\begin{equation}
    H_n(x) = H_n(r\xi) = r^n H_n(\xi).
    \label{eq:sepvar}
\end{equation}
In $p$-dimensional spherical coordinates,  the radial dependence and the angular dependence of the functions
$H_n$ can be separated.

\begin{definition}
  A \emph{spherical harmonic of degree $n$}, denoted $Y_n(\xi)$, is a harmonic homogeneous polynomial
  of degree $n$ in $p$ variables restricted to the unit
  $(p-1)$-sphere.
In other words, $Y_n$ is the map
  \[
    Y_n: S^{p-1} \rightarrow \mathbb{R}, \text{ given by }
    Y_n(\xi) = H_n(\xi) \text{ for every } \xi \in S^{p-1}
  \]
  for some harmonic homogeneous polynomial $H_n$. We can write
  $Y_n = H_n|_{S^{p-1}}$.
\end{definition}

In Chapter \ref{ch:intro}, we introduced functions $Y$ (which we claimed were spherical harmonics) that turned out to be
eigenfunctions of the angular part of the Laplace operator. Using Proposition \ref{prop:AngLap}, we can show that spherical harmonics are
indeed eigenfunctions of $\Delta_{S^{p-1}}$.

\begin{proposition}
  \begin{equation}
   \Delta_{S^{p-1}} Y_n = n(2-p-n) Y_n.
   \label{eq:Eigen}
  \end{equation}
\end{proposition}

\begin{proof}
  Let $H_n$ be a harmonic homogeneous polynomial and $Y_n$ its
  associated spherical harmonic. As in \eqref{eq:sepvar}, we have
  $H_n(x) = r^nY_n(\xi)$. Then, using Proposition
  \ref{prop:AngLap},
  \[
    0 = \Delta_p \left( r^nY_n \right) = n(n-1)r^{n-2}Y_n + {p-1
    \over r}\,nr^{n-1}Y_n + {1\over r^2} \, r^n \Delta_{S^{p-1}}
    Y_n.
  \]
  Rearranging,
  \[
    r^{n-2} \left[ \Delta_{S^{p-1}} Y_n + n(n+p-2)Y_n \right] = 0,
  \]
  which implies
  \[
    \Delta_{S^{p-1}}Y_n = n(2-p-n)Y_n,
  \]
  as sought.
\qed
\end{proof}

\begin{remark}
In three dimensions \eqref{eq:Eigen} becomes
\[
    \left[ {1
    \over \sin{\theta}} {\partial \over \partial \theta} \left(
    \sin{\theta} {\partial  \over \partial \theta} \right) + {1 \over
    \sin^2{\theta}} {\partial^2  \over \partial \phi^2} \right]Y_\ell =
    -\ell(\ell+1) Y_\ell.
\]
As promised in Subsection \ref{subs:3dang}, this allows us to show that three-dimensional spherical  harmonics carry
a definite amount of quantum mechanical angular momentum. Referring to \eqref{eq:L2},
\[
    \hat{\vec{L}}^2\,  Y_\ell = \hbar^2 \ell(\ell + 1) Y_\ell.
\]
\end{remark}

\begin{remark}
  Since the spherical harmonic $Y_n(\xi)$ is defined as the restriction of some
  $H_n(x)$ to the unit sphere, $Y_n(\xi)$ must also be
  homogeneous. However, $Y_n(t\xi)$ is not always defined since
  the domain of the spherical harmonic is $S^{p-1}$. In fact,
  it is only defined for $|t\xi|=1$, and since
  $|\xi|=1$, we see that we can only write $Y_n(t\xi)$ for $t = \pm 1$.
  The case $t=1$ is trivial, but the case $t=-1$ tells us the parity of $Y_n(\xi)$.
  Since $Y_n(-\xi) = (-1)^n Y_n(\xi)$, we see that the
  transformation $\xi \longmapsto -\xi$ sends $Y_n
  \longmapsto -Y_n$ if $n$ is odd and leaves $Y_n$ invariant if $n$ is even.
\end{remark}

We now come to the main result of this section, which we hinted at
in Section \ref{sec:SepVar}.

\begin{theorem}
  Let $Y_n(\xi), Y_m(\xi)$ be two spherical harmonics. Then
  \[
    \underset{S^{p-1}}{\int} Y_n(\xi)Y_m(\xi) \, d\Omega_{p-1} = 0, \quad \text{if } n \neq m.
  \]
  That is, spherical harmonics of different degrees are orthogonal
  over the sphere.
  \label{thm:SHorth}
\end{theorem}

\begin{proof}
  Let us perform the computation using the harmonic homogeneous polynomials $H_n$ and $H_m$ where $Y_n = H_n |_{S^{p-1}}$ and
  $Y_m = H_m |_{S^{p-1}}$.

Now, we start with the divergence  theorem in $p$ dimensions for a vector $A$,
\begin{align*}
         \underset{B^p}{\int}  \nabla_p\cdot  A(x)  \, d^px  = \underset{S^{p-1}}{\int}   A(\xi) \cdot \xi  \, d\Omega_{p-1} ,
\end{align*}
apply it twice for  the vectors $H_n\nabla H_m$ and $H_m\nabla H_n$, and subtract the results\footnote{Incidentally, we point out that the
identity thus obtained for any two functions $f, g$
\[
       \underset{B^p}{\int}  (f \Delta_p g - g \Delta_p f) \, d^px
    = \underset{S^{p-1}}{\int} (f \nabla_p g - g \nabla_p f) \cdot \xi \, d\Omega_{p-1},
\]
\label{GreensTheorem}
is known as \emph{Green's theorem}.
We may also observe that $\nabla_p f \cdot \xi$ is the directional derivative of $f$ along $\xi$ and write
\[
       \underset{B^p}{\int}  (f \Delta_p g - g \Delta_p f) \, d^px
    = \underset{S^{p-1}}{\int} \left(f {\partial g\over\partial\xi} - g {\partial f\over\partial\xi}\right)  \, d\Omega_{p-1}.
\]}:
  \begin{align*}
      \underset{B^p}{\int} \nabla_p\, \cdot &\left[ H_n(x) \nabla_p H_m(x) - H_m(x) \nabla_p H_n(x) \right] \, d^px 
      \\
      &= \underset{S^{p-1}}{\int} \left[ H_n(\xi) \nabla_p H_m(\xi) - H_m(\xi) \nabla_p H_n(\xi) \right]\cdot \xi \, d\Omega_{p-1},
  \end{align*}
or
\begin{equation}
      0 = \underset{S^{p-1}}{\int} \left[ H_n(\xi) \nabla_p H_m(\xi) - H_m(\xi) \nabla_p H_n(\xi) \right]\cdot \xi \, d\Omega_{p-1},
\label{eq:Orthogonal1}
\end{equation}
where we used the property
  \[
    \nabla_p \cdot \left( H_n \nabla_p H_m \right) = \nabla_p H_n \cdot \nabla_p H_m + H_n \Delta_p H_m ,
  \]
  and the fact $\Delta_p H_m=\Delta_p H_n=0$.

   Euler's equation \eqref{eq:Euler} for a homogeneous polynomial,
  \[
    \sum_{j=1}^p {\partial H_n(\xi) \over \partial \xi_j} \,\xi_j = n \, H_n(\xi),
  \]
  can be written in the form
\[
        \nabla_p H_n(\xi) \cdot \xi = n\,H_n(\xi) .
\]
 With the help of this result, equation \eqref{eq:Orthogonal1} takes the form
  \[
    (m-n) \underset{S^{p-1}}{\int} H_n(\xi)H_m(\xi) \, d\Omega_{p-1} = 0.
  \]
  But the integral is carried out over $S^{p-1}$ where $Y_n = H_n$ and $Y_m = H_m$. This equation is thus equivalent to
  \[
    (m-n) \underset{S^{p-1}}{\int} Y_n(\xi)Y_m(\xi) \, d\Omega_{p-1} = 0.
  \]
  By hypothesis, $n \neq m$; therefore $Y_n, Y_m$ are orthogonal  over the sphere.
\qed
\end{proof}

Given a set of $N(p,n)$ linearly independent spherical harmonics
of degree $n$, we can use the Gram-Schmidt orthonormalization
procedure to produce an orthonormal set of spherical harmonics,
i.e., a set
\begin{equation}
    \{ Y_{n,i}(\xi) \}_{i=1}^{N(p,n)} \quad \text{with} \quad \underset{S^{p-1}}{\int} Y_{n,i}(\xi) Y_{n,j}(\xi) \, d\Omega_{p-1} =
    \delta_{ij},
    \label{eq:orth}
\end{equation}
where
\[
    \delta_{ij} = \left\{
        \begin{array}{lr}
            1 & \text{if } i=j,\\
            0 & \text{if } i \neq j,
        \end{array}
    \right.
\]
is the Kronecker delta. For the remainder of this chapter, unless
indicated otherwise, we will let $Y_{n,i}(\xi)$ denote an $n$-th
degree spherical harmonic belonging to an orthonormal set of
$N(p,n)$ such functions, as in \eqref{eq:orth}.

In what follows, let $R$ be an orthogonal matrix that acts on
$\xi$ as a rotation of coordinates. Notice that since the
integration is taken over the entire sphere in \eqref{eq:orth},
the orthonormal set of spherical harmonics remains orthonormal in
a rotated coordinate frame. That is,
\begin{equation}
    \underset{S^{p-1}}{\int} Y_{n,i}(R\xi) Y_{n,j}(R\xi) \, d\Omega_{p-1} =
    \delta_{ij}.
    \label{eq:rot}
\end{equation}

\begin{proposition}
  If $Y_n(\xi)$ is a spherical harmonic of degree $n$, then $Y'_n(\xi) = Y_n(R\xi)$ is also
  a spherical harmonic of degree $n$, for any rotation matrix $R$.
\end{proposition}

\begin{proof}
  Let $Y_n(\xi)$ be a spherical harmonic of degree $n$. Then there exists a harmonic homogeneous polynomial $H_n(x)$ of degree $n$
  such that $Y_n = H_n | _{S^{p-1}}$. Denote $H'_n(x) = H_n(Rx)$. We claim  that $Y'_n = H'_n|_{S^{p-1}}$. To see this, first notice that
  $H'_n(x)$ is a polynomial in $x_1, x_2, \ldots, x_p$. Indeed, $H'_n(x) = H_n(Rx)$ is a linear combination of powers of the
  $\sum_{j=1}^p R_{ij} x_j$ and thus a linear combination of powers of the $x_j$. Next, notice that $H'_n$ is homogeneous of degree $n$,
  \[
    H'_n(tx) = H_n(tRx) = t^nH_n(Rx) = t^nH'_n(x).
  \]
  Finally, notice that $H'_n$ is harmonic, by Proposition \ref{prop:invlap}.
  Restricting $H'_n$ to the unit sphere thus gives a  spherical harmonic of degree $n$, $Y'_n(\xi)$.
\qed
\end{proof}

Since the set $\{ Y_{n,i}(\xi) \}_{i=1}^{N(p,n)}$ in \eqref{eq:orth} is a maximal linearly independent set of
spherical harmonics of degree $n$, it serves as a basis for all
such functions. We have just shown that $Y_{n,j}(R\xi)$ is a spherical
harmonic of degree $n$ provided $Y_{n,j}(\xi)$ is as well. Thus, we
can write $Y_{n,j}(R\xi)$ in terms of the basis functions:
\[
    Y_{n,j}(R\xi) = \sum_{\ell=1}^{N(p,n)} C_{\ell j} Y_{n, \ell}(\xi).
\]
Using this expression to rewrite the integral in \eqref{eq:rot}, we can show that the matrix $C$ defined in the above expression is
orthogonal. Indeed,
\begin{align*}
    \delta_{ij} &= \underset{\xi \in S^{p-1}}{\int} \left( \sum_{k=1}^{N(p,n)} C_{ki} Y_{n,k}(\xi) \right)
                          \left( \sum_{\ell=1}^{N(p,n)} C_{\ell j} Y_{n,\ell}(\xi) \right) \, d\Omega_{p-1} \\
                    &= \sum_{k, \ell = 1}^{N(p,n)} C_{ki}C_{\ell  j} \underset{\xi \in S^{p-1}}{\int} Y_{n,k}(\xi)Y_{n, \ell}(\xi) \, d\Omega_{p-1}\\
                    &= \sum_{k, \ell = 1}^{N(p,n)} C_{ki}C_{\ell  j} \delta_{k\ell} =  \sum_{k=1}^{N(p,n)} C^t_{ik}C_{kj}.
\end{align*}

For the following discussion, let $\xi, \eta$ be two unit vectors.
Let us consider the function given by
\begin{equation}
    F_n(\xi, \eta) = \sum_{j=1}^{N(p,n)} Y_{n,j}(\xi) Y_{n,j}(\eta).
    \label{eq:F}
\end{equation}

\begin{lemma}
  The function $F_n$ defined above is invariant under a rotation of
  coordinates.
  \label{prop:Finv}
\end{lemma}

\begin{proof}
  Let $R$ be a rotation matrix. Then, using the orthogonal matrix
  $C$ discussed above,
\begin{align*}
    F_n(R\xi,R\eta) &= \sum_{j=1}^{N(p,n)} Y_{n,j}(R\xi)
    Y_{n,j}(R\eta)\\
    &= \sum_{j=1}^{N(p,n)} \left( \sum_{\ell=1}^{N(p,n)} C_{\ell j}Y_{n,\ell}(\xi) \right) \left( \sum_{m=1}^{N(p,n)} C_{mj}Y_{n,m}(\eta) \right)
\end{align*}
  so
\begin{align*}
    F_n(R\xi,R\eta) &= \sum_{\ell, m = 1}^{N(p,n)} Y_{n,\ell}(\xi) Y_{n,m}(\eta)
    \left( \sum_{j=1}^{N(p,n)} C_{\ell j} C^t_{jm} \right)\\ &=
    \sum_{\ell=1}^{N(p,n)} Y_{n,\ell}(\xi) Y_{n,\ell}(\eta) =
    F_n(\xi, \eta),
\end{align*}
  as claimed.
\qed
\end{proof}

Since the dot product $\langle \xi , \eta \rangle$ is also invariant under the rotation $R$,
this suggests that $F_n(\xi, \eta)$ could be a function of
$\langle \xi, \eta \rangle$ alone. In fact, this is the case, as
shown in the following lemma.

\begin{lemma}
  If $F_n$ is defined as in \eqref{eq:F}, then $F_n(\xi, \eta) = p(\langle \xi, \eta
  \rangle)$ where $p(t)$ is a polynomial.
  \label{lem:Finn}
\end{lemma}

\begin{proof}
    A rotation of coordinates leaves $\langle \xi, \eta \rangle$ invariant and, by the above lemma, does not change $F_n(\xi,\eta)$ either.
    For some $-1 \leq t \leq 1$, there exists a
    rotation $R$ that sends
    \begin{align*}
        \xi   &\overset{R}{\longmapsto} \xi' = (t, \sqrt{1-t^2}, 0, \ldots, 0) , \\
        \eta &\overset{R}{\longmapsto}\eta'= (1, 0, \ldots, 0) .
    \end{align*}
    To see this, just rotate coordinates so that $\eta$ points along the $x_1$-axis. Then rotate coordinates around the $x_1$-axis until the
    component of $\xi$ orthogonal to $\eta$ points along the $x_2$-axis.

    Notice that $\langle \xi, \eta \rangle = \langle \xi', \eta' \rangle = t$. Since $F_n(\xi, \eta)$ is a sum of products of spherical
    harmonics, which are each polynomials in the components of  $\xi$ or $\eta$, $F_n$ is a polynomial, say $p$, in the
    components of its arguments, i.e.,
    \begin{equation}
        F_n(\xi, \eta) = F_n(\xi', \eta') = p(t, \sqrt{1-t^2}).
        \label{eq:poly1}
    \end{equation}
    We can impose another rotation of coordinates, again without
    changing $F_n$ or $\langle \xi, \eta \rangle$. Let $\tilde{R}$ be the
    transformation that rotates vectors by $\pi$ radians about the
    $x_1$-axis, i.e.,
    \begin{align*}
        \xi'    &\overset{\tilde{R}}{\longmapsto} \xi'' = (t, -\sqrt{1-t^2}, 0, \ldots, 0) , \\
         \eta' &\overset{\tilde{R}}{\longmapsto} \eta''= (1, 0, \ldots, 0).
    \end{align*}
    Just as in \eqref{eq:poly1}, we conclude that
    \begin{equation}
        F_n(\xi, \eta) = F_n(\xi'', \eta'') = p(t, -\sqrt{1-t^2}).
        \label{eq:poly2}
    \end{equation}
    From \eqref{eq:poly1} and \eqref{eq:poly2}, we see that
    $\sqrt{1-t^2}$ must appear only with even powers in
    $F_n$. Thus, $F_n$ is really a polynomial in $t$ and
    $1-t^2$, which is just a polynomial in $t$. Since $t = \langle \xi, \eta
    \rangle$, the lemma is proved. \qed
\end{proof}

%====================================================================
\section{Legendre Polynomials}
%====================================================================

\begin{theorem}
  Let $\eta = (1, 0, \ldots, 0)$ and let $L_n(x)$ be a harmonic homogeneous polynomial of degree
  $n$ satisfying
  \begin{itemize}
    \item[(i)] $L_n(\eta) = 1$,
    \item[(ii)] $L_n(Rx) = L_n(x)$ for all rotation matrices
    $R$ such that $R\eta = \eta$.
  \end{itemize}
  Then $L_n(x)$ is the only harmonic homogeneous polynomial of
  degree $n$ obeying these properties. In particular, these two
  properties uniquely determine the corresponding spherical
  harmonic $L_n | _{S^{p-1}}$. Moreover, this spherical harmonic $L_n(\xi)$ is a polynomial in $\langle \xi, \eta \rangle$.
  \label{thm:leg}
\end{theorem}

\begin{proof}
    Let $\xi \in S^{p-1}$. Then there exist  $\nu \in {S}^{p-1}$ such that $\langle \nu , \eta \rangle = 0$ and
    $t \in [-1,1]$ such that $\xi = t\eta + \sqrt{1-t^2} \,
    \nu$. Notice $\langle \xi, \eta \rangle = t$ and that
    \begin{equation}
        \xi_1 = t, \quad \xi_2^2 + \cdots + \xi_p^2 = 1-t^2.
        \label{eq:t}
    \end{equation}
    As in the proof of Theorem \ref{thm:N}, we can write
    \begin{equation}
        L_n(x) = \sum_{j=1}^n x_1^j h_{n-j}(x_2, x_3, \ldots,
        x_p),
        \label{eq:exp}
    \end{equation}
    where the $h_{n-j}$ are homogeneous polynomials of degree
    $n-j$. Let $R$ be a rotation matrix as described in property
    (ii) in the statement of the theorem. When $R$ acts on $x$, it
    sends
    \[
        x_1, x_2, \ldots, x_p \overset{R}{\longmapsto} x_1, x'_2, \ldots, x'_p.
    \]
    By property (ii),
    \[
        0 = L_n(x) - L_n(Rx) = \sum_{j=1}^n x_1^j \left[ h_{n-j}(x_2, \ldots, x_p) - h_{n-j}(x'_2, \ldots, x'_p) \right],
    \]
    and using the linear independence of the $x_1^j$, we see that  all the $h_{n-j}(x_2, \ldots, x_p)$ are invariant under the
    rotation of coordinates $R$. Thus, these polynomials must depend only on the radius $\sqrt{x_2^2 +  \cdots + x_p^2}$, i.e.,
    \begin{equation}
        h_{n-j}(x_2, \ldots, x_p) = c_{n-j} \left( \sqrt{x_2^2 + \cdots +
        x_p^2} \right) ^{n-j},
        \label{eq:h}
    \end{equation}
    where the $c_{n-j}$ are constants, and $c_{n-j} = 0$ for all
    odd $n-j$ since the $h_{n-j}$ must be polynomials. Property
    (i) gives us one of these coefficients,
    \[
        1 = L_n(\eta) = c_0
    \]
    since all the $x_2, \ldots, x_p$ are zero for this vector.
    Since the form of the $h_{n-j}$ is given by \eqref{eq:h},
    knowing $c_0$ is enough to determine the rest of the $h_{n-j}$
    using the recursive relation \eqref{eq:recursion}
    and the fact that the $c_{n-j}$ are zero for odd $n-j$. Therefore, $L_n$ as well as the spherical
    harmonic $L_n|_{S^{p-1}}$ are uniquely determined by properties
    (i) and (ii). Finally, using \eqref{eq:exp}, \eqref{eq:h}, and \eqref{eq:t},
    \begin{equation*}
        L_n(\xi) = \underset{n-j=\text{even}}{\sum_{j=0}^n} \xi_1^j c_{n-j} (\xi_2^2 + \cdots  \xi_p^2)^{(n-j)/2},
    \end{equation*}
    or
    \begin{equation}
        L_n(\xi(t)) = \underset{n-j=\text{even}}{\sum_{j=0}^n} t^j \, c_{n-j} (1-t^2)^{(n-j)/2}.
        \label{eq:Pn}
    \end{equation}
    That is, $L_n(\xi)$ is a polynomial in $t = \langle \xi, \eta \rangle$, and the theorem is proved.
\qed
\end{proof}

\begin{definition}
  We call the polynomial $L_n(\xi)$ introduced in Theorem \ref{thm:leg} the \emph{Legendre polynomial\footnote{We follow the naming convention used in \cite{SpecFun:Hochstadt} and common in physics. Mathematicians usually refer to these as Legendre polynomials only when $p=3$ and as \emph{ultraspherical polynomials} for arbitrary $p$.} of degree $n$}. Written in terms of the variable $t$, we denote it by $P_n(t)$.
\end{definition}

\begin{remark}
  We can quickly obtain some properties of Legendre polynomials.
  First, from \eqref{eq:Pn}, we can see $P_n(t)$ is a
  polynomial of degree $n$. We can also compute
  \[
    1 = L_n(\eta) = P_n(\langle \eta, \eta \rangle) = P_n(1).
  \]
   We can determine the parity of the Legendre polynomials in
  more than one way. First, from \eqref{eq:Pn}, we can see that if
  $n$ is even, $P_n$ contains only even powers of $t$, while if
  $n$ is odd, $P_n$ contains only odd powers of $t$.
  Alternatively, we can see that
  \begin{align}
    P_n(-t) &= P_n(\langle -\xi, \eta \rangle) = L_n(-\xi) =
    (-1)^nL_n(\xi) \notag \\
    &= (-1)^nP_n(\langle \xi, \eta \rangle) =
    (-1)^nP_n(t).
    \label{eq:LegPar}
  \end{align}
  Either way, we determine that $P_n$ is even whenever $n$ is even,
  and $P_n$ is odd whenever $n$ is odd.
\end{remark}

In the following theorem, we demonstrate how to write the Legendre
polynomials in terms of an orthonormal set of spherical harmonics.

\begin{theorem}[Addition Theorem for Legendre Polynomials] 
  Let \\ $\{ Y_{n,j}(\xi) \}_{j=1}^{N(p,n)}$ be an
  orthonormal set of $n$-th degree spherical harmonics. Then the Legendre
  polynomial of degree $n$ may be written as
  \begin{equation} \label{eq:addthm}
    P_n(\langle \xi, \eta \rangle) = {\Omega_{p-1} \over N(p,n)}
    \sum_{j=1}^{N(p,n)} Y_{n,j}(\xi)Y_{n,j}(\eta).
  \end{equation}
  \label{thm:add}
\end{theorem}

\begin{proof}
  Since \eqref{eq:addthm} is invariant under coordinate rotations, we can choose $\eta =
  (1,0,\ldots,0)$. Consider again the function $F_n(\xi, \eta)$
  defined in \eqref{eq:F}. Since we have already defined the
  vector $\eta$, we will think of it as a fixed parameter in the
  function $F_n$. For now, we will write $F_n(\xi;\eta)$ and think of $F_n$
  as a function of one vector $\xi$.

  First, notice that $F_n$, being a linear combination of
  spherical harmonics $Y_{n,j}(\xi)$ (since we consider the
  $Y_{n,j}(\eta)$ to be constants), is itself a spherical
  harmonic, i.e., the restriction of a harmonic homogeneous
  polynomial to the unit sphere. Next, notice that any
  rotation of coordinates $R$ that leaves $\eta$ fixed leaves the
  function $F_n$ invariant. Indeed, using Lemma
  \ref{prop:Finv} or Lemma \ref{lem:Finn},
  \[
    F(R\xi; \eta) = F(R\xi;R\eta) = F(\xi;\eta) \text{ for all } R \text{ such that } R\eta = \eta.
  \]
  Notice that in the application of the rotation $R$, we only rotate the coordinates of $\xi$ in $F_n$ since we consider
  $\eta$ a fixed parameter, though this made no difference in the  calculation.

  Let us normalize the function $F_n$ by dividing it by   the constant $F_n(\eta;\eta)$.
  Thus, we have found a function
  $F_n(\xi;\eta) / F_n(\eta;\eta)$ that obeys
  all the properties of the Legendre polynomial described in the
  statement of Theorem \ref{thm:leg}. Since, by the same theorem,
  these properties uniquely define the Legendre polynomial, we
  conclude
  \[
    P_n(\langle \xi, \eta \rangle) =
    {F_n(\xi,\eta) \over F_n(\eta,\eta)}.
  \]
  To complete the proof, we will compute $F_n(\eta,\eta)$. Since, by Lemma \ref{lem:Finn}, $F_n(\eta,\eta)$ depends only on the
  inner product $\langle \eta, \eta \rangle = 1$, it is a constant. Thus,
  \begin{align*}
      F_n(\eta,\eta)\Omega_{p-1} &= \underset{\eta \in S^{p-1}}{\int} F_n(\eta,\eta)\, d\Omega_{p-1} \\
                                                    &= \underset{\eta \in S^{p-1}}{\int} \sum_{j=1}^{N(p,n)}Y_{n,j}(\eta)^2\, d\Omega_{p-1}\\
                                                    &= \sum_{j=1}^{N(p,n)} \underset{\eta \in S^{p-1}}{\int} Y_{n,j}(\eta)^2\, d\Omega_{p-1}
                                                    = N(p,n),
  \end{align*}
  using the orthonormality of the $n$-th degree spherical
  harmonics. We see that
  \[
    P_n(\langle \xi, \eta \rangle) = {F_n(\xi, \eta) \over
    F_n(\eta, \eta)} = {\Omega_{p-1} \over N(p,n)}
    \sum_{j=1}^{N(p,n)} Y_{n,j}(\xi) Y_{n,j}(\eta),
  \]
  and we are finished. \qed
\end{proof}

In addition, we can expand spherical harmonics in terms of
Legendre polynomials. To show this, we will need first the
following result.

\begin{lemma}
\label{lem:det}
  For any set $\{ Y_{n,j} \}_{j=1}^k$ of $k  \leq N(p,n)$ linearly
  independent $n$-th degree spherical harmonics, there exists a
  set $\{ \eta_i \}_{i=1}^k$ of unit vectors such that the
  $k \times k$ determinant
  \begin{equation}
    \left\vert
    \begin{array}{cccc}
    Y_{n,1}(\eta_1) & Y_{n,1}(\eta_2) & \cdots & Y_{n,1}(\eta_k)\\
    Y_{n,2}(\eta_1) & Y_{n,2}(\eta_2) & \cdots & Y_{n,2}(\eta_k)\\
        \vdots      &     \vdots      & \ddots &      \vdots    \\
    Y_{n,k}(\eta_1) & Y_{n,k}(\eta_2) & \cdots & Y_{n,k}(\eta_k)
    \end{array}
    \right\vert
    \label{eq:det3}
  \end{equation}
is nonzero.
\end{lemma}

\begin{proof}
  We will prove the lemma by induction on $k$.
  First, consider a linearly independent set $\{Y_{n,1}\}$ of one spherical harmonic of degree $n$.
  $Y_{n,1}$ cannot be the zero function. Then, there exists a unit vector, call it $\eta_1$, such that the
  determinant $|Y_{n,1}(\eta_1)| = Y_{n,1}(\eta_1) \neq 0$. Thus, the lemma holds for the case $k=1$. Now, suppose the lemma holds for some
  $k=\ell-1 \leq N(p,n)-1$, and let $\{ Y_{n,j} \}_{j=1}^\ell$ be a set of $\ell \leq N(p,n)$ linearly
  independent $n$-th degree spherical harmonics. By the induction hypothesis, we can choose a
  set $\{ \eta_i \}_{i=1}^{\ell-1}$ of unit vectors such that the
  $(\ell-1) \times (\ell-1)$ determinant
  \begin{equation}
  \label{eq:det}
     \Delta_\ell =
    \left\vert
    \begin{array}{cccc}
    Y_{n,1}(\eta_1)      & Y_{n,1}(\eta_2)      & \cdots & Y_{n,1}(\eta_{\ell-1}) \\
    Y_{n,2}(\eta_1)      & Y_{n,2}(\eta_2)      & \cdots & Y_{n,2}(\eta_{\ell-1}) \\
        \vdots           &     \vdots           & \ddots &          \vdots        \\
    Y_{n,\ell-1}(\eta_1) & Y_{n,\ell-1}(\eta_2) & \cdots & Y_{n,\ell-1}(\eta_{\ell-1}) \\
    \end{array}
    \right\vert \neq 0.
  \end{equation}
  Now consider the spherical harmonic defined by the $\ell \times \ell$
  determinant,
  \begin{equation}
    \Delta =
    \left\vert
    \begin{array}{ccccc}
    Y_{n,1}(\eta_1)      & Y_{n,1}(\eta_2)      & \cdots & Y_{n,1}(\eta_{\ell-1})      & Y_{n,1}(\xi)          \\
    Y_{n,2}(\eta_1)      & Y_{n,2}(\eta_2)      & \cdots & Y_{n,2}(\eta_{\ell-1})      & Y_{n,2}(\xi)           \\
        \vdots           &     \vdots           & \ddots &          \vdots             &     \vdots             \\
    Y_{n,\ell-1}(\eta_1) & Y_{n,\ell-1}(\eta_2) & \cdots & Y_{n,\ell-1}(\eta_{\ell-1}) & Y_{n,\ell-1}(\xi)\\
    Y_{n,\ell}(\eta_1)   & Y_{n,\ell}(\eta_2)   & \cdots & Y_{n,\ell}(\eta_{\ell-1}) & Y_{n,\ell}(\xi)\\
    \end{array}
    \right\vert.
    \label{eq:det2}
  \end{equation}
  If we compute this determinant by performing a cofactor  expansion down the last column and indicating by $\Delta_j$ the minor determinant
  corresponding to $Y_{n,j}(\xi)$,
  \[
        \Delta = \sum_{j=1}^\ell (-1)^{\ell+j} \, \Delta_j \, Y_{n,j}(\xi) ,
  \]
  we can see that we will have a  linear combination of spherical harmonics $Y_{n,j}(\xi)$ which are linearly independent. Thus, for the
  determinant to vanish identically, all the coefficients of the $Y_{n,j}(\xi)$ must vanish. But notice that
  the coefficient of $Y_{n,\ell}(\xi)$ is the determinant
  \eqref{eq:det} which does not vanish. Thus, the spherical
  harmonic given by the determinant \eqref{eq:det2} is not the
  zero function; i.e., there exists a unit vector $\xi = \eta_\ell$ such that the determinant \eqref{eq:det2} is nonzero.
  Therefore, the lemma holds for the case $k=\ell$ and, by induction, for all $k \leq N(p,n)$. \qed
\end{proof}

\begin{theorem}
  For any spherical harmonic $Y_n(\xi)$ of degree $n$, there exist
  coefficients $a_k$ and unit vectors
  $\eta_k$ such that
  \[
    Y_n(\xi) = \sum_{k=1}^{N(p,n)} a_k P_n(\langle \xi, \eta_k
    \rangle).
  \]
  \label{thm:SHexp}
\end{theorem}

\begin{proof}
    Let $\{ Y_{n,j}(\xi) \}_{j=1}^{N(p,n)}$ be an orthonormal set of $n$-th degree spherical harmonics, and let $Y_n(\xi)$ be any spherical
    harmonic of degree $n$. For a unit vector $\eta$,  we can write
  \begin{equation}
    P_n(\langle \xi, \eta \rangle) = {\Omega_{p-1} \over N(p,n)}
    \sum_{j=1}^{N(p,n)} Y_{n,j}(\xi)Y_{n,j}(\eta),
    \label{eq:add}
  \end{equation}
  by Theorem \ref{thm:add}. Let us choose a set of unit vectors $\eta_j$ such that the determinant
  \eqref{eq:det3} with $k=N(p,n)$ is nonzero; this is possible by Lemma
  \ref{lem:det}. Replacing $\eta$ by $\eta_j$ in \eqref{eq:add} for $1 \leq j \leq
  N(p,n)$ creates the system of equations
  \[
    {N(p,n) \over \Omega_{p-1}}\!\!
    \left(\!\!\!
        \begin{array}{c}
            P_n(\langle \xi, \eta_1 \rangle) \\ P_n(\langle \xi,
            \eta_2
            \rangle)\\ \vdots \\ P_n(\langle \xi, \eta_{N(p,n)} \rangle)
        \end{array}
    \!\!\!\! \right) \!\!=\!\!
    \left(\!\!\!
    \begin{array}{ccc}
    Y_{n,1}(\eta_1) & \cdots & Y_{n,N(p,n)}(\eta_1)\\
    Y_{n,1}(\eta_2) & \cdots & Y_{n,N(p,n)}(\eta_2)\\
        \vdots      & \ddots &      \vdots    \\
    Y_{n,1}(\eta_{N(p,n)}) & \cdots & Y_{n,N(p,n)}(\eta_{N(p,n)})\\
    \end{array}
    \!\!\!\!\right)\!\!\!
    \left(\!\!\!
        \begin{array}{c}
            Y_{n,1}(\xi) \\ Y_{n,2}(\xi) \\ \vdots \\ Y_{n,N(p,n)}(\xi)
        \end{array}
    \!\!\!\!\right).
  \]
  The determinant of the $N(p,n) \times N(p,n)$ coefficient matrix
  on the right-hand side of this system has nonzero determinant by
  our choice of the vectors $\eta_j$. Thus, this system is
  invertible; i.e., there exist coefficients
  $c_{\ell, m}$ such that
  \begin{equation}
    Y_{n, \ell}(\xi) = \sum_{m=1}^{N(p,n)} c_{\ell,m} P_n(\langle \xi,
    \eta_m
    \rangle), \text{ for each } 1 \leq \ell \leq N(p,n).
    \label{eq:almost}
  \end{equation}
  Now since any spherical harmonic, and in particular $Y_n(\xi)$, can be expanded in terms of the
  basis functions $\{ Y_{n,i}(\xi) \}_{i=1}^{N(p,n)}$, there
  exist coefficients $a_k$ such that
  \[
    Y_n(\xi) = \sum_{k=1}^{N(p,n)} a_k P_n(\langle \xi, \eta_k
    \rangle),
  \]
  by \eqref{eq:almost}, and the theorem is proved. \qed
\end{proof}

We will now list and prove several basic properties of spherical
harmonics and Legendre polynomials, some of which are useful for
making estimates. The proofs are straightforward.

\begin{lemma}
  For any spherical harmonic $Y_n(\xi)$,
  \begin{equation}
    Y_n(\xi) = {N(p,n) \over \Omega_{p-1}} \underset{\eta \in S^{p-1}}{\int}
    Y_n(\eta) P_n(\langle \xi, \eta \rangle)\, d\Omega_{p-1}.
    \label{eq:goal}
  \end{equation}
  \label{lem:Yint}
\end{lemma}

\begin{proof}
  Let $Y_n(\xi)$ be any $n$-th degree spherical harmonic, and let
  $\{Y_{n,j}\}_{j=1}^{N(p,n)}$ be an orthonormal set of such
  functions. Then, we can expand $Y_n(\eta)$ in terms of this
  basis; i.e., for some coefficients $a_j$,
  \[
    Y_n(\eta) = \sum_{j=1}^{N(p,n)} a_j Y_{n,j}(\eta).
  \]
  Using this expansion  and Theorem \ref{thm:add} to rewrite $P_n(\langle \xi, \eta \rangle)$, the right-hand side of
  \eqref{eq:goal} becomes
  \begin{align*}
    &{N(p,n) \over \Omega_{p-1}}  \underset{\eta \in S^{p-1}}{\int} \left[ \sum_{j=1}^{N(p,n)} a_j Y_{n,j}(\eta)
    \right] \left[ {\Omega_{p-1} \over N(p,n)}
    \sum_{k=1}^{N(p,n)} Y_{n,k}(\xi)Y_{n,k}(\eta) \right] \,
    d\Omega_{p-1},
  \end{align*}
  or
  \begin{align*}
    \sum_{j,k=1}^{N(p,n)} a_j Y_{n,k}(\xi) \left[
    \underset{\eta \in S^{p-1}}{\int} Y_{n,j}(\eta)Y_{n,k}(\eta) \, d\Omega_{p-1}
    \right] = \sum_{j=1}^{N(p,n)} a_j Y_{n,j}(\xi) = Y_n(\xi),
  \end{align*}
  as required. \qed
\end{proof}

\begin{proposition}
  The  Legendre polynomials $P_n(t)$ are bounded
\begin{align}
     |P_n(t)| \leq 1, \text{ for all } t \in [0,1]
  \end{align}
  and obey the following normalization condition
  \begin{align}
     \underset{\xi \in S^{p-1}}{\int} P_n(\langle \xi, \eta\rangle)^2 \, d\Omega_{p-1} = {\Omega_{p-1} \over N(p,n)}.\label{eq:prop}
\end{align}
\end{proposition}

\begin{proof}
  We  rewrite $P_n(t)^2$ using Theorem   \ref{thm:add}. Then, we use the Cauchy-Schwarz inequality, viewing the sum of products in this expression
  as  a dot product, to derive the required result:
        \begin{align*}
             P_n(\langle \xi, \eta \rangle)^2 &= \left[ {\Omega_{p-1} \over N(p,n)} \sum_{j=1}^{N(p,n)} Y_{n,j}(\xi) Y_{n,j}(\eta) \right]^2 \\
                                                                 & \leq \left[ {\Omega_{p-1} \over N(p,n)} \sum_{j=1}^{N(p,n)} Y_{n,j}(\xi)^2 \right]
                                                                     \left[ {\Omega_{p-1} \over N(p,n)} \sum_{j=1}^{N(p,n)} Y_{n,j}(\eta)^2 \right], \label{eq:prop}
        \end{align*}
        so
        \begin{align*}
             P_n&(\langle \xi, \eta \rangle)^2 \leq P_n(\langle \xi , \xi \rangle) P_n(\langle \eta , \eta \rangle)= P_n(1)^2 = 1,
        \end{align*}
proving the first result.

 Again, we  use Theorem \ref{thm:add} to rewrite the integral on the left side of \eqref{eq:prop} as
        \begin{align*}
            \underset{\xi \in S^{p-1}}{\int}
            \!\!\!\! {\Omega_{p-1}^2 \over N(p,n)^2}
            \sum_{j,k=1}^{N(p,n)}
            Y_{n,j}(\xi)Y_{n,j}(\eta)Y_{n,k}(\xi)Y_{n,k}(\eta) \, d\Omega_{p-1}
        \end{align*}
        which then becomes
        \[
            {\Omega_{p-1}^2 \over N(p,n)^2}
            \sum_{j=1}^{N(p,n)} Y_{n,j}(\eta)Y_{n,j}(\eta) = {\Omega_{p-1} \over
            N(p,n)} P_n(\langle \eta, \eta \rangle) = {\Omega_{p-1} \over N(p,n)},
        \]
 thus proving the second result.
 \qed
 \end{proof}

  \begin{proposition}
  The spherical harmonics $Y_n(\xi)$ satisfy the following inequality
  \begin{align}
     |Y_n(\xi)| \leq \sqrt{{N(p,n) \over \Omega_{p-1}}\, \underset{\eta \in S^{p-1}}{\int} Y_n(\eta)^2\, d\Omega_{p-1}}.
  \end{align}
\end{proposition}

\begin{proof}
        We start by taking the square of equation \eqref{eq:goal}:
        \[
            Y_n(\xi)^2 = {N(n,p)^2 \over \Omega_{p-1}^2} \left[ \underset{\eta \in S^{p-1}}{\int}
                                 Y_n(\eta)P_n(\langle \xi, \eta \rangle) \,     d\Omega_{p-1}\right]^2.
        \]
        Viewing the integral as an inner product, we apply the Cauchy-Schwarz inequality,
        \begin{align*}
            Y_n(\xi)^2 &\leq {N(n,p)^2 \over \Omega_{p-1}^2} \left[ \underset{\eta \in S^{p-1}}{\int}
                Y_n(\eta)^2\,d\Omega_{p-1}\right] \left[ \underset{\eta \in S^{p-1}}{\int}
                P_n(\langle \xi, \eta \rangle)^2 \,  d\Omega_{p-1}\right].
        \end{align*}
        Thus,
        \begin{align*}
            Y_n(\xi)^2 &\leq {N(p,n) \over \Omega_{p-1}} \underset{S^{p-1}}{\int} Y_n(\eta)^2\, d\Omega_{p-1},
        \end{align*}
        where we used property (\ref{eq:prop}) in the last step.
\qed
\end{proof}

  We will now begin to investigate the properties of Legendre polynomials as orthogonal polynomials.
  Let us rewrite the integral in (\ref{eq:prop}).
  Note that the integrand depends only on the
  inner product $\langle \xi, \eta \rangle$ and that we integrate $\xi$ over the surface of the
  $(p-1)$-sphere. We can take advantage of these observations to reduce the $(p-1)$-dimensional integral to a
  one-dimensional integral.

  Since we integrate over the entire sphere, we can perform any
  rotation of coordinates without changing the value of the
  integral. Let us impose a coordinate change $R$ that aligns the unit vector $\eta$
  along the $x_p$-axis, which we will picture pointing ``north,''
  i.e., take $\eta = (0, \ldots, 0, 1)$. If we let $t = \langle \xi, \eta \rangle$, we can write the unit vector $\xi$ as
  \[
    \xi = \langle \xi, \eta \rangle \eta + \left( \xi - \langle \xi, \eta \rangle \eta
    \right)= t \eta + \sqrt{1-t^2} \: \nu,
  \]
  for some unit vector $\nu$ normal to $\eta$. Notice that any
  such $\nu$ gives the same value for the inner
  product $\langle \xi, \eta \rangle$ and thus the same value for
  the integrand in (\ref{eq:prop}). 
  Note further that the collection of all such vectors $\nu$,
  \[
       \{\nu \in \mathbb{R}^p : |\nu|=1, \langle \nu, \eta \rangle = 0\} ,
  \]
 forms  a parallel of the $(p-1)$-sphere which is  a $(p-2)$-sphere:
   \[
       \{ \nu\in \mathbb{R}^p : |\nu|=1, \nu_p = 0\} = S^{p-2} .
  \]
  A visual representation is depicted in Figure \ref{fig:sph}.

  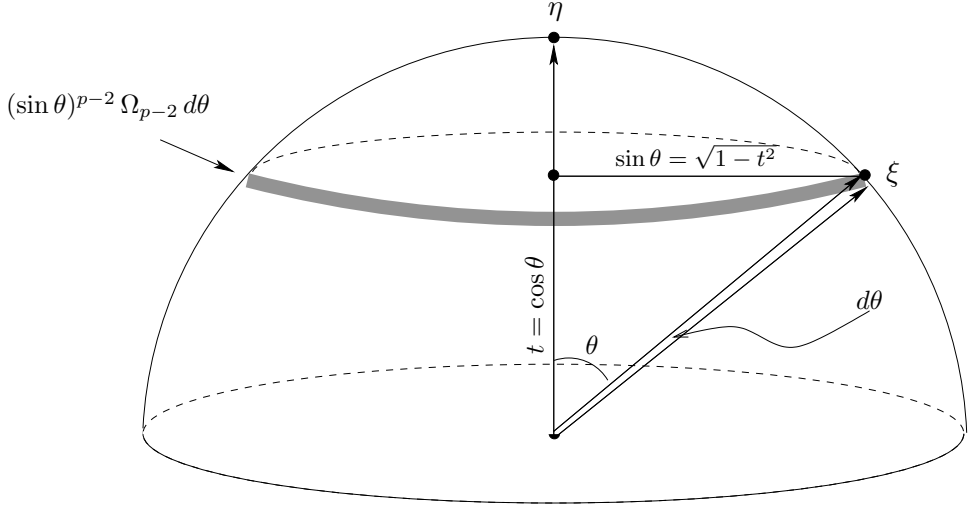
\begin{figure}[h!]
    \begin{center}
        \setlength{\unitlength}{1mm}
        \begin{picture}(98,70)
            \put(0,0){\includegraphics[width=11cm]{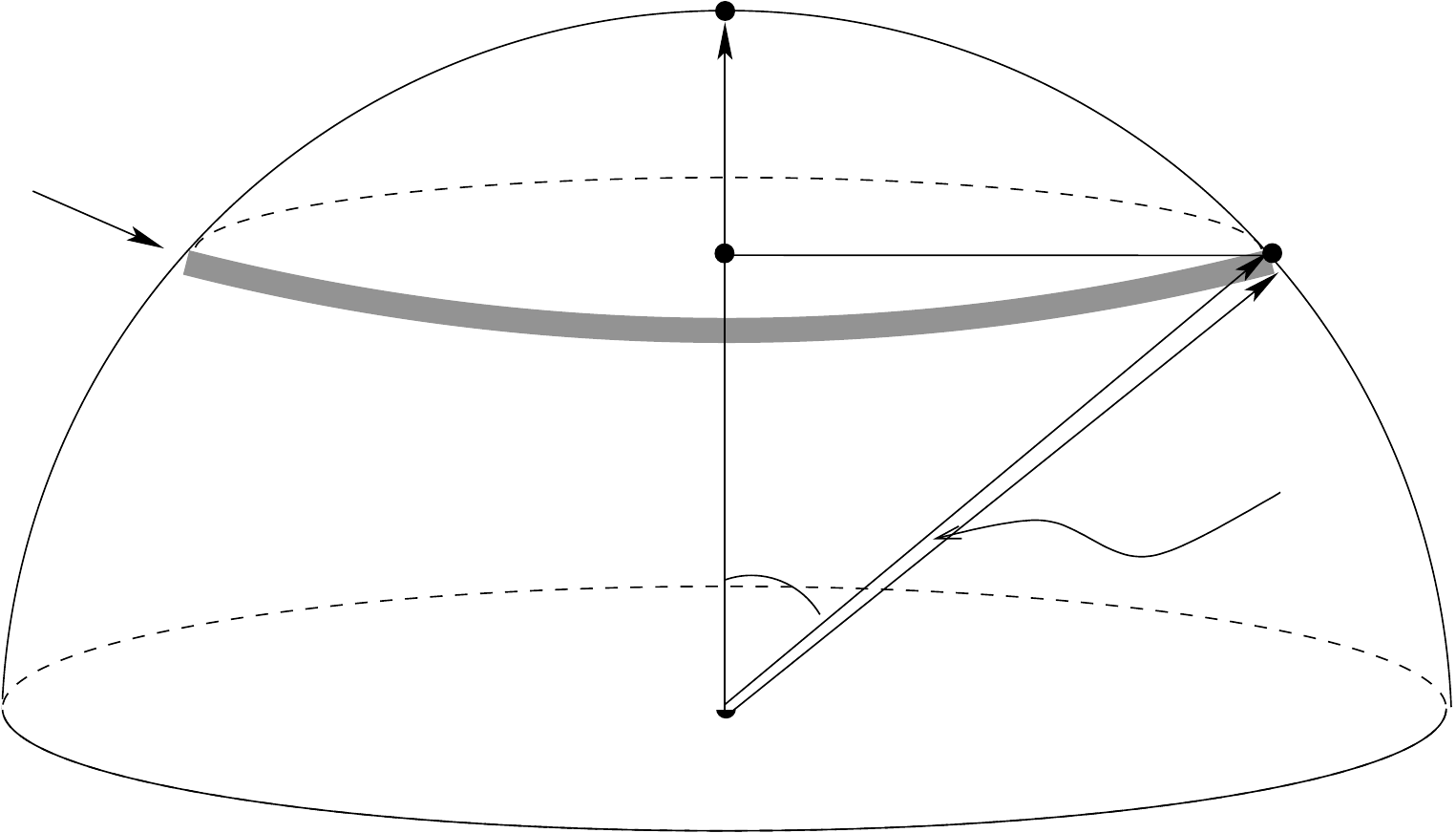}}
            \put(54,65){$\eta$}  \put(99,43){$\xi$} \put(59,20){\small $\theta$}
            \put (95,26){\small $d\theta$}
            \put(51,20){\small \rotatebox{90}{$t=\cos\theta$}}
            \put(63,45){\footnotesize $\sin\theta=\sqrt{1-t^2}$}
            \put(-18,52){\small $(\sin\theta)^{p-2} \, \Omega_{p-2}  \, d\theta$}
        \end{picture}
    \end{center}
    \caption{\footnotesize In the reduction of a spherically symmetric $(p-1)$-dimensional integral to a 1-dimensional integral,
                    we imagine the $(p-1)$-dimensional sphere as a sum of infinitesimal $(p-2)$-dimensional spheres (parallels) of
                    thickness $d\theta$.}
    \label{fig:sph}
  \end{figure}

To get some intuition, let's think of the familiar 3-dimensional case. The 2-dimensional sphere can be thought as a sum of infinitesimal rings
oriented along the parallels of the sphere. The infinitesimal ring defined by the azimuthal angle $\theta$ has an infinitesimal thickness $d\theta$ and,
therefore, it corresponds to a solid angle
$$
    d\Omega_2 =  (2\pi\sin\theta) \, d\theta .
$$
The term inside the parenthesis is the length of the ring; multiplied by its thickness it gives the ``area'' of the ring. We can rewrite $d\Omega_2$ as
$-2\pi d(\cos\theta)$ and thus
$$
    d\Omega_2  = - \Omega_1 \, dt
$$
in terms of the variable $t$.

Similarly, in $p$ dimensions
$$
   d\Omega_{p-1} = \Omega_{p-2} \, (\sin\theta)^{p-2} \, d\theta ~.
$$
The factor $(\sin\theta)^{p-2}$ is easy to explain: The radius of the $(p-2)$-sphere is $R=\sin\theta$; its volume will be proportional to
$R^{p-2}$. It is straightforward to express $d\Omega_{p-1}$ in terms of $t$:
\begin{align*}
    d\Omega_{p-1} &= -\Omega_{p-2} \, (\sin\theta)^{p-3} \, d (\cos\theta) \\
     &= - \Omega_{p-2} \, (1-t^2)^{p-3 \over 2} \, dt.
  \end{align*}

  Now we are ready to write (\ref{eq:prop}) as a 1-dimensional integral over $t$. Indeed,
  \[
    \underset{\xi \in S^{p-1}}{\int} P_n(\langle \xi, \eta\rangle)^2 \, d\Omega_{p-1}
          = \int_{-1}^1 P_n(t)^2  \, (1-t^2)^{p-3 \over 2} \Omega_{p-2}\,dt,
  \]
  which implies
  \[
    \int_{-1}^1 P_n(t)^2 \, (1-t^2)^{p-3 \over 2} \,dt = {\Omega_{p-1} \over N(p,n)\Omega_{p-2}}.
  \]
  This expression gives us the norm of the $n$-th Legendre
  polynomial with respect to the weight $w(t) = (1-t^2)^{p-3 \over 2}$,
  namely
  \begin{equation}
    \| P_n(t) \|_w = \sqrt{\langle P_n(t), P_n(t) \rangle_w} =
    \sqrt{{\Omega_{p-1} \over N(p,n) \, \Omega_{p-2}}}.
    \label{eq:LegNorm}
  \end{equation}

Note that in coming up with this fact, we have essentially proved the following lemma, which will be used
to show the next few results.

\begin{lemma}
  Let $\eta$ be a unit vector and $f$ be a function. Then
  \[
    \underset{\xi \in S^{p-1}}{\int} f(\langle \xi, \eta
    \rangle)\,d\Omega_{p-1} = \Omega_{p-2} \int_{-1}^1
    f(t)(1-t^2)^{(p-3)/2}\,dt.
  \]
  \label{lem:red}
\end{lemma}

\begin{theorem}
  Any two distinct Legendre polynomials $P_n(t),P_m(t)$
  are orthogonal over the interval $[-1,1]$ with
  respect to the weight $(1-t^2)^{p-3 \over 2}$. That is,
  \[
    \int_{-1}^1P_n(t)P_m(t)(1-t^2)^{p-3 \over 2} \,dt = 0, \quad
    \text{for } n \neq m.
  \]
  \label{thm:LegOrth}
\end{theorem}

\begin{proof}
  Let $\eta = (1, 0, \ldots, 0)$. The Legendre polynomial
  $P_n(\langle \xi, \eta \rangle)$ is equal to the spherical harmonic $L_n(\xi)$ having
  the properties listed in Theorem \ref{thm:leg}. By Theorem
  \ref{thm:SHorth},
  \[
    0 = \underset{\xi \in S^{p-1}}{\int} L_n(\xi)L_m(\xi)\,d\Omega_{p-1} =
    \underset{\xi \in S^{p-1}}{\int} P_n(\langle \xi, \eta \rangle)P_m(\langle \xi, \eta \rangle)\,d\Omega_{p-1},
    \quad \text{for } n \neq m.
  \]
  Since the integral on the right-hand side of the above equation
  depends only on the inner product $\langle \xi, \eta \rangle$,
  we can use Lemma \ref{lem:red} to rewrite
  this integral as
  \[
    \Omega_{p-2} \int_{-1}^1P_n(t)P_m(t)(1-t^2)^{p-3 \over 2} \,dt = 0, \quad
    \text{for } n \neq m,
  \]
  thus completing the proof.
\qed
\end{proof}

    The above theorem allows us to use the results from Chapter \ref{ch:pol} to write down more properties of the
    Legendre polynomials. Using the comment just above Section \ref{sec:approx}, we can find the
    Rodrigues formula for the Legendre polynomials in $p$
    dimensions. Theorem \ref{thm:LegOrth} tells us that the Legendre polynomials are orthogonal
    with respect to $w(t) = (1-t^2)^{(p-3)/2}$ so that
    \begin{align*}
        P_n(t) &= c_n (1-t^2)^{-(p-3)/2} \left({d \over dt}\right)^n
        \left[ (1-t^2)^{(p-3)/2}(1-t^2)^n \right] \\
        &= c_n (1-t^2)^{(3-p)/2} \left({d \over dt}\right)^n
        (1-t^2)^{n+(p-3)/2},
    \end{align*}
    and it remains to compute $c_n$. We know that $P_n(1)=1$, so
    \[
      1 = \: c_n (1-t^2)^{(3-p)/2} \left({d \over dt}\right)^n
        (1-t^2)^{n+(p-3)/2} \Big |_{t=1}.
    \]
    Carrying out two differentiations,
    \begin{align*}
      1 =& \: c_n (1-t^2)^{(3-p)/2} \left({d \over dt}\right)^{n-1}
      \left( {n+(p-3) \over 2} \right)(-2t) (1-t^2)^{n-1+(p-3)/2} \Big |_{t=1}\\
      =& \: c_n (1-t^2)^{(3-p)/2}\! \left({d \over dt}\right)^{\!\!n-2}\!
      \left[ \left( {n+(p-3) \over 2} \right)_{\!2}\!\!(-2t)^2 (1-t^2)^{n-2+(p-3)/2} + \cdots \right] \Big |_{t=1}
    \end{align*}
    where we have used the falling factorial notation and left
    terms that are higher order in $1-t^2$ in the ``$\, \cdots \,$''
    because they will vanish when we substitute $t=1$. Continuing
    the pattern,
    \begin{align*}
      1 =& \:c_n (1-t^2)^{(3-p)/2} \left[ \left( n+(p-3)/2 \right)_n(-2t)^n (1-t^2)^{n-n+(p-3)/2} + \cdots \right] \Big
      |_{t=1}\\
      =& \:c_n \left( n+(p-3)/2 \right)_n(-2)^n,
    \end{align*}
    implying that
    \[
        c_n = { (-1)^n \over 2^n \left( n+(p-3)/2 \right)_n }.
    \]
    We have shown the following.

\begin{proposition}[Rodrigues Formula for Legendre Polynomials]
  \begin{equation}
    P_n(t) = { (-1)^n \over 2^n \left( n+(p-3)/2 \right)_n } (1-t^2)^{(3-p)/2} \left({d \over dt}\right)^n
        (1-t^2)^{n+(p-3)/2}.
    \label{eq:LegRod}
  \end{equation}
\end{proposition}

The above Rodrigues formula allows us to prove the following  two  properties of the Legendre polynomials.

\begin{proposition}
In  $p$ dimensions, the Legendre polynomial $P_n(t)$ of degree $n$
satisfies the differential equation
    \begin{equation}
        (1-t^2)P''_n(t) + (1-p)tP'_n(t) + n(n+p-2)P_n(t) = 0.
        \label{eq:diffeq}
    \end{equation}
\end{proposition}

In what follows, we will often write $P_n$ and $w$ instead of
$P_n(t)$ and $w(t)$ for simplicity. A portion
of the proof will be left as a straightforward computation for the
reader.

\begin{proof}
  Consider the expression ${d \over dt} \left[ (1-t^2)^{(p-1)/2} P_n'  \right]$. We can use the product rule to rewrite this as
  \[
    \left({p-1 \over 2}\right)(1-t^2)^{(p-3)/2}(-2t)P'_n +
    (1-t^2)^{(p-1)/2}P''_n,
  \]
  or,
  \begin{equation}
    w \, \left[ (1-p)tP'_n+(1-t^2)P''_n
    \right].
    \label{eq:1stDE}
  \end{equation}
  The term in brackets is a polynomial of degree $n$ which
  can be written as a linear combination of the first $n$ Legendre
  polynomials. Hence
  \begin{equation}
    {d \over dt}\left[ (1-t^2)wP'_n \right] = w \sum_{j=0}^n c_jP_j.
    \label{eq:2ndDE}
  \end{equation}
  Multiplying each side by $P_k$, where $0 \leq k \leq n$ and
  integrating  over the interval $[-1,1]$, we get
  \[
      c_k \|P_k\|^2 = \int_{-1}^1 P_k \, {d \over dt} \left[ (1-t^2)wP'_n \right] \, dt.
  \]
  Integrating the right-hand side by parts now twice and noticing that the boundary terms vanish, we find
  \begin{align*}
    c_k \|P_k\|^2 &= \underset{0}{\underbrace{P_k(1-t^2)wP'_n \big|_{-1}^1}} -
    \int_{-1}^1P'_k(1-t^2) w P'_n \, dt\\
    &= \underset{0}{\underbrace{-P'_k(1-t^2)wP_n \big|_{-1}^1}} + \int_{-1}^1
    \left\{ {d \over dt} \left[ P'_k (1-t^2)w \right] \right\} P_n \, dt.
  \end{align*}
  The expression in curly brackets can be written as $w q_k$, where $q_k$ is a polynomial of degree $k$. We can then use Proposition
  \ref{fact:int0} for the two integrals to discover that $c_k=0$ for all $k < n$. Returning to equation \eqref{eq:2ndDE}, this result
  implies that
  \[
    {d \over dt}\left[ (1-t^2)wP'_n \right] = w c_nP_n.
  \]
  To determine $c_n$, we will compute the coefficient of the
  highest power of $t$ on each side of the above equation, using \emph{Newton's binomial expansion}.
  By keeping only the highest-order term in each binomial expansion, the reader can show that the left-hand side becomes
  \[
    {(-1)^{(p-1)/2} \over 2^n} \, {(2n+p-3)_n \over
    (n+(p-3)/2)_n}\,
    n(n+p-2) t^{n+p-3} + \cdots,
  \]
  while the right-hand side becomes
  \[
    c_n \, {(-1)^{(p-3)/2} \over 2^n} \, {(2n+p-3)_n \over  (n+(p-3)/2)_n} \, t^{n+p-3}+\cdots
  \]
  Comparing the above equations, we conclude that $c_n = -n(n+p-2)$. Inserting this into \eqref{eq:2ndDE} and using
  \eqref{eq:1stDE}, we find
  \[
    (1-t^2)P''_n + (1-p)tP'_n +n(n+p-2)P_n = 0,
  \]
  completing the proof.
\qed
\end{proof}

\begin{proposition}
  The Legendre polynomials in $p$ dimensions satisfy the recurrence relation
  \[
    (n+p-2)P_{n+1}(t) - (2n+p-2)tP_n(t) + nP_{n-1}(t) = 0.
  \]
\end{proposition}

Just as in the previous proof, we will leave part of the following
proof to the reader.

\begin{proof}
  We know from Proposition \ref{fact:recurr} that a relation of the form
  \begin{equation}
    P_{n+1} - (A_nt+B_n)P_n + C_nP_{n-1} = 0
    \label{eq:LegRecur1}
  \end{equation}
  exists. We can quickly determine $B_n$. Recall from
  \eqref{eq:LegPar} that $P_n$ is even (respectively, odd)
  whenever $n$ is even (respectively, odd). Rewriting the above
  equation as
  \[
    P_{n+1} - A_ntP_n + C_nP_{n-1} = B_nP_n,
  \]
  we have an odd polynomial equal to an even polynomial, unless
  both sides vanish. Since the first option is not possible, we
  conclude that both sides of the equation must cancel, which
  implies that $B_n = 0$. We will use the notation and results of Proposition \ref{fact:recurr} to compute $A_n,
  C_n$. Keeping only the highest-order terms in each binomial
  expansion in the Rodrigues formula \eqref{eq:LegRod} and carrying
  out the derivatives, it is straightforward to show
  that the leading coefficient of $P_n$ is
  \[
    k_n = {(2n+p-3)_n \over 2^n (n+(p-3)/2)_n},
  \]
  from which we determine the leading coefficient of $P_{n+1}$,
  \[
    k_{n+1} = {(2n+p-1)_{n+1} \over 2^{n+1} (n+(p-1)/2)_{n+1}},
  \]
  allowing us to compute
  \[
    A_n = {k_{n+1} \over k_n} = {2n+p-2 \over n+p-2}.
  \]
  Now, using \eqref{eq:LegNorm},
  \begin{align*}
    C_n &= {A_n \over A_{n-1}} {\|P_n\|^2 \over \|P_{n-1}\|^2}\\
    &= {2n+p-2 \over n+p-2} \, {n+p-3 \over 2n+p-4} \, {\Omega_{p-1}
    \over N(p,n) \Omega_{p-2}} \, {N(p,n-1) \Omega_{p-2} \over
    \Omega_{p-1}}.
  \end{align*}
  Using Theorem \ref{thm:N}, we can compute
  \[
    C_n = {n \over n+p-2}.
  \]
  Inserting these results into \eqref{eq:LegRecur1} and
  multiplying by $n+p-2$, we find the required result. \qed
\end{proof}

It is also interesting to observe that once we find the Legendre
polynomials in $p=2$ and $p=3$ dimensions, they are determined for
all higher dimensions. In particular, the following theorem proves, more
specifically, that the Legendre polynomials for even $p$ can be
found from the $p=2$ case and that those for odd $p$ can be found
from the $p=3$ case. For the next theorem, we will let $P_{n,p}(t)$
denote the $n$-th degree Legendre polynomial in $p$ dimensions.

\begin{theorem}
For all $j = 0, 1, \ldots, n$,
  \[
    P_{n-j, 2j+p}(t) = { ((p-1)/2)_j \over (-n)_j(n+p-2)_j } \left(
    {d \over dt} \right)^j P_{n,p}(t).
  \]
\end{theorem}

\begin{proof}
  Differentiating \eqref{eq:diffeq} once with respect to $t$, we get
  \[
    (1-t^2)P_{n,p}'''+(-1-p)tP_{n,p}''+(n-1)(n+p-1)P_{n,p}' = 0,
  \]
  but this is just \eqref{eq:diffeq} again with the following substitutions:
  \begin{align*}
        P_{n,p} & \longmapsto P_{n,p}', \\
        p       & \longmapsto p+2,      \\
        n       & \longmapsto n-1.
  \end{align*}
  Thus, solving this new differential equation for $P_{n,p}'$ will give
  \[
    P_{n,p}'(t) \propto P_{n-1,p+2}(t).
  \]
  Continuing in this way, if we differentiate \eqref{eq:diffeq} $j$ times we see that
  \[
    \left( {d \over dt} \right)^j P_{n,p}(t) \propto P_{n-j,p+2j}(t)
  \]
  for any $0 \leq j \leq n$. Since $P_{n-j,p+2j}(1)=1$, in order to create an equality
  we must divide the left side of the above equation by its value at $t=1$, i.e.,
  \[
    P_{n-j,p+2j}(t) = { \left( {d \over dt} \right)^j P_{n,p}(t) \over \left( {d \over dt} \right)^j P_{n,p}(t) \big|_{t=1} }.
  \]
  Using the Rodrigues formula \eqref{eq:LegRod},
  \[
    \left( {d \over dt} \right)^j P_{n,p}(t) \big|_{t=1} = { (-n)_j(n+p-2)_j  \over ((p-1)/2)_j },
  \]
  as the reader can verify by computation. Therefore,
  \[
    P_{n-j, 2j+p}(t) = { ((p-1)/2)_j \over (-n)_j(n+p-2)_j } \left(
    {d \over dt} \right)^j P_{n,p}(t),
  \]
  completing the proof.
\qed
\end{proof}

The following lemma will be useful in the proof of the theorem that follows immediately afterwards.

\begin{lemma}
  Let $\xi$ and $\zeta$ be unit vectors, $f$ a function, and $F(\zeta, \xi)$
  given by
  \[
    F(\zeta, \xi) = \!\!\underset{\eta \in S^{p-1}}{\int} f(\langle
    \xi, \eta \rangle)P_n(\langle \eta , \zeta \rangle) \,
    d\Omega_{p-1}.
  \]
  Then,
  \[
    F(\zeta, \xi) = \Omega_{p-2} P_n(\langle \xi, \zeta \rangle)
    \int_{-1}^1 f(t)P_n(t)(1-t^2)^{(p-3)/2}\,dt.
  \]
  \label{lem:intred}
\end{lemma}

\begin{proof}
  First, notice that $F$ is invariant under any coordinate  rotation $R$, i.e., $F(R\zeta, R\xi) = F(\zeta,\xi)$. Indeed,
  the rotation
  \begin{align*}
    \xi      &\overset{R}{\longmapsto} \xi'  = R \xi , \\
    \zeta  &\overset{R}{\longmapsto} \zeta' = R \zeta,
  \end{align*}
  can be undone by a change of variables
  \[
    \eta \overset{R}{\longmapsto} \eta'=R\eta,
  \]
  in the integration.
  Thus, we are allowed to choose our coordinates such that
  \[
    \xi = (1,0,\ldots,0), \quad \zeta = (s,\sqrt{1-s^2},0,\ldots,0),
  \]
  so that $\langle \zeta, \xi \rangle = s$.

  Let us think of $F$ as a
  function of $\zeta$ alone and $\xi$ as a fixed parameter; we
  will write $F(\zeta; \xi)$. The argument of $F$, namely $\zeta$,
  only shows up inside the Legendre polynomial $P_n(\langle \eta,
  \zeta \rangle)$, which is a spherical harmonic --- in
  particular, a harmonic homogeneous polynomial of degree $n$ in the variables
  $\zeta_1, \zeta_2, \ldots, \zeta_p$. Therefore, the function $F$
  is also a harmonic homogeneous polynomial in the components of
  $\zeta$, i.e., in $s, \sqrt{1-s^2}$. But, since we equally well could have
  chosen coordinates such that $\zeta = (s, -\sqrt{1-s^2},0,\ldots,0)$, $F$ must really be a polynomial in
  $s$. 

  Then, being only a function of the inner
  product $s=\langle \zeta, \xi \rangle$, we see that $F$
  satisfies all the defining properties of the Legendre polynomials in
  Theorem \ref{thm:leg} except (i). We conclude
  \[
    F(\zeta, \xi) = cP_n(s), \quad \text{for some constant } c.
  \]
  We can determine the constant by considering the case $s=1$,
  i.e., $\zeta = \xi$, where $F(\xi, \xi) = cP_n(1) = c$. Using
  Lemma \ref{lem:red},
  \[
    c = \underset{\eta \in S^{p-1}}{\int} f(\langle \xi,
    \eta\rangle) P_n(\langle \eta, \xi \rangle)\,d\Omega_{p-1} =
    \Omega_{p-2}\int_{-1}^1 f(t)P_n(t)(1-t^2)^{(p-3)/2}\,dt,
  \]
  as sought.
\qed
\end{proof}

\begin{theorem}\emph{(Hecke-Funk Theorem)}
  Let $\xi$ be a unit vector, $f$ a function, and $Y_n$ a spherical harmonic. Then,
  \begin{equation}
    \underset{\eta \in S^{p-1}}{\int}
    f(\langle\xi,\eta\rangle)Y_n(\eta)\,d\Omega_{p-1} =
    \Omega_{p-2}Y_n(\xi)\int_{-1}^1 f(t)P_n(t)(1-t)^{(p-3)/2}\,dt.
    \label{eq:funk}
  \end{equation}
\end{theorem}

\begin{proof}
    By Theorem \ref{thm:SHexp}, there exist a set of coefficients $\{a_k\}_{k=1}^{N(p,n)}$ and a set of unit vectors
  $\{\zeta_k\}_{k=1}^{N(p,n)}$ such that
  \[
    Y_n(\eta) = \sum_{k=1}^{N(p,n)} a_k P_n(\langle \eta, \zeta_k
    \rangle).
  \]
  Then, we can rewrite the left-hand side of \eqref{eq:funk} as
  \begin{align*}
    \sum_{k=1}^{N(p,n)} & a_k \underset{\eta \in S^{p-1}}{\int} f(\langle
    \xi, \eta \rangle)P_n(\langle \eta, \zeta_k
    \rangle)\,d\Omega_{p-1},
  \end{align*}
  or,  using Lemma \ref{lem:intred},
  \begin{align*}
    \Omega_{p-2} \left(\sum_{k=1}^{N(p,n)} a_k
    P_n(\langle \xi, \zeta_k\rangle) \right)  \int_{-1}^1
    f(t)P_n(t)(1-t^2)^{(p-3)/2}\,dt .
  \end{align*}
That is
  \begin{align*}
    \Omega_{p-2}Y_n(\xi)\int_{-1}^1
    f(t)P_n(t)(1-t)^{(p-3)/2}\,dt,
  \end{align*}
which is exactly the right-hand side of \eqref{eq:funk}.
\qed
\end{proof}

We wish to find now an  integral representation of the
Legendre polynomials. Towards this goal, we first prove a lemma.

Consider a vector $\eta$ of the $(p-1)$-dimensional sphere $S^{p-1}$. Without loss of generality, we may
take it along the $x_1$-axis.  With $\{\eta\}^\bot$ we indicate the set of all vectors that are perpendicular to $\eta$; this is obviously a
hyperplane. The set $\{\eta\}^\bot \cap  S^{p-1} $ is then the equator of $S^{p-1}$ (a $(p-2)$-sphere) that is orthogonal to $\eta$; we will indicate
it by $S_\eta^{p-1}$.

\begin{lemma}
  Let $\eta = (1,0,\ldots,0)$ and $x \in \mathbb{R}^p$. Then the
  function
\begin{equation}
    L_n(x) = {1 \over \Omega_{p-2}} \underset{\zeta\in S^{p-1}_\eta}{\int}
    \left( \langle x, \eta \rangle + i \langle x, \zeta \rangle\right)^n \, d\Omega_{p-2},
\label{eq:legint}
\end{equation}
  when restricted to the sphere,
  is the $n$-th Legendre polynomial: $L_n(\xi) = P_n(\langle \xi, \eta \rangle)$.
\end{lemma}

\begin{proof}
  Clearly, $L_n(x)$ is a polynomial in the components of $x$.
  By the binomial theorem, this polynomial is homogeneous of
  degree $n$. We can also see that $L_n(x)$ is harmonic. Indeed,
  applying the Laplace operator on \eqref{eq:legint} and switching
  the order of differentiation and integration, the integrand
  becomes
   \[
    \sum_{j=1}^p {\partial^2 \over \partial x_j^2} \left(\langle x,\eta \rangle + i \langle x, \zeta \rangle \right)^n,
  \]
 or
  \[
    \sum_{j=1}^p {\partial \over \partial x_j} \left[ n \left(\langle x,
    \eta \rangle + i \langle x, \zeta \rangle \right)^{n-1}
    (\eta_j+i\zeta_j) \right],
  \]
  or
  \[
     n(n-1) \left(\langle x,
    \eta \rangle + i \langle x, \zeta \rangle \right)^{n-2}
    \sum_{j=1}^p (\eta_j+i\zeta_j)^2.
  \]
  But
  \[
    \sum_{j=1}^p(\eta_j+i\zeta_j)^2 = \langle \eta, \eta \rangle -
    \langle \zeta, \zeta \rangle + 2i\langle \eta, \zeta \rangle = 0 ,
  \]
  since $\{ \eta, \zeta \}$ is an orthonormal set. Also, notice
  \begin{align*}
    L_n(\eta) &= {1 \over \Omega_{p-2}} \underset{ \zeta\in  S^{p-1}_\eta}{\int} \left[ \langle \eta, \eta \rangle
                          + i \langle \eta, \zeta \rangle\right]^n \, d\Omega_{p-2}  \\
                     &= {1 \over \Omega_{p-2}} \underset{ \zeta\in  S^{p-1}_\eta}{\int}  d\Omega_{p-2}     = 1.
  \end{align*}
  Now, let $R$ be any coordinate rotation leaving $\eta$
  invariant, i.e., let $R$ be any rotation about the $x_1$-axis.
  The integrand of $L_n(Rx)$ is then
  \[
    \left[ \langle Rx, \eta \rangle + i \langle Rx, \zeta \rangle \right]^n =
    \left[ \langle x, \eta \rangle + i \langle x, R^t\zeta \rangle \right]^n,
  \]
  which can be reset to
  \[
    \left[ \langle x, \eta \rangle + i \langle x, \zeta \rangle \right]^n
  \]
  by a change of variables $\zeta\to\zeta'=R\zeta$
  and thus $L_n(x)$ is invariant under all such coordinate rotations.

  We have shown
  that $L_n(x)$ has all the properties described in Theorem
  \ref{thm:leg}, so that, when restricted to the sphere, it
  becomes the $n$-th Legendre polynomial. \qed
\end{proof}

Now we can prove the following integral representation for $P_n(t)$.

\begin{theorem}
  \[P_n(t) = {\Omega_{p-3} \over \Omega_{p-2}} \int_{-1}^1 \left(t +
  i s \sqrt{1-t^2} \right)^n (1-s^2)^{(p-4)/2}\,ds.\]
\end{theorem}

\begin{proof}
  From Lemma \ref{eq:legint}, we have
  \[
    P_n(\langle \xi, \eta \rangle) = {1 \over \Omega_{p-2}}
    \underset{ \zeta\in S^{p-1}_\eta}{\int} \left( \langle
    \xi, \eta \rangle + i \langle \xi, \zeta \rangle
    \right)^n\,d\Omega_{p-2}.
  \]
  Choose a constant $t$ and unit vector $\nu$ such that $\xi = t \eta + \sqrt{1-t^2} \: \nu$
  and $\langle \nu, \eta \rangle = 0$. Then, the above equation
  becomes, using Lemma \ref{lem:red} and replacing $p$ by $p-1$,
  \begin{align*}
    P_n(t) =& \: {1 \over \Omega_{p-2}} \underset{ \zeta\in S^{p-1}_\eta}{\int} \left(t + i\sqrt{1-t^2} \langle \nu, \zeta
    \rangle \right)^n \, d\Omega_{p-2}\\
    =&\: {\Omega_{p-3} \over \Omega_{p-2}} \int_{-1}^1
    \left( t+is\sqrt{1-t^2} \right)^n(1-s^2)^{(p-4)/2}\,ds,
  \end{align*}
  as sought.
\qed
\end{proof}

%====================================================================
\section{Boundary Value Problems}
%====================================================================

We conclude this discussion with an application of the ideas we
have developed to boundary value problems, where they display most
of their physical importance.

We know from Proposition \ref{fact:Hilb} that if a set of functions in a
Hilbert space is closed, then it is complete. In the following
theorem, we will see that a maximal linearly independent set of
spherical harmonics of all degrees is closed and thus complete.
This result allows us to develop expansions of functions as linear
combinations of spherical harmonics, which will be useful in
application to boundary value problems. In what follows, let
\[
    S=\{Y_{n,j}: n \in \mathbb{N}_0, \: 1 \leq j \leq N(p,n)\},
\]
be a maximal set of orthogonal spherical harmonics,
and let $\sum_{n,j}$ denote the sum $\sum_{n=0}^\infty
\sum_{j=1}^{N(p,n)}$.

\begin{theorem}
  Let the function $f: S^{p-1} \rightarrow \mathbb{R}$ be
  continuous. If $f$ is orthogonal to the set $S$, i.e., if
  \[
    \underset{\xi \in S^{p-1}}{\int} f(\xi)Y_{n,j}(\xi) \,
    d\Omega_{p-1} = 0, \text{ for all }  n,j
  \]
  then $f$ is the zero function, i.e.,
  \[
    f(\xi) = 0, \text{ for all } \xi \in S^{p-1}.
  \]
\end{theorem}

The requirement that $f$ be continuous is actually too strict, and
the theorem really applies to all square-integrable functions $f$,
i.e., all $f$ such that
\[
    \underset{S^{p-1}}{\int} f(\xi)^2 \, d\Omega_{p-1} < \infty.
\]
However, we will only prove the weaker version of this theorem.

\begin{proof}
  We will prove it by contradiction. Suppose that $f$ satisfies the hypotheses of the above theorem,
  i.e., is continuous and orthogonal to $S$, but is not the zero
  function. Then there exists some $\eta \in S^{p-1}$ such that
  $f(\eta) \neq 0$. We can assume that $f(\eta)>0$, for if
  $f(\eta)$ is negative we could consider $-f$ instead. By the
  continuity of $f$, there is some neighborhood around $\eta$ on
  the sphere where $f$ is positive. That is, there exists a
  constant $s$ such that $f(\xi) > 0$ whenever $s \leq
  \langle \xi, \eta \rangle \leq 1$. Define the function
  \[
    \psi(t) = \left\{
        \begin{array}{ll}
            1-{(1-t)^2 \over (1-s)^2} & \text{if } s \leq t \leq 1,\\
            0                         & \text{if } -1 \leq t \leq
            s.
        \end{array}
    \right.
  \]
  For $s \leq \langle \xi, \eta \rangle \leq 1$, the product
  $f(\xi)\psi(\langle \xi, \eta \rangle)$ is positive, and it vanishes for all other $\xi$. Thus,
  \begin{equation}
    \underset{\xi \in S^{p-1}}{\int}f(\xi)\psi(\langle \xi, \eta
    \rangle) \, d\Omega_{p-1} c > 0.
    \label{eq:nonvan}
  \end{equation}
  For simplicity, we will indicate the above integral by $c$.
  By the Weierstrass approximation theorem (Proposition \ref{fact:Weie}), we can
  find a polynomial $p(t)$ for any given $\epsilon > 0$ such that
  \[
    | \psi(t) - p(t) | \leq \epsilon, \text{ for all } t \in [-1,1].
  \]
  For any such $\epsilon$ and $p(t)$, 
  \begin{align*}
    \underset{\xi \in S^{p-1}}{\int}  \hspace{-3mm} f(\xi) \left[
    \psi(\langle \xi, \eta \rangle) - p(\langle \xi, \eta
    \rangle)\right]\,d\Omega_{p-1}  &\le
    \Bigg | \underset{\xi \in S^{p-1}}{\int}  f(\xi) \left[
    \psi(\langle \xi, \eta \rangle) - p(\langle \xi, \eta
    \rangle)\right]\,d\Omega_{p-1} \Bigg |\\
    & \leq \underset{\xi \in S^{p-1}}{\int}  \big| f(\xi)\big| \, \big|\psi(\langle \xi, \eta \rangle) - p(\langle \xi, \eta
    \rangle)\big|\,d\Omega_{p-1} \\
    & \leq \epsilon \,  \underset{\xi \in S^{p-1}}{\int}  \big| f(\xi)\big| \,  d\Omega_{p-1}.
  \end{align*}
  Since $f(\xi)$ is continuous and $\xi\in S^{p-1}$, there exists an $M$ such that  $M \geq |f(\xi)|$ for any $\xi \in S^{p-1}$.
  This implies
  \[
    \underset{\xi \in S^{p-1}}{\int} f(\xi) p(\langle \xi, \eta
    \rangle) \, d\Omega_{p-1} \geq c - M \epsilon \, \Omega_{p-1} .
  \]
  And since we can choose $\epsilon$ arbitrarily small,
  \begin{equation}
    \underset{\xi \in S^{p-1}}{\int} f(\xi) p(\langle \xi, \eta
    \rangle) \, d\Omega_{p-1} >0.
    \label{eq:contr}
  \end{equation}
  For the remainder of the proof, we fix $\epsilon$ and the
  corresponding $p(t)$ for which this expression is true.

  Now, let $m$ denote the degree of $p(t)$. We can write the
  polynomial $p(t)$ as a linear combination of the first $m$
  Legendre polynomials since each $P_n(t)$ is of degree $n$; i.e., we can find $c_k$ such that
  \[
    p(t) = \sum_{k=0}^m c_kP_k(t).
  \]
  We can thus rewrite the integral in \eqref{eq:contr} as
  \[
    \underset{\xi \in S^{p-1}}{\int} f(\xi) \sum_{k=0}^m c_kP_k(\langle \xi,\eta \rangle) \,
    d\Omega_{p-1}.
  \]
  But since the Legendre polynomials are just a special collection of
  spherical harmonics in $\xi$, this integral must vanish by
  hypothesis, contradicting the assertion in \eqref{eq:contr}.
  Therefore, our initial assumption that $f$ is not the zero
  function must be false. \qed
\end{proof}

So for any reasonable function $f$ defined on the sphere, we can
write
\begin{equation}
    f(\xi) = \sum_{n,j} c_{n,j}Y_{n,j}(\xi).
    \label{eq:fexp}
\end{equation}
To find $c_{n',j'}$, multiply both sides of this equation by
$Y_{n'j'}$ and integrate over the sphere. Using the orthonormality
of the set $S$, we find
\begin{equation}
    c_{n',j'} = \underset{S^{p-1}}{\int}
    f(\xi)Y_{n',j'}(\xi)\,d\Omega_{p-1}.
    \label{eq:lastcoeff}
\end{equation}
This expansion will be used in the following demonstration.

\begin{problem}
\label{problem:1}
  Consider the following boundary-value problem. Find $V$ in the closed unit ball $\bar{B}^p$, such
  that
  \begin{equation}
    \Delta_p V = 0, \quad \text{and} \quad V=f(\xi), \text{ for all } \xi \in S^{p-1}.
    \label{eq:bvp}
  \end{equation}
\end{problem}

\begin{solution}[Solution]
    We know that harmonic homogeneous polynomials are
    solutions to the Laplace equation, as well as any linear combination of
    them. We have also seen in \eqref{eq:sepvar} that we can write each of these polynomials as a power of the radius
    multiplied by a spherical harmonic. Thus, we can construct the
    solution  to \eqref{eq:bvp} as a linear combination of
    $r^nY_{n,j}(\xi)$ terms. To satisfy the boundary condition, we can use
    the coefficients in \eqref{eq:lastcoeff} to find
    \begin{equation}
      V = \sum_{n,j} r^nc_{n,j}Y_{n,j}(\xi) = \sum_{n,j} r^n Y_{n,j}(\xi) \underset{S^{p-1}}{\int}
    f(\eta)Y_{n,j}(\eta)\,d\Omega_{p-1},
      \label{eq:sol1}
    \end{equation}
thus solving the problem with the solution being in the form of a series.
\qed
\end{solution}

In the next subsection, we  will learn to solve this problem by another method.
 Since the solution of this boundary-value problem is unique (as we know from the theory of differential equations), we can
 equate the answers. And, amazingly,  this procedure will give us a generating function for the  Legendre polynomials $P_n(t)$.

%%%%%%%%%%%%%%%%%%%%%%%%%%%%%%%%%%%%%%%%%%%%%%%%%%%%%%%%%%
\subsection*{Green's Functions}

Given a differential equation
$$
        D_x(f(x)) = 0 ~,
$$
where $D_x$ is a differential operator acting on the unknown function $f(x)$, the corresponding  Green's function $G$ is defined
by the equation
$$
        D_x(G(x)) = \delta(x-x_0) ~,
$$
 where the function $\delta$ appearing in the right-hand side is  the Dirac delta function defined by
    \[
        \delta(x-x_0)=0, \text{ for all } x \neq x_0,
    \]
    and
    \[
     \underset{B^p_\epsilon(x_0)}{\int} \delta(x-x_0)\, d^px    = 1, \text{ for all } \epsilon>0.
    \]

 Let's assume that the given differential operator is the Laplacian in $p$ dimensions, that is, we seek the Green's function which satisfies
    \begin{equation}
        \Delta_p \tilde G = \delta(x-x_0) .
        \label{eq:green1}
    \end{equation}
    If we think of $\tilde G$ as   electric potential, then \eqref{eq:green1} describes the electric potential caused by a point charge at $x_0 \in
    \mathbb{R}^p$.

    Since the Laplacian is  invariant under translations, we see that $\tilde{G}$ can only depend on the
    distance from $x_0$, i.e., on $\rho = |x-x_0|$. When $\rho \neq 0$, $\tilde{G}$ must satisfy Laplace's
    equation:
    \begin{align*}
      0 &= \Delta_p \tilde{G}(\rho) = \sum_{i=1}^p{\partial^2\over\partial x_i^2}  \tilde{G}(\rho) \\
       &= \sum_{i=1}^p{\partial\over\partial x_i} \left( {\partial\tilde{G}(\rho)\over\partial\rho}\,{\partial\rho\over\partial x_i}   \right)\\
       &= \tilde{G}''(\rho) \sum_{i=1}^p \left( {\partial \rho \over \partial x_i} \right)^2 + \tilde{G}'(\rho) \sum_{i=1}^p
            {\partial^2 \rho \over \partial x_i^2}\\
      &= \tilde{G}''(\rho) + {p-1 \over \rho} \tilde{G}'(\rho) .
    \end{align*}
    We can easily solve this differential equation by separation of variables to find
    \[
          G(\rho) = a \rho + b , \quad \text{if}~p=1,
     \]
     and
    \[
        \tilde{G}'(\rho) =  {\text{const.} \over \rho^{p-1}} ,
    \]
    if $p\ge2$. This, in turn, implies
     \[
        \tilde{G}(\rho) =  a \, \ln\rho ,  \quad\text{if}~p=2,
    \]
    and
    \[
              \tilde{G}(\rho) = {a \over \rho^{\, p-2}}, \quad \text{if}~p\ge3.
    \]

    We will focus on the last case but the reader should explore the cases $p=1,2$ on his or her own to get a better understanding.
    To find the undetermined constant,  we integrate the defining equation over a ball of radius $\epsilon$ centered at the point $x_0$:
     \begin{align*}
         \underset{B^p_\epsilon(x_0)}{\int} \Delta_p \tilde{G} \,  d^px   =
         \underset{B^p_\epsilon(x_0)}{\int}  \delta(x-x_0) \,  d^px .
    \end{align*}
 By the properties of the Dirac delta function, the right-hand side is 1. The left-hand side can be rewritten by the use of
  the divergence theorem:
    \begin{align*}
         \underset{B^p_\epsilon(x_0)}{\int} \Delta_p \tilde{G} \, d^px
               & = \underset{S^{p-1}_\epsilon(x_0)}{\int} \nabla_p \tilde{G} \cdot \xi \, (\epsilon^{p-1}\, d\Omega_{p-1}) \\
               &= \epsilon^{p-1} \, \underset{S^{p-1}_\epsilon(x_0)}{\int} {d \tilde{G}  \over d\rho } \, d\Omega_{p-1}\\
              &=  \epsilon^{p-1}\, \underset{S^{p-1}_\epsilon(x_0)}{\int} {a\, (2-p) \over \epsilon^{p-1}} \, d\Omega_{p-1} \\
             & =  a \, (2-p) \, \Omega_{p-1}.
    \end{align*}
 Hence
 \[
      a =  {1\over (2-p)\,\Omega_{p-1} } .
 \]

     Of course, differential equations come with boundary conditions. So, let's modify the previous problem as follows.
     Let's  seek the Green's function which satisfies the same equation
    \begin{equation}
        \Delta_p G = \delta(x-x_0)  ,
        \label{eq:green2}
    \end{equation}
 for all $x\in B_p(0)$ and subjected to the boundary condition
  \begin{equation}
       G(\xi)=0, \text{ for all } \xi \in S^{p-1} .
  \label{eq:green3}
  \end{equation}
Equation \eqref{eq:green2} now describes the   electric potential caused by a point charge at $x_0$ and an ideal conducting sphere with center at
the origin.

To construct $G$,  we write
    \[
        G(x; x_0) = \tilde{G}(\rho) + g = {1 \over   (2-p) \Omega_{p-1} \rho^{\, p-2}} + g,
    \]
    and require that $g$ is harmonic in $B_p$
    \[
        \Delta_p \, g = 0
    \]
    and cancel $\tilde G$ on the boundary of $B_p$:
    \[
        g(\xi) = -\tilde{G}(\xi) = {1 \over  (p-2) \Omega_{p-1} \rho^{\, p-2}} \text{ for all } \xi \in  S^{p-1}.
    \]
    In fact, the functional expression of $g$ is
    identical to that of $\tilde G$.  However, there are two parameters that have to be fixed: the location of the singular point
    representing a point charge and the strength of the charge.
    The location of the singular point of $g$ cannot be inside $B_p(0)$.
    We will place the charge at the symmetric point $x_0'$ to $x_0$ with respect to the sphere, i.e., the
    point $x_0'$ that lies on the line passing through the origin
    and $x_0$ with $|x_0'| = 1/|x_0|$. We see then that
    \[
        g(\rho') \propto {1 \over   (2-p) \Omega_{p-1} \rho'^{\, p-2}},
    \]
    where $\rho' = |x-x_0'|$. We must choose the charge to be of the correct strength so
    that $G$ vanishes on the sphere. We easily see that the correct choice of $g$ is
    \[
        g(\rho') = - \,{1 \over
        (2-p) \Omega_{p-1} (|x_0| \rho')^{\, p-2}},
    \]
    so that
    \begin{equation}
        G(x; x_0) = {1 \over
        (2-p) \Omega_{p-1}} \left( {1 \over \rho^{\, p-2}} - {1 \over
        (|x_0| \rho')^{\, p-2}}\right).
        \label{eq:G}
    \end{equation}
    Indeed, \eqref{eq:G} satisfies Laplace's equation.
    Using
    the law of cosines and letting $\theta$ be the angle between
    the ray from the origin to $x$ and the ray from the origin to
    $x_0$,
    \begin{align*}
        \rho =& \sqrt{|x|^2 + |x_0|^2 -2|x|\,|x_0| \cos{\theta}},\\
        \rho' =& \sqrt{|x|^2 + |x_0'|^2 -2|x|\,|x_0'| \cos{\theta}}
        = \sqrt{|x|^2 + {1 \over |x_0|^2} -2{|x|\over |x_0|}
        \cos{\theta}},
    \end{align*}
    we see that \eqref{eq:G} vanishes on the unit sphere, i.e., when $|x|=1$. The method of constructing G as described above
    is known as  \emph{the method of images}. Physicists use it routinely without paying attention to the mathematical details!
    The reason is that if a solution is found for a boundary problem, it must be unique.

Let's now return to the problem stated on page \pageref{problem:1} and present an alternative solution using the results
on the Green's function for the Laplace equation.

\begin{solution}[Alternative Solution]
    Green's theorem (as shown in the footnote of page \pageref{GreensTheorem}) for the function $G$ and $V$,
    \[
        \underset{B^p}{\int} (V \Delta_p G - G \Delta_p V) \, d^px
        = \underset{S^{p-1}}{\int}\left(V {\partial G \over
        \partial \xi} - G {\partial V \over \partial \xi}
        \right)\,d\Omega_{p-1},
    \]
   reduces to
    \[
        \underset{B^p}{\int} V \delta (x-x_0) \, d^px
        = \underset{S^{p-1}}{\int}   V {\partial G \over \partial \xi} \,d\Omega_{p-1},
    \]
    since $V$ is harmonic and $G$ obeys \eqref{eq:green2} and \eqref{eq:green3}. Since
    the left-hand side of the above equation becomes
    \[
        \underset{B^p}{\int} V \delta (x-x_0) \, d^px =
        V(x_0) \underset{B^p}{\int}  \delta (x-x_0) = V(x_0),
    \]
    and since
    \[
        {\partial G \over \partial \xi} = {\partial G \over\partial |x|}\Big|_{|x| = 1}
         = {1-|x_0|^2 \over \Omega_{p-1} (1 + |x_0|^2 -2|x_0|\cos{\theta})^{p/2}},
    \]
    we can write the solution,
    \begin{equation}
      V(x_0) = {1 \over \Omega_{p-1}} \underset{ \xi \in
      S^{p-1}}{\int} f(\xi)\,{1-|x_0|^2 \over
        (1 + |x_0|^2 -2|x_0|\cos{\theta})^{p/2}}\,d\Omega_{p-1}.
      \label{eq:sol2}
    \end{equation}
Obviously, this gives the potential as an integral representation.
\qed
\end{solution}

Equating the two solutions, we can arrive at the following result.

\begin{theorem}
  In $\mathbb R^p$ we have
  \[
    \sum_{n=0}^\infty r^n N(p,n) P_n(t) = {1-r^2 \over (1  -2rt+r^2)^{p/2}}.
  \]
\end{theorem}

\begin{proof}
  Let $x_0 = |x_0| \eta$. Starting from equation \eqref{eq:sol1},
  \begin{align*}
    V(x_0) &= \sum_{n,j} |x_0|^n Y_{n,j}(\eta) \underset{\xi \in S^{p-1}}{\int}   f(\xi)Y_{n,j}(\xi)\,d\Omega_{p-1} \\
               &= \sum_{n=0}^\infty |x_0|^n \underset{\xi \in S^{p-1}}{\int}  f(\xi)) \sum_{j=1}^{N(p,n)} Y_{n,j}(\xi)Y_{n,j}(\eta)
                      \,d\Omega_{p-1},
  \end{align*}
  we rewrite it using Theorem \ref{thm:add},
  \begin{align*}
    V(x_0) &= \sum_{n=0}^\infty |x_0|^n \underset{\xi \in S^{p-1}}{\int}
                      f(\xi) {N(p,n) \over \Omega_{p-1}} P_n(\langle \xi, \eta \rangle) \, d\Omega_{p-1}\\
               &= {1 \over \Omega_{p-1}} \underset{\xi \in S^{p-1}}{\int} f(\xi) \sum_{n=0}^\infty
                    |x_0|^n N(p,n) P_n(\cos{\theta})\,d\Omega_{p-1}.
  \end{align*}
  Since the function $f$ is arbitrary, we can compare the above
  equation with \eqref{eq:sol2} and set $r = |x_0|$ and $t = \cos{\theta}$ to complete the proof. \qed
\end{proof}

This concludes our development of spherical harmonics in $p$ dimensions. We have briefly considered
an application to boundary value problems; we will not delve further into applications. We have achieved our main goal
to study the theory of spherical harmonics and the corresponding Legendre polynomials in $\mathbb{R}^p$. We urge the
reader to seek out applications on his or her own. Perhaps search for instances in physics where spherical harmonics are used in $\mathbb{R}^3$
and try to generalize to $p$ dimensions. One could start with the multipole expansion of an
electrostatic field (see \cite{ED:Jackson},\cite{ED:Landau}) or the wave function of an electron in a hydrogenic atom (see
\cite{QM:Shankar},\cite{QM:Landau}).

 \CleanPage
%
%\bibliographystyle{plain}
%\bibliography{end/sh2}

%
\end{document}